\documentclass[a4,11pt,twoside]{amsart}

\usepackage[T1]{fontenc}
\usepackage[latin1]{inputenc}%
\usepackage[english]{babel}%

\usepackage{amsmath,amsthm,amssymb,amsfonts,amscd}
\usepackage{url}
\usepackage{color}
\usepackage{setspace}
\usepackage{enumerate}
\usepackage[all]{xy}

\usepackage{epsfig}
\usepackage{graphicx}
\usepackage{mathrsfs}
\usepackage[coloneqq]{mathtools}
\usepackage{pinlabel}
\usepackage[outercaption]{sidecap}

\definecolor{light-gray}{gray}{0.60}

\newcounter{notes}%

\theoremstyle{plain}
\newtheorem{theorem}{Theorem}
\newtheorem{proposition}[theorem]{Proposition}
\newtheorem{corollary}[theorem]{Corollary}
\newtheorem{lemma}[theorem]{Lemma}
\newtheorem{fact}[theorem]{Fact}

\newtheorem*{BenoistKobayashiThm}{Benoist--Kobayashi's properness criterion}

\theoremstyle{definition}
\newtheorem{definition}[theorem]{Definition}
\newtheorem{example}[theorem]{Example}

\newtheorem{remark}[theorem]{Remark}

\newtheorem{question}[theorem]{Question}
\newtheorem{conjecture}[theorem]{Conjecture}

\numberwithin{theorem}{section}
\numberwithin{equation}{section}

\newcommand{\N}{\mathbb{N}}
\newcommand{\Z}{\mathbb{Z}}
\newcommand{\K}{\mathbb{K}}
\newcommand{\Q}{\mathbb{Q}}
\newcommand{\R}{\mathbb{R}}
\newcommand{\C}{\mathbb{C}}
\newcommand{\HH}{\mathbf{H}}
\newcommand{\PP}{\mathbf{P}}

\newcommand{\SL}{\mathrm{SL}}
\newcommand{\GL}{\mathrm{GL}}
\newcommand{\SO}{\mathrm{SO}}
\newcommand{\OO}{\mathrm{O}}
\newcommand{\PO}{\mathrm{PO}}

\newcommand{\PSL}{\mathrm{PSL}}

\newcommand{\Sp}{\mathrm{Sp}}
\newcommand{\U}{\mathrm{U}}
\newcommand{\SU}{\mathrm{SU}}

\newcommand{\g}{\mathfrak{g}}

\newcommand{\h}{\mathfrak{h}}
\newcommand{\kk}{\mathfrak{k}}
\newcommand{\p}{\mathfrak{p}}
\newcommand{\q}{\mathfrak{q}}
\newcommand{\aaa}{\mathfrak{a}}
\newcommand{\bb}{\mathfrak{b}}

\renewcommand{\1}{\mathbf{1}}

\newcommand{\Rrank}{\mathrm{rank}_{\mathbb{R}}}

\newcommand{\Ad}{\operatorname{Ad}}
\newcommand{\ad}{\operatorname{ad}}

\newcommand{\Hom}{\mathrm{Hom}}

\newcommand{\ie}{i.e.\ }
\newcommand{\eg}{e.g.\ }
\newcommand{\resp}{resp.\ }

\newcommand{\AdS}{\mathrm{AdS}}

\newcommand{\Diag}{\mathrm{Diag}}

\newcommand{\cohdim}{\mathrm{vcd}}
\newcommand{\Lie}{\mathrm{Lie}}
\newcommand{\vect}[1]{ \overset{\rightarrow}{#1}}
\newcommand{\equaldef}{\coloneqq}

\newcommand{\red}{\textit{red}}

\newcommand{\multidist}{\vect{\smash{d}\vphantom{x}}\vphantom{d}}

\title[Sharpness of proper and cocompact actions]{Sharpness of proper and cocompact actions on reductive homogeneous spaces}

\author{Fanny Kassel}
\address{CNRS and Laboratoire Alexander Grothendieck, Institut des Hautes \'Etudes Scientifiques, Universit\'e Paris-Saclay, 35 route de Chartres, 91440 Bures-sur-Yvette, France}
\email{kassel@ihes.fr}

\author{Nicolas Tholozan}
\address{DMA, ENS, 45 rue d'Ulm, 75005 Paris, France}
\email{Nicolas.Tholozan@ens.fr}

\thanks{This project received funding from the European Research Council (ERC) under the European Union's~Ho\-rizon 2020 research and innovation programme (ERC starting grant DiGGeS, grant agreement No. 715982).}

\begin{document}

\begin{abstract}
We prove that if $G$ is any noncompact connected real reductive linear Lie group and $\Gamma$ any discrete subgroup of~$G$ acting properly discontinuously and cocompactly on some homogeneous space $G/H$ of~$G$, then $\Gamma$ is quasi-isometrically embedded in~$G$ and the action of $\Gamma$ on $G/H$ is sharp, \ie satisfies a strong, quantitative form of proper discontinuity.
For noncompact reductive~$H$, this was known as the Sharpness Conjecture, with applications to spectral analysis on pseudo-Riemannian locally symmetric spaces developed in \cite{kk16}.
For $G/H$ rational of real corank one, we use sharpness to fully characterize properly discontinuous and cocompact actions on $G/H$ in terms of Anosov representations.
This enables us to show that in real corank one, acting properly discontinuously and cocompactly on $G/H$ is an open property, and also to prove that a number of homogeneous spaces do not admit compact quotients, such as $\SL(n+1,\mathbb{K})/\SL(n,\mathbb{K})$ for $n>1$ and $\mathbb{K}=\R$, $\C$, or the quaternions.
\end{abstract}

\maketitle
\tableofcontents

%%%%%%%%%%%%%%%%%%%%%%%%%%%%%%%%%%%%%%%%%%%%%%%%%%%
\section{Introduction}

%%%%%%%%%%%%%%%%%%%%%%%%%
\subsection{Compact quotients of reductive homogeneous spaces} \label{subsec:intro-cpt-quot}

Let $X$ be a manifold equipped with a faithful and transitive action of a connected real linear Lie group $G$.
Then $X$ identifies with the right quotient $G/H$, where $H$ is the stabilizer of a basepoint in $X$.
A \emph{compact quotient} of~$X$ (sometimes also called a \emph{compact Clifford--Klein form} of $X$) is a closed orbifold of the form $\Gamma \backslash X$, where $\Gamma$ is a discrete subgroup of $G$ acting properly discontinuously and cocompactly on~$X$.
Up to passing to a finite-index subgroup, one can furthermore assume (by the Selberg lemma \cite[Lem.\,8]{sel60}) that the action of $\Gamma$ is free, so that the quotient $\Gamma \backslash X$ is a closed manifold.

Compact quotients of homogeneous spaces are thus (up to a finite orbifold cover) manifolds whose universal covering has a ``geometric incarnation''; they were both a motivation and a source of examples for Poincar\'e's introduction of the notion of fundamental group.

Research on compact quotients of homogeneous spaces~$X$ is driven by the following three main questions:

\begin{question} \label{ques:Existence}
Does $X$ admit compact quotients?
\end{question}

\begin{question} \label{ques:Rigidity}
If $X$ admits compact quotients, are these quotients rigid?
\end{question}

\begin{question} \label{ques:Topology}
If $X$ admits compact quotients, what is their topology/geometry?
\end{question}

These questions have been thoroughly investigated when the isotropy group $H$ is compact.
In this case, $G$ preserves a Riemannian metric on $X$, any discrete subgroup $\Gamma$ of $G$ automatically acts properly discontinuously on $X$, and this action is cocompact if and only if $\Gamma$ is a \emph{uniform lattice} in~$G$.
When $G$ is simple, Question~\ref{ques:Existence} is thus answered positively by Borel's existence theorem for uniform lattices \cite{bor63}, while Calabi--Weil's local rigidity theorem \cite{wei62} states that any uniform lattice in $G$ is rigid except when $G$ is isogenous to $\PSL(2,\R)$, answering Question~\ref{ques:Rigidity}; if $K$ is a maximal compact subgroup of~$G$, then $\Gamma \backslash X$ fibers over $\Gamma \backslash G/K$, which is a closed nonpositively-curved Riemannian (hence aspherical) orbifold.
This already gives very strong information on the inexhaustible Question~\ref{ques:Topology}.

In contrast, the case where $H$ is noncompact is far from understood.
In this case, not all discrete subgroups of~$G$ act properly discontinuously on $G/H$ (\eg infinite discrete subgroups of~$H$ do not, as they have a global fixed point in $G/H$); determining when the action of a discrete subgroup $\Gamma$ can be simultaneously properly discontinuous (which requires $\Gamma$ to be ``not too large'') and cocompact (which requires $\Gamma$ to be ``not too small'') is a difficult problem, which has given rise to a very rich literature (see \eg the surveys \cite{ky05,con14} or the introduction of \cite{kas-PhD}).
Various obstructions have been given to the existence of compact quotients of certain homogeneous spaces $G/H$, but Question~\ref{ques:Existence} remains open in general.
For instance, it is known from \cite{cm62,wol62,kob89} that if $G$ and~$H$ are reductive with $\Rrank(G) = \Rrank(H)$, then only finite groups can act properly discontinuously on $G/H$ (\emph{Calabi--Markus phenomenon}).
In some cases where $\Rrank(G) > \Rrank(H)$, only virtually cyclic or virtually abelian groups can act properly discontinuously on $G/H$ (see \cite{ben96,kas08}), and these are too small to give rise to compact quotients.
In other cases certain nonabelian free groups can act properly discontinuously, but still $G/H$ does not admit any compact quotients.
The known constructions of compact quotients are often rigid, but several infinite families of homogeneous spaces admitting nonrigid compact quotients have also been found \cite{gol85,kob98,kas12,mst}.
Finally, very little is known about the topology of compact quotients.
For instance, it is not known whether the fundamental group of a compact quotient of the affine space must be virtually solvable (this is the famous Auslander conjecture, see \eg \cite{ddgs22}).

In this paper, we will focus on the case that $G$ is reductive (and $H$ noncompact).
We will prove a general ``sharpness'' property of compact quotients of $G/H$ (Theorem~\ref{thm:sharp}), which has strong applications to Questions \ref{ques:Existence}, \ref{ques:Rigidity}, and~\ref{ques:Topology} above.
Let us start by explaining what sharpness is.

%%%%%%%%%%%%%%%%%%%%%%%%%
\subsection{Sharpness} \label{subsec:intro-sharp}

A general and systematic study of compact quotients of homogeneous spaces $G/H$ for reductive~$G$ and noncompact~$H$ was initiated by Kobayashi in the late 1980s.
He and Benoist independently formulated a criterion for a subgroup $\Gamma$ of $G$ to act properly discontinuously on $G/H$.
Let $\aaa$ be a Cartan subspace of the Lie algebra $\g$ of~$G$, let $\aaa^+ \subset \aaa$ be a choice of closed positive Weyl chamber, and let $\mu: G \to \aaa^+$ be the corresponding Cartan projection (see Section~\ref{subsec:Cartan-decomp}).
Let $d_{\aaa}$ be the distance function on~$\aaa$ associated to some Euclidean norm $\Vert\cdot\Vert$.

\begin{BenoistKobayashiThm}[\cite{ben96,kob96}]
Let $G$ be a connected real linear reductive Lie group, with Cartan projection $\mu : G\to\aaa^+$, and let $H$ and $\Gamma$ be two closed subgroups of~$G$.
Then $\Gamma$ acts properly discontinuously on $G/H$ if and only if for every $R>0$, the set
$$\{\gamma \in \Gamma ~|~ d_{\aaa}(\mu(\gamma),\mu(H))\leq R\}$$
is finite.
\end{BenoistKobayashiThm}

In her thesis \cite{kas-PhD}, the first-named author proved a stronger property for properly~discon\-tinuous actions on the $3$-dimensional anti-de Sitter space $\mathrm{AdS}^3 = G/H = \PO(2,2)/\OO(2,1)\linebreak \simeq (\PSL(2,\R)\times\PSL(2,\R))/\Diag(\PSL(2,\R))$: namely, in this case, if a discrete subgroup $\Gamma$ of~$G$ acts properly discontinuously and cocompactly on $G/H$, then the image of~$\Gamma$ under the Cartan projection~$\mu$ escapes \emph{linearly} from the set $\mu(H)$.
This led to the introduction of the notion of \emph{sharpness}, which should be thought of as a strong, quantitative version of proper discontinuity.

\begin{definition}[\cite{kk16}] \label{def:sharp}
Let $G$ be a connected real linear reductive Lie group, $H$ a closed subgroup of~$G$, and $\Gamma$ a discrete subgroup of~$G$.
The action of $\Gamma$ on $G/H$ is \emph{sharp} if there exist $c,c'>0$ such that
$$d_{\aaa}(\mu(\gamma),\mu(H)) \geq c\,\Vert\mu(\gamma)\Vert - c' \quad\quad\text{for all }\gamma\in\Gamma.$$ 
\end{definition}

Equivalently, the action of $\Gamma$ on $G/H$ is sharp if the \emph{limit cones} of $\Gamma$ and $H$ (see Section~\ref{subsec:lim-cone}) intersect only at~$0$.

If the action of $\Gamma$ on $G/H$ is sharp, then it is properly discontinuous by the Benoist--Kobayashi properness criterion.
The converse is false in general (see \cite[\S\,10.1]{gk17}).
However, the Sharpness Conjecture \cite[Conj.\,4.12]{kk16} states that if the action of $\Gamma$ on $G/H$ is properly discontinuous \emph{and cocompact}, then it should be sharp.
In the present paper, we prove this conjecture.

\begin{theorem} \label{thm:sharp}
Let $G$ be a connected real linear reductive Lie group and $H$ a closed subgroup of~$G$.
Then the Sharpness Conjecture holds for $G/H$.
More precisely, any discrete subgroup $\Gamma$ of~$G$ acting properly discontinuously and cocompactly on $G/H$ is finitely generated and \emph{sharply embedded} in~$G$ with respect to~$H$.
\end{theorem}

Here we use the following terminology.

\begin{definition} \label{def:sharp-embed}
Let $G$ be a connected real linear reductive Lie group, with Cartan projection $\mu : G\to\aaa^+$, and let $H$ be a closed subgroup of~$G$.
A compactly generated closed subgroup $\Gamma$ of $G$, with compact generating subset~$S$ and associated word length function $|\cdot|_S : \Gamma\to\N$, is \emph{sharply embedded in~$G$ with respect to~$H$} if there exist $c,c'>0$ such that $d_{\aaa}(\mu(\gamma),\mu(H)) \geq c\,|\gamma|_S - c'$ for all $\gamma\in\Gamma$.
\end{definition}

Note that changing the compact generating set~$S$ may change the constants $c,c'$, but not their existence, so the notion is well defined independently of~$S$ (see Section~\ref{subsec:compact-gen}).
This notion is equivalent to the action of $\Gamma$ on $G/H$ being sharp and $\Gamma$ being quasi-isometrically embedded in~$G$, see Corollary~\ref{cor:sharp-embed-explain}.

Theorem~\ref{thm:sharp} was previously proved for $G/H = (G_0\times G_0)/\mathrm{Diag}(G_0)$ with $G_0=\PSL(2,\R)$ in \cite{kas-PhD} (this is the $3$-dimensional anti-de Sitter space), with $G_0=\SO(n,1)$ in \cite{gk17}, and with $G_0$ a general semisimple linear group of real rank one in \cite{ggkw17}.
Note that \cite{kas-PhD,gk17,ggkw17} actually have weaker assumptions than $\Gamma$ being cocompact.
In particular, for $G_0= \PSL(2,\R)$, properly discontinuous actions are sharp as soon as $\Gamma$ is finitely generated.
On the other hand, there are some infinitely generated groups acting properly discontinuously but not sharply on the $3$-dimensional anti-de Sitter space \cite[\S\,10.1]{gk17}.

%%%%%%%%%%%%%%%%%%%%%%%%%
\subsection{Application 1: spectral analysis on pseudo-Riemannian locally symmetric spaces}

Throughout the paper, we denote by $K_H$ a maximal compact subgroup of $H$ and by $K$ a maximal compact subgroup of~$G$ containing~$K_H$.

Suppose that the homogeneous space $G/H$ is of reductive type, \ie the closed subgroup $H$ has finitely many connected components and is stable under a Cartan involution of the reductive Lie group~$G$. Then the space $G/H$ carries a $G$-invariant pseudo-Riemannian metric of signature $(\mathtt{d}_+,\mathtt{d}_-)$, where 
\begin{equation} \label{eqn:d+d-}
\mathtt{d}_+ = \dim (G/K)-\dim(H/K_H) \quad \textrm{and} \quad \mathtt{d}_-= \dim(K_G/K_H).
\end{equation}
(For semisimple~$G$, such a metric is induced by the Killing form on the Lie algebra of~$G$.)
This pseudo-Riemannian metric descends to any quotient manifold $\Gamma \backslash G/H$ and defines a Laplacian on the quotient.
We note that the principal symbol of this differential operator is the pseudo-Riemannian metric, hence it is not elliptic when $H$ is noncompact.

In \cite{kk11,kk16}, Kobayashi and the first-named author initiated the study of the spectral theory of the Laplacian on $\Gamma\backslash G/H$.
In particular, assuming $G/H$ is symmetric, they proved \cite[Th.\,1.8]{kk16} that if the discrete spectrum of the Laplacian on $G/H$ is infinite and if the action of $\Gamma$ on $G/H$ is sharp, then the discrete spectrum of the Laplacian on $\Gamma\backslash G/H$ is infinite; this was done by constructing generalized Poincar\'e series.
The fact that the discrete spectrum of the Laplacian on $G/H$ is infinite is known \cite{fle80,mo84} to be equivalent to the rank condition $\mathrm{rank}(G/H) = \mathrm{rank}(K/K_H)$; see \cite[\S\,2]{kk16} for examples.
By Theorem~\ref{thm:sharp}, the action of $\Gamma$ on $G/H$ is sharp whenever the quotient $\Gamma\backslash G/H$ is compact.

%%%%%%%%%%%%%%%%%%%%%%%%%
\subsection{Application 2: openness for proper and cocompact actions}

It is expected that properly discontinuous and cocompact actions on homogeneous spaces $G/H$ of reductive type should always be stable under small deformations.

\begin{conjecture} \label{conj:openness}
Let $G/H$ be a homogeneous space of reductive type and $\Gamma$ a discrete subgroup of~$G$ acting properly discontinuously and cocompactly on $G/H$.
Then there is a neighborhood $\mathcal{U} \subset \Hom(\Gamma,G)$ of the natural inclusion such that every $\rho \in \mathcal{U}$ is discrete and faithful and satisfies that the action of $\Gamma$ on $G/H$ via~$\rho$ is properly discontinuous and cocompact.
\end{conjecture}

\begin{remark}
By the Ehresmann--Thurston principle (see \eg \cite{bg04}), small deformations of the inclusion $\iota: \Gamma \to G$ are holonomies of $(G,G/H)$-structures on a closed manifold $\Gamma \backslash G/H$. Conjecture~\ref{conj:openness} would thus hold if we knew that closed manifolds locally modeled on $G/H$ are always covered by $G/H$, which is a reductive version of the \emph{Markus completeness conjecture} for affine manifolds (see \cite{car89,kli96}).
\end{remark}

Conjecture~\ref{conj:openness} has been proved when $\Gamma$ is a lattice in a connected Lie subgroup $L$ of~$G$ \cite{kas12} (the so-called \emph{standard quotients}), and also when $G/H = (G_0\times G_0)/ \Diag(G_0)$, with $G_0$ a semisimple Lie group of real rank~$1$ \cite{kas-PhD,gk17,ggkw17}.
Here we use Theorem~\ref{thm:sharp} to prove the following, where $W \equaldef N_G(\aaa)/Z_G(\aaa)$ is the restricted Weyl group of $G$.

\begin{theorem} \label{thm:openness}
Conjecture~\ref{conj:openness} holds whenever $\Rrank(G)-\Rrank(H) = 1$.

More generally, let $G$ be a connected real linear reductive Lie group and $H$ any closed subgroup of~$G$ such that $\mu(H)$ is the intersection of~$\aaa^+$ with the $W$-orbit of a hyperplane of~$\aaa$.
Then for any discrete subgroup  $\Gamma$ of~$G$ acting properly discontinuously and cocompactly on $G/H$, there is a neighborhood $\mathcal{U} \subset \Hom(\Gamma,G)$ of the natural inclusion such that every $\rho \in \mathcal{U}$ is discrete and faithful and satisfies that the action of $\Gamma$ on $G/H$ via~$\rho$ is properly discontinuous and cocompact.
\end{theorem}

Theorem~\ref{thm:openness} applies in particular to some homogeneous spaces $G/H$ where $H$ is not reductive, and not even cocompact in a reductive subgroup of~$G$, such as the exotic homogeneous spaces $G/H$ with $G=\SO(2,2n)$ constructed by Oh--Witte \cite{ow00}.

\begin{remark} \label{rem:G-connected}
Throughout the paper we assume that $G$ is connected for convenience.
However, Theorem~\ref{thm:openness} holds for any real linear reductive Lie group~$G$ with finitely many connected components, and so do Theorem~\ref{thm:proper-compact<->Ano}, Proposition~\ref{prop:cork-1-proper-compact->Ano}, and Theorem~\ref{thm:cpt-quotient->H-QI}.
Indeed, if $G$ is a real linear reductive Lie group and $H,\Gamma$ are closed subgroups of~$G$ with $\Gamma$ discrete, then, denoting by $G_0$ the identity component of~$G$, the group $\Gamma \cap G_0$ has finite index in~$\Gamma$, and $\Gamma$ acts properly discontinuously and cocompactly on $G/H$ if and only if $\Gamma \cap G_0$ acts properly discontinuously and cocompactly on $G_0/(H \cap G_0)$.
\end{remark}

In order to establish Theorem~\ref{thm:openness}, we develop further the connection between sharpness and \emph{Anosov representations} which was initiated in \cite{ggkw17}.

%%%%%%%%%%%%%%%%%%%%%%%%%
\subsection{Sharpness and Anosov representations} \label{subsec:intro-Ano}

Anosov representations are representations of word hyperbolic groups into noncompact real semisimple (or reductive) Lie groups which have finite kernel, discrete image, and certain strong dynamical properties.
They were first introduced by Labourie \cite{lab06}, and further studied by Guichard--Wienhard \cite{gw12} and many other authors.
They now play an important role in higher Teichm\"uller theory and in recent developments on discrete subgroups of Lie groups (see \eg \cite{kas-notes}).

Anosov representations with values in~$G$ are defined with respect to the choice of a conjugacy class of proper parabolic subgroups of~$G$, which is equivalent to the choice of a nonempty subset of the simple restricted roots defining the Weyl chamber~$\aaa^+$.

%%%%%%%%
\subsubsection{The case where $\mu(H)$ is a union of walls}

Here is an application of Theorem~\ref{thm:sharp} in the case that $\mu(H)$ is a union of walls of~$\aaa^+$.

\begin{theorem} \label{thm:proper-compact<->Ano}
Let $G$ be a connected real linear reductive Lie group with Cartan projection $\mu : G\to\aaa^+$.
Let $\theta\subset\Delta$ be a subset of the simple restricted roots of~$G$, and $H$ a closed subgroup of~$G$ such that $\mu(H) = \aaa^+ \cap \bigcup_{\alpha\in\theta} \mathrm{Ker}(\alpha)$.
Then for any discrete subgroup $\Gamma$ of~$G$, the following are equivalent:
\begin{enumerate}
  \item\label{item:proper-compact} $\Gamma$ acts properly discontinuously and cocompactly on $G/H$,
  \item\label{item:Ano-dim} $\Gamma$ is word hyperbolic, its Gromov boundary has covering dimension $\dim(G/K) - \dim(H/K_H)-\nolinebreak 1$, and the natural inclusion $\Gamma \hookrightarrow G$ is a $P_{\theta}$-Anosov representation.
\end{enumerate}
\end{theorem}

Here we denote by $P_{\theta}$ a parabolic subgroup of type~$\theta$ (with the convention that $P_{\Delta}$ is a minimal parabolic subgroup of~$G$, see \eqref{eqn:P-theta}).

The implication \eqref{item:proper-compact}~$\Rightarrow$~\eqref{item:Ano-dim} follows from Theorem~\ref{thm:sharp}, from the main result of \cite{klp-morse} (see Fact~\ref{fact:klp}), and from Facts \ref{fact:vcd} and~\ref{fact:dim-bound-vcd} below.
The implication \eqref{item:Ano-dim}~$\Rightarrow$~\eqref{item:proper-compact} follows from the Benoist--Kobayashi properness criterion and from Facts \ref{fact:vcd} and~\ref{fact:dim-bound-vcd}; it was first observed in \cite{ggkw17}.

One important property of Anosov representations \cite{lab06} is that they form an open subset of $\Hom(\Gamma,G)$.
Theorem~\ref{thm:proper-compact<->Ano} and this property immediately imply Conjecture~\ref{conj:openness} for homogeneous spaces $G/H$ of reductive type such that $\mu(H)$ is a union of walls of~$\aaa^+$.

\begin{table}[!ht]
\centering
\begin{tabular}{|c|c|c|c|c|c|}
\hline
& $G$ & $H$ & restricted root system of $G$ & $\theta$\tabularnewline
\hline
(i) & $\SO(p,q)_0$ & $\SO(p,q-1)_0$ & $B_q$ & $\{\alpha_q\}$\tabularnewline
(ii) & $\SU(p,q)$ & $\U(p,q-1)$ & $(BC)_q$ (if $p>q$) or $C_q$ & $\{\alpha_q\}$\tabularnewline
(iii) & $\Sp(p,q)$ & $\Sp(p,q-1)\times\Sp(1)$ & $(BC)_q$ (if $p>q$) or $C_q$ & $\{\alpha_q\}$\tabularnewline
(iv) & $\SO(2p+k,2)_0$ & $\U(p,1)$ & $B_2$ & $\{\alpha_1\}$\tabularnewline
(v) & $\SU(2p+k,2)$ & $\Sp(p,1)$ & $(BC)_2$ (if $2p+k>2$) or $C_2$ & $\{\alpha_1\}$\tabularnewline
%(vi) & $\Sp(2m,\R)$, $m\geq 1$ & $\U(m-1,1)$ & $C_m$ & $\{\alpha_1\}$\tabularnewline
\hline
\end{tabular}
\vspace{0.2cm}
\caption{Some triples $(G,H,\theta)$ to which Theorem~\ref{thm:proper-compact<->Ano} applies. In example~(i) we assume $p>q>1$, in examples (ii)--(iii) we assume $p\geq q>1$, and in examples (iv)--(v) we assume $p\geq 1$ and $k\geq 0$, with $2p+k>2$ in example~(iv).}
\label{table1}
\end{table}

Theorem~\ref{thm:proper-compact<->Ano} applies to the examples in Table~\ref{table1}.
In these examples, the set $\mu(H)$ is a single wall $\mathrm{Ker}(\alpha)$ of~$\aaa^+$, for some simple restricted root $\alpha$ of~$G$, and $\theta = \{\alpha\}$.
Discrete subgroups $\Gamma$ of~$G$ acting properly discontinuously and cocompactly on $G/H$ are known to exist in case~(i) for $(p,q)\in (2\N^*,2)$ or $(4\N^*,4)$ or $(8,8)$, in case~(ii) for $(p,q) \in (2\N^*,2)$, and in cases (iv) and~(v) for $p\in\N^*$ and $k=0$.

We refer to Proposition~\ref{prop:QI-sharp<->Ano} below for a version of Theorem~\ref{thm:proper-compact<->Ano} which holds for discrete subgroups $\Gamma$ of~$G$ whose action on $G/H$ is sharp but not necessarily cocompact.
In that broader setting, all values of $p,q,k$ may arise in examples (i), (ii), (iii), (iv), (v) of Table~\ref{table1}.

%%%%%%%%
\subsubsection{Geometric aspects in one example}

A particularly interesting example is $G/H = \SO(2p,2)/\U(p,1)$.
This is a pseudo-Riemannian symmetric space of signature $(\mathtt{d}_+,\mathtt{d}_-) = (2p,p(p-1))$, which also carries a $G$-invariant indefinite K\"ahler structure (see \cite[\S\,10.3]{kk16}).
Here $G/\{\pm\mathrm{I}\} = \PO(2p,2)$ is the isometry group of the $(2p+1)$-dimensional \emph{anti-de Sitter space} $\AdS^{2p+1} = \HH^{2p,1} = \PO(2p,2)/\OO(2p,1)$, which is a model space for $(2p+1)$-dimensional Lorentzian manifolds of constant negative sectional curvature.
We denote by $P_1 = P_{\{\alpha_1\}}$ the stabilizer in $G=\SO(2p,2)$ of an isotropic line of $\R^{2p,2}$.

\begin{theorem} \label{thm:SO/U}
Let $\Gamma$ be a discrete subgroup of $\SO(2p,2)$, where $p\geq 1$.
Then the following are equivalent:
\begin{enumerate}
  \item[(1)]\label{item:SO/U-1} $\Gamma$ acts properly discontinuously and cocompactly on $\SO(2p,2)/\U(p,1)$,
  \item[(2)]\label{item:SO/U-2} $\Gamma$ is word hyperbolic, its Gromov boundary has covering dimension $2p-1$, and the natural inclusion $\Gamma \hookrightarrow \SO(2p,2)$ is $P_1$-Anosov,
  \item[(2')]\label{item:SO/U-2'} $\Gamma$ is word hyperbolic, its Gromov boundary is a $(2p-1)$-dimensional sphere, and the natural inclusion $\Gamma \hookrightarrow \SO(2p,2)$ is $P_1$-Anosov,
  \item[(3)]\label{item:SO/U-3} $\Gamma$ preserves and acts cocompactly on a convex spacelike hypersurface in $\AdS^{2p+1} =\nolinebreak \HH^{2p,1}$,
  \item[(4)]\label{item:SO/U-4} $\Gamma$ is $\HH^{2p,1}$-convex cocompact in the sense of \cite{dgk-Hpq-cc,dgk-proj-cc} and has virtual cohomological dimension~$2p$.
\end{enumerate}
\end{theorem}

(See Section~\ref{subsec:proof-vcd} for the notion of virtual cohomological dimension.)

Condition~(3) states that $\Gamma$ is the holonomy of a GHMC (globally hyperbolic maximal spatially compact) $(2p+1)$-dimensional Lorentzian manifold.
The equivalence (2')~$\Leftrightarrow$~(3) was established in \cite{bm12}.
The equivalences (2)~$\Leftrightarrow$~(2')~$\Leftrightarrow$~(4) were established in \cite{dgk-Hpq-cc,dgk-proj-cc}.
The implication (2)~$\Rightarrow$~(1) was first observed in \cite{ggkw17}.
Theorem~\ref{thm:proper-compact<->Ano} applied to $\SO(2p,2)/\U(p,1)$ gives the converse implication (1)~$\Rightarrow$~(2) (see Remark~\ref{rem:G-connected}).

These groups have a rich deformation theory; Barbot \cite{bar15} proved that the space of $P_1$-Anosov representations $\rho : \Gamma\to\SO(2p,2)$ is not only open in $\mathrm{Hom}(\Gamma,\SO(2p,2))$, but also closed.
Exotic examples (which are not deformations of $\SO(2p,1)$-lattices) were constructed by Lee--Marquis \cite{lm19} and by Monclair, Schlenker and the second-named author \cite{mst}.
It is proved in \cite{mst} that if condition~(3) of Theorem~\ref{thm:SO/U} holds, then there is a natural $\Gamma$-equivariant fibration of $\SO(2p,2)/\U(p,1)$ over a $\Gamma$-invariant spacelike hypersurface in $\AdS^{2p+1}$.
Combining this result with our Theorem~\ref{thm:proper-compact<->Ano} (which gives the implication (1)~$\Rightarrow$~(2) of Theorem~\ref{thm:SO/U}), we obtain the following.

\begin{corollary}
Let $\Gamma$ be a discrete subgroup of $G= \SO(2p,2)$ acting freely, properly discontinuously and cocompactly on $G/H=\SO(2p,2)/\U(p,1)$. Then $\Gamma \backslash G/H$ admits a smooth fibration over a closed aspherical manifold of dimension $\mathtt{d}_+ = 2p$ whose fibers are left translates of $K/K_H$.
\end{corollary}

%%%%%%%%
\subsubsection{The general corank-one case}

In the corank-one case, even if $\mu(H)$ is not a union of walls of~$\aaa^+$, we can still prove that properly discontinuous and cocompact actions on $G/H$ yield Anosov representations in~$G$, using Theorem~\ref{thm:sharp}, \cite{kas08}, and \cite{klp-morse}.

\begin{proposition} \label{prop:cork-1-proper-compact->Ano}
Let $G$ be a connected real linear reductive Lie group, with Cartan projection $\mu : G\to\aaa^+$, and let $H$ be a closed subgroup of~$G$ such that $\mu(H)$ is the intersection of~$\aaa^+$ with the $W$-orbit of a hyperplane in~$\aaa$ (for instance, $G/H$ could be any homogeneous space of reductive type with $\Rrank(G) - \Rrank(H) = 1$).
Let $\Gamma$ be a discrete subgroup of~$G$ acting properly discontinuously and cocompactly on $G/H$.
Then $\Gamma$ is word hyperbolic and its Gromov boundary has covering dimension $\dim(G/K) - \dim(H/K_H) - 1$.
Moreover, the natural inclusion $\Gamma \hookrightarrow G$ is a $P$-Anosov representation, for some proper parabolic subgroup $P$ of~$G$, unless possibly if $\h \supset \aaa \cap [\g,\g]$ (in which case $\Gamma$ is virtually cyclic).
\end{proposition}

Recall that a group $\Gamma$ is said to \emph{virtually} satisfy a property if some finite-index subgroup of~$\Gamma$ satisfies it.
The degenerate case $\h \supset \aaa \cap [\g,\g]$ (in which $\Gamma$ is virtually cyclic) cannot happen when $G$ is semisimple.

Proposition~\ref{prop:cork-1-proper-compact->Ano} applies in particular to the cases in Table~\ref{table2}, where $G/H$ is a pseudo-Riemannian symmetric space with $\Rrank(G) - \Rrank(H) = 1$.
(This table is obtained from the classification of pseudo-Riemannian symmetric spaces of split rank one given \eg in \cite[Table\,II]{os84}.) 
Proposition~\ref{prop:cork-1-proper-compact->Ano} was previously proved in \cite{kas-PhD,gk17,ggkw17} for $G/H = (G_0\times G_0)/\mathrm{Diag}(G_0)$ with $\Rrank(G_0) =\nolinebreak 1$ (case~III$_i^d$ of Table~\ref{table2}).

\begin{table}[!ht]
\centering
\begin{tabular}{|c|c|c|c|c|c|}
\hline
& $G$ & $H$ & restr.\ root system of $G$ & $\theta$\tabularnewline
\hline
I$_1$ & $\SO(p,q)$ & $\OO(p,q-1)$ & $B_q$ & $\{\alpha_q\}$\tabularnewline
I$'_1$ & $\SO(p,p)$ & $\OO(p,p-1)$ & $D_p$ & $\{\alpha_{p-1}\}$ or $\{\alpha_q\}$\tabularnewline
I$_2$ & $\SU(p,q)$ & $\U(p,q-1)$ & $(BC)_q$ (if $p>q$) or $C_q$ & $\{\alpha_q\}$\tabularnewline
I$_3$ & $\Sp(p,q)$ & $\Sp(p,q-1)\times\Sp(1)$ & $(BC)_q$ (if $p>q$) or $C_q$ & $\{\alpha_q\}$\tabularnewline
III$_1$ & $\SO(2p,\C)$ & $\SO(2p-1,\C)$ & $D_p$ & $\{\alpha_{p-1}\}$ or $\{\alpha_p\}$\tabularnewline
III$_i^d$ & $G_0\times G_0$ & $\mathrm{Diag}(G_0)$ & $A_1\times A_1$ & $\{\alpha_1\}$ or $\{\alpha_2\}$\tabularnewline
IV$_1$ & $\SO^*(4p)$ & $\SO^*(4p-2)\times\SO^*(2)$ & $C_p$ & $\{\alpha_p\}$\tabularnewline
IV$_1^d$ & $\SO(2p,2)$ & $\U(p,1)$ & $B_2$ (if $p>1$) & $\{\alpha_1\}$ (if $p>1$)\tabularnewline
& $\SO(2,2)$ & $\U(1,1)$ & $A_1\times A_1$ & $\{\alpha_1\}$ or $\{\alpha_2\}$\tabularnewline
IV$_2^d$ & $\SU(2p,2)$ & $\Sp(p,1)$ & $(BC)_2$ (if $p>1$) or $C_2$ & $\{\alpha_1\}$\tabularnewline
IV$_3^d$ & $E_{6(-14)}$ & $F_{4(-20)}$ & $(BC)_2$ & $\{\alpha_2\}$\tabularnewline
\hline
\end{tabular}
\vspace{0.2cm}
\caption{List of all irreducible pseudo-Riemannian symmetric spaces $G/H$ of split rank~$1$ with $H$ noncompact and $\Rrank(G) - \Rrank(H) = 1$, up to coverings and compact factors. In case~I$_1$ we assume $p>q\geq 1$. In case~I$'_1$ we assume $p\geq 3$. In cases I$_2$, I$_3$, III$_1$, IV$_1$, IV$_1^d$, IV$_2^d$ we assume $p\geq q\geq 1$. In case~III$_i^d$ we assume $\Rrank(G_0)=1$. The last column gives a subset $\theta$ of the simple roots $\Delta$ such that Proposition~\ref{prop:cork-1-proper-compact->Ano} applies with $P=P_{\theta}$.}
\label{table2}
\end{table}

In general, properly discontinuous and cocompact actions are not necessarily fully characterized by an Anosov property in~$G$.
However, we will see (Proposition~\ref{prop:corank-1-sharp<->Ano} and Corollary~\ref{cor:corank-1-sharp-qi-open}) that they are essentially characterized by an Anosov property \emph{after composing with some appropriate finite-dimensional linear representation of~$G$}.
This will enable us to prove Theorem~\ref{thm:openness}.

\begin{remark}
When $G/H$ has real corank one and $G$ is semisimple, the sharp embeddedness condition in~$G$ with respect to~$H$ is equivalent to the \emph{$\omega$-undistorted} condition introduced independently by Davalo and Riestenberg in \cite{dr}, where $\omega$ is a linear form on $\aaa$ such that $\mu(H) = W\cdot \mathrm{Ker}(\omega) \cap \aaa^+$.
These authors also make a link with Anosov representations, and mention that as a consequence of their work $\omega$-undistortedness is an open condition.
We prove this openness in the present paper (see Corollary~\ref{cor:corank-1-sharp-qi-open}) with a completely different approach.
\end{remark}

%%%%%%%%%%%%%%%%%%%%%%%%%
\subsection{Application 3: nonexistence of compact quotients}

As mentioned in Section~\ref{subsec:intro-cpt-quot}, the existence problem for compact quotients of reductive homogeneous spaces has a long history and has given rise to a very rich literature.

In order to act properly discontinuously and cocompactly on $G/H$, a discrete subgroup $\Gamma$ of~$G$ must have sufficiently large \emph{virtual cohomological dimension}, as described by the following classical fact.
Recall that by the Cartan--Iwasawa--Malcev theorem (see \cite{bor-bourbaki50}), every connected Lie group (and indeed every connected locally compact group) admits maximal compact subgroups and they are all conjugate to one another.

\begin{fact} \label{fact:vcd}
Let $G$ be a real connected Lie group and $H$ a closed connected subgroup of~$G$.
Let $K$ be a maximal compact subgroup of~$G$ and $K_H$ a maximal compact subgroup of~$H$, with $K_H\subset K$.
Then for any finitely generated discrete subgroup $\Gamma$ of~$G$ acting properly discontinuously on $G/H$, the virtual cohomological dimension of~$\Gamma$ satisfies
\[\mathrm{vcd}(\Gamma) \leq \dim(G/K) - \dim(H/K_H) ,\]
with equality if and only if the action of $\Gamma$ on $G/H$ is cocompact.
\end{fact}

We refer to Section~\ref{subsec:proof-vcd} below for the notion of virtual cohomological dimension.

Fact~\ref{fact:vcd} is proved in \cite[Cor.\,5.5]{kob89} for $G/H$ of reductive type, and in \cite[Lem.\,2.2]{mor17} in general. We give a short proof in Section~\ref{subsec:proof-vcd} for the reader's convenience.

On the other hand, when $\Gamma$ is the image of an Anosov representation to~$G$, the existence of a transverse boundary map from $\partial_\infty \Gamma$ into some flag variety of~$G$ gives an upper bound on the covering dimension of $\partial_\infty \Gamma$, hence on the virtual cohomological dimension of~$\Gamma$ (see Section~\ref{subsec:Ano-vcd}).
In many cases, this bound prevents $\Gamma$ from acting cocompactly on $G/H$, and Theorem~\ref{thm:proper-compact<->Ano} or Proposition~\ref{prop:cork-1-proper-compact->Ano} then implies that $G/H$ does not admit any compact quotients.
This approach enables us for instance to prove the following.

\begin{corollary} \label{cor:no-cpt-quot}
The homogeneous spaces $G/H$ given in Table~\ref{table3} do not admit any compact quotients for $k\geq 2$.
\end{corollary}

\begin{table}[!ht]
\centering
\begin{tabular}{|c|c|c|c|c|}
\hline
& $G$ & $H$ & $\dim(G/K) - \dim(H/K_H)$\tabularnewline
\hline
(i) & $\SL(2\ell,\mathbb{K})$ & $\SL(2\ell-1,\mathbb{K})$ & $2\ell$ ($\mathbb{K}\!=\!\R$), $4\ell-1$ ($\mathbb{K}\!=\!\C$), or $8\ell-3$ ($\mathbb{K}\!=\!\mathbb{H}$)\tabularnewline
(ii) & $\Sp(2\ell,\mathbb{K})$ & $\Sp(2\ell-2,\mathbb{K})$ & $2\ell$ ($\mathbb{K}\!=\!\R$) or $4\ell-1$ ($\mathbb{K}\!=\!\C$)\tabularnewline
(iii) & $\SO(2k+k',2)$ & $\U(k,1)$ & $2(k+k')$\tabularnewline
(iv) & $\SU(2k+k',2)$ & $\Sp(k,1)$ & $4(k+k')$\tabularnewline
(v) & $\SO(\ell,\ell)$ & $\SL(\ell,\R)$ & $(\ell^2-\ell+2)/2$\tabularnewline
(vi) & $\SU(\ell,\ell)$ & $\SL(\ell,\C)$ & $\ell^2+1$\tabularnewline
(vii) & $\Sp(\ell,\ell)$ & $\SL(\ell,\mathbb{H})$ & $2\ell^2+\ell+1$\tabularnewline
(viii) & $\SO^*(4\ell)$ & $\SL(\ell,\mathbb{H})$ & $2\ell^2-\ell+1$\tabularnewline
(ix) & $\SO(\ell+1,\ell)$ & $\mathrm{S}(\OO(2)\times\OO(\ell-1,\ell)\mathrm{)}$ & $2\ell$\tabularnewline
(ix)' & $\SU(\ell+1,\ell)$ & $\mathrm{S}(\U(2)\times\U(\ell-1,\ell)\mathrm{)}$ & $4\ell$\tabularnewline
(ix)'' & $\Sp(\ell+1,\ell)$ & $\Sp(2)\times\Sp(\ell-1,\ell)\mathrm{)}$ & $8\ell$\tabularnewline
(x) & $\SO^*(4\ell)$ & $\U(\ell+1,\ell-1)$ & $2(\ell^2-\ell+1)$\tabularnewline
(x)' & $\SO^*(4\ell+2)$ & $\U(\ell+2,\ell-1)$ & $2(\ell^2+2\ell-2)$\tabularnewline
(xi) & $\mathrm{Spin}(5,3)$ & $G_{2(2)}$ & $7$\tabularnewline
(xii) & $\mathrm{Spin}(5,4)$ & $\mathrm{Spin}(4,3)$ & $8$\tabularnewline
\hline
\end{tabular}
\vspace{0.2cm}
\caption{Some homogeneous spaces $G/H$ that do not admit any compact quotients for $k,k'\geq 1$ and $\ell\geq 2$ (Corollary~\ref{cor:no-cpt-quot}); we require $\ell\geq 3$ in cases (v) and~(x), and $\ell\geq 4$ in case~(viii). Here $\mathbb{K}$ is $\R$, $\C$, or (in case~(i)) the ring $\mathbb{H}$ of quaternions. In case~(vii) we consider the composition of the spinor embedding $\mathrm{Spin}(4,3) \hookrightarrow \mathrm{Spin}(4,4)$ with the standard embedding $\mathrm{Spin}(4,4) \hookrightarrow \mathrm{Spin}(5,4)$.}
\label{table3}
\end{table}

Cases (ix)--(x)' were previously treated in \cite{kob92} by a different method, but it seems that all other cases are new.
(Note that $G/H$ does admit compact quotients in cases (v) and~(x) for $\ell=2$ --- the group $\SO^*(8)$ is locally isomorphic to $\SO(6,2)$.)

In particular, Corollary~\ref{cor:no-cpt-quot} states that $\SO(2,2k+1)/\U(1,k)$ does not admit any compact quotients for $k\geq 1$.
This had been conjectured by Oh--Witte \cite[Conj.\,1.8]{ow02}, and implies, by \cite[Th.\,1.9]{ow02}, that the only closed connected subgroups $H$ of $G=\SO(2,2k+1)$ for which $G/H$ admits compact quotients are the subgroups $H$ that are either compact or cocompact in~$G$.

Interestingly, some cases of Table~\ref{table3} give subgeometries without compact quotients of geometries with compact quotients: see Table~\ref{table4}.
This means that, in these cases, $G$ embeds into some semisimple Lie group~$G'$ and acts transitively on some homogeneous space $G'/H'$ of~$G'$ with $G\cap H'= H$; furthermore, there exist discrete subgroups $\Gamma$ of~$G'$ acting properly discontinuously and cocompactly on $G'/H'$ (see \cite{ky05}), but Corollary~\ref{cor:no-cpt-quot} shows that such $\Gamma$ cannot be contained in~$G$.

\begin{table}[!ht]
\centering
\begin{tabular}{|c|c|c|}
\hline
& $G/H$ & $G'/H'$\tabularnewline
\hline
(i) & $\SL(4,\R)/\SL(3,\R)$ & $\SO(4,4)/\SO(4,3)$\tabularnewline
& $\SL(8,\R)/\SL(7,\R)$ & $\SO(8,8)/\SO(8,7)$\tabularnewline
& $\SL(4,\C)/\SL(3,\C)$ & $\SO(8,\C)/\SO(7,\C)$\tabularnewline
(ii) & $\Sp(4,\R)/\Sp(2,\R)$ & $\SO(4,4)/\SO(4,3)$\tabularnewline
& $\Sp(8,\R)/\Sp(6,\R)$ & $\SO(8,8)/\SO(8,7)$\tabularnewline
& $\Sp(4,\C)/\Sp(2,\C)$ & $\SO(8,\C)/\SO(7,\C)$\tabularnewline
(iii) & $\SO(2k+1,2)/\U(k,1)$ & $\SO(2k+2,2)/\U(k+1,1)$\tabularnewline
(iv) & $\SU(2k+1,2)/\Sp(k,1)$ & $\SU(2k+2,2)/\Sp(k+1,1)$\tabularnewline
(xi) & $\mathrm{Spin}(5,3)/G_{2(2)}$ & $\SO(8,\C)/\SO(7,\C)$\tabularnewline
(xii) & $\mathrm{Spin}(5,4)/\mathrm{Spin}(4,3)$ & $\SO(8,8)/\SO(8,7)$\tabularnewline
\hline
\end{tabular}
\vspace{0.2cm}
\caption{Homogeneous spaces $G/H$ that do not admit compact quotients, but which are a subgeometry of a homogeneous space $G'/H'$ that does admit compact quotients}
\label{table4}
\end{table}

Corollary~\ref{cor:no-cpt-quot} also states that $\SL(n,\R)/\SL(n-1,\R)$ does not admit compact quotients for even $n\geq 4$.
Note that it had been an open conjecture since the early 1990s that $\SL(n,\R)/\SL(m,\R)$ does not admit compact quotients for $n>m>1$.
The case $n > \lceil 3m/2\rceil$ was treated by Kobayashi \cite{kob92} in 1992 using a maximality principle, the case $n\geq m+3 \geq 5$ was treated in a series of papers in the 1990s by Labourie, Mozes and Zimmer \cite{zim94,lmz95,lz95} using cocycle rigidity, the case $n=m+1$ odd was treated by Benoist \cite{ben96} in 1996 using Lie theory, the case $m=2$ by Shalom \cite{sha00} in 2000 using unitary representations, and the case $m$ even by Morita \cite{mor17} and the second-named author \cite{tho-vol} in the mid-2010s using cohomological considerations.
Thus the remaining open cases prior to the present paper were $n=m+1$ even and $n=m+2$ odd.
Corollary~\ref{cor:no-cpt-quot} treats the case of $n=m+1$ even.
It also applies to $\SL(n,\mathbb{K})/\SL(m,\mathbb{K})$ for $\mathbb{K}=\C$ or the quaternions.

\begin{remark}
In an independent paper with Yosuke Morita \cite{kmt}, we give a different obstruction to the existence of compact quotients, which is of a topological nature and does not involve sharpness.
This enables us to prove in \cite{kmt} that $\SL(n,\R)/\SL(m,\R)$ does not admit compact quotients for all values of $(n,m)$ except possibly $(4,3)$ and $(8,7)$.
The cases of $\SL(4,\R)/\SL(3,\R)$ and $\SL(8,\R)/\SL(7,\R)$ are covered by Corollary~\ref{cor:no-cpt-quot} of the present paper; in these cases, sharpness seems to be the only known obstruction to the existence of compact quotients.
\end{remark}

We actually prove a version of Theorem~\ref{thm:sharp} for all \emph{compactly generated} (not necessarily discrete) closed subgroups $\Gamma$ of~$G$ (see Theorem~\ref{thm:sharp-general} below).
It has the following consequence.

\begin{theorem} \label{thm:cpt-quotient->H-QI}
Let $G$ be a connected real linear reductive Lie group and $H$ a closed subgroup of~$G$.
If $G/H$ admits compact quotients, then $H$ must be compactly generated and quasi-isometrically embedded in~$G$.
\end{theorem}

This also provides an obstruction to the existence of compact quotients of $G/H$ for certain nonreductive subgroups $H$, see Section~\ref{subsec:H-QI}.

%%%%%%%%%%%%%%%%%%%%%%%%
\subsection{Application 4: compact quotients of group manifolds of type~$A_2$}

A \emph{group manifold} is a semisimple Lie group $G_0$ seen as a homogeneous space under the action of $G_0 \times G_0$ by left and right multiplication:
\[(g,h)\cdot  y = g y h^{-1}~.\]
Note that the stabilizer of the identity element of~$G_0$ is the diagonal $\Diag(G_0)$ of $G_0\times G_0$.
Therefore a group manifold is a homogeneous space of the form $G/H = (G_0\times G_0)/\Diag(G_0)$ where $G_0$ is semisimple.

Compact quotients of group manifolds of real rank $1$ were the first to be thoroughly studied and well understood (see \cite{kas-PhD,gk17,ggkw17}).
We hope that the Sharpness Theorem~\ref{thm:sharp} will help understand the higher-rank case.
To illustrate this, we use Theorem~\ref{thm:sharp} to classify compact quotients of some group manifolds of real rank~$2$:

\begin{theorem} \label{thm:group-mfd-SL3}
Let $G_0$ be a real linear simple Lie group with a restricted root system of type~$A_2$ (\ie $G_0$ is locally isomorphic to $\SL(3,\R)$, $\SL(3,\C)$, $\SL(3,\mathbb{H})$, or $E_{6(-26)}$).
Let $\Gamma$ be a discrete subgroup of $G_0\times G_0$ acting properly discontinuously and cocompactly on\linebreak $(G_0\times G_0)/\mathrm{Diag}(G_0)$.
Then, up to finite index and to switching the two factors of $G_0\times G_0$, we have $\Gamma = \Gamma_0 \times \{1_{G_0}\}$ where $\Gamma_0$ is a uniform lattice in~$G_0$.
\end{theorem}

Note that, unlike in the case where $G_0$ has real rank one, in Theorem~\ref{thm:group-mfd-SL3} the groups $\Gamma = \Gamma_0\times\{ 1\}$ are locally rigid in $G = G_0\times G_0$ by work of Raghunathan \cite{rag65}.

Our proof of Theorem~\ref{thm:group-mfd-SL3} relies crucially on the fact that when the restricted root system is of type~$A_2$, the opposition involution is nontrivial. For real simple Lie groups of rank~$2$ with trivial opposition involution, there may exist other discrete subgroups of $G_0\times G_0$ acting properly discontinuously and cocompactly on $(G_0\times G_0)/\mathrm{Diag}(G_0)$ (see \cite[\S\,2.2]{kk16}), including some that are not rigid.
For instance, for $G_0 = \SO(2n,2)$ we can take $\Gamma = \Gamma_1\times\Gamma_2$ where $\Gamma_1$ is a uniform lattice of $\U(n,1)$ and $\Gamma_2$ acts properly discontinuously and cocompactly on $\SO(2n,2)/\U(n,1)$; such $\Gamma_2$ can be deformed \eg by bending (see \cite{kas12}).
This example leads us to formulate the following general conjecture; we refer to Section~\ref{subsec:lim-cone} for the notion of limit cone.

\begin{conjecture} \label{conj:group-mfd}
Let $G_0$ be any real linear simple Lie group of real rank $\geq 2$, with maximal compact subgroup~$K_0$, and let $\Gamma$ be any discrete subgroup of $G_0\times G_0$ acting properly discontinuously and cocompactly on $(G_0\times G_0)/\mathrm{Diag}(G_0)$.
Then, up to replacing $\Gamma$ by a finite-index subgroup, we have $\Gamma = \Gamma_1 \times \Gamma_2$ where $\Gamma_1$ and~$\Gamma_2$ are discrete subgroups of~$G_0$ with disjoint limit cones and
\[\mathrm{vcd}(\Gamma_1) + \mathrm{vcd}(\Gamma_2) = \dim(G_0/K_0) .\]
\end{conjecture}

In Conjecture~\ref{conj:group-mfd}, one of the groups $\Gamma_1$ or~$\Gamma_2$ is allowed to be trivial (hence with empty limit cone); this case, called \emph{standard}, appears in Theorem~\ref{thm:group-mfd-SL3}.

The proof of Theorem~\ref{thm:group-mfd-SL3} includes some steps towards a proof of Conjecture~\ref{conj:group-mfd}.
In general, we are not able for the moment to rule out the possibility that both projections of $\Gamma$ are faithful with dense image in~$G_0$.
An open conjecture, first formulated by Benoist in the 1990s (see \cite[Q.\,3.4--3.5]{bflm26}), predicts that a discrete subgroup of $G_0 \times G_0$ with dense projections must be a lattice (hence cannot act properly discontinuously on $(G_0\times\nolinebreak G_0)/\mathrm{Diag}(G_0)$).
Solving this conjecture would be a major step forward in the investigation of compact quotients of general group manifolds $(G_0\times G_0)/\mathrm{Diag}(G_0)$.

%%%%%%%%%%%%%%%%%%%%%%%%%
\subsection{Organization of the paper}

Section~\ref{sec:Cartan-polar} is devoted to some reminders about the Cartan decomposition of a reductive Lie group~$G$, the polar decomposition of~$G$ with respect to a reductive subgroup~$H$, and compactly generated groups.
In Section~\ref{sec:proof-sharpness} we prove the Sharpness Theorem~\ref{thm:sharp} by establishing a slightly more general version of it (Theorem~\ref{thm:sharp-general}).
In Section~\ref{sec:sharp-Ano} we relate sharp embeddedness to Anosov representations, and prove Theorem~\ref{thm:proper-compact<->Ano} and Proposition~\ref{prop:cork-1-proper-compact->Ano}.
In Section~\ref{sec:no-cpt-quot} we prove Corollary~\ref{cor:no-cpt-quot} about the nonexistence of compact quotients for the homogeneous spaces $G/H$ of Table~\ref{table3}.
In Section~\ref{sec:openness} we study further the links between sharp embeddedness and Anosov representations, via linear representations of~$G$, and deduce Theorem~\ref{thm:openness} stating the openness of proper and cocompact actions in the corank-one case.
In Section~\ref{sec:P-Popp} we prove a general preliminary result on discrete groups acting properly discontinuously and cocompactly on group manifolds.
We use it in Section~\ref{sec:quotients-SL3} to prove Theorem~\ref{thm:group-mfd-SL3} about compact quotients of group manifolds of type~$A_2$.

%%%%%%%%%%%%%%%%%%%%%%%%%
\subsection{Acknowledgements}

We are grateful to Yosuke Morita for pointing out some of the examples of Table~\ref{table3}, and for interesting discussions around compact quotients which led to our joint work \cite{kmt}.

The results of this paper (with the exception of Theorem~\ref{thm:group-mfd-SL3}) were presented at the \emph{Geometry Winter Workshop} in Luxembourg in January 2022 and at the birthday conferences of Grigory Margulis, Marc Burger, and Toshiyuki Kobayashi in Chicago, Zurich, and Tokyo in 2022; we thank the organizers of these events.

Progress on Theorem~\ref{thm:group-mfd-SL3} was made while the first-named author was in residence at the Institute for Advanced Study in Princeton in the Spring 2024, supported by the National Science Foundation under Grant No. DMS-1926686; she thanks the IAS for its hospitality and excellent working conditions.

%%%%%%%%%%%%%%%%%%%%%%%%%%%%%%%%%%%%%%%%%%%%%%%%%%%
\section{Reminders: Cartan and polar decompositions} \label{sec:Cartan-polar}

In the whole paper, we fix a noncompact connected real linear reductive Lie group $G$, with Lie algebra~$\g$.

%%%%%%%%%%%%%%%%%%%%%%%%%
\subsection{Cartan decomposition} \label{subsec:Cartan-decomp}

Let $K$ be a maximal compact subgroup of~$G$; it is the set of fixed points of some Cartan involution $\sigma$ of~$G$. 
The Cartan decomposition $\g = \kk + \p$ holds, where $\kk = \Lie(K)$ and $\p$ are respectively the fixed points and anti-fixed points of the involution $\mathrm{d}\sigma$ in~$\g$.
Let $\aaa$ be a maximal abelian subspace of~$\p$ (also known as a \emph{Cartan subspace} of~$\g$), and let $\Sigma = \Sigma(\g,\aaa)$ be the set of restricted roots of $\aaa$ in~$\g$, namely the set of linear forms $\alpha\in\aaa^*$ such that $\g_{\alpha} \equaldef \{ v\in\g ~|~ [a,v] = \alpha(a) v\ \text{for all }a\in\aaa\}$ is nonzero.
We choose a set $\Delta$ of simple restricted roots of $\aaa$ in~$\g$, \ie a subset of~$\Sigma$ such that any element of~$\Sigma$ can be written as a linear combination of elements of~$\Delta$ with all coefficients nonnegative or all coefficients nonpositive.
We denote by $\Sigma^+$ the set of elements of~$\Sigma$ that can be written as a linear combination of elements of~$\Delta$ with all coefficients nonnegative.
The choice of~$\Delta$ defines a closed positive Weyl chamber
$$\aaa^+ \equaldef \{ a\in\aaa ~|~ \alpha(a)\geq 0\ \text{for all }\alpha\in\Delta\},$$
which is a closed convex cone in~$\aaa$, and a fundamental domain for the action of the restricted Weyl group $W = N_G(\aaa)/Z_G(\aaa)$ on~$\aaa$.

The Cartan decomposition $G = K \exp(\aaa^+)K$ holds: any element $g\in G$ can be written $g = k\exp(a)k'$ for some $k,k'\in K$ and a unique $a\in\aaa^+$.
Setting $\mu(g) \equaldef a$ defines a map $\mu : G\to\aaa^+$, called the \emph{Cartan projection} associated to the Cartan decomposition $G = K\exp(\aaa^+)K$.
It is continuous, proper, and surjective.

Let $w_0$ be the longest element of the Weyl group~$W$ (\ie $w_0 \cdot \aaa^+ = -\aaa^+$). 
Then $\iota \equaldef -w_0 : \aaa^+ \to \aaa^+$ is the \emph{opposition involution}.
We have $\mu(g^{-1}) = \iota(\mu(g))$ for all~$g\in G$.

\begin{example} \label{ex:Cartan-decomp-SLn}
Let $n\geq 2$.
If $G = \SL(n,\mathbb{K})$ where $\mathbb{K} = \R$ (\resp $\C$, \resp the ring $\mathbb{H}$ of quaternions), then we may take $K$ to be $\SO(n)$ (\resp $\SU(n)$, \resp $\Sp(n)$) and $\aaa$ to be the set of diagonal matrices of the form $\mathrm{diag}(t_1,\dots,t_n)$ where $t_1,\dots,t_n\in\R$ satisfy $t_1 + \dots + t_n = 0$.
If $G$ is the identity component of $\GL(n,\R)$, then we may take $K$ to be $\SO(n)$ and $\aaa$ to be the set of all diagonal matrices $\mathrm{diag}(t_1,\dots,t_n)$ with $t_1,\dots,t_n\in\R$.
In either case, we may take $\aaa^+$ to be the closed subcone of~$\aaa$ defined by $t_1\geq\dots\geq t_n$.
The Cartan projection $\mu : G\to\aaa^+$ takes $g\in G$ to $\mathrm{diag}(\mu_1(g),\dots,\mu_n(g))$ where $e^{\mu_1(g)}\geq\dots\geq e^{\mu_n(g)}>0$ are the square roots of the eigenvalues of ${}^t\!\overline{g}g$.
The opposition involution $\iota : \aaa^+\to\aaa^+$ maps $\mathrm{diag}(t_1,\dots,t_n)$ to $\mathrm{diag}(-t_n,\dots,-t_1)$.
The set of simple restricted roots associated to~$\aaa^+$ is $\Delta = \{\alpha_1,\dots,\alpha_{n-1}\}$ where $\alpha_i : \mathrm{diag}(t_1,\dots,t_n) \mapsto t_i$.
\end{example}

%%%%%%%%%%%%%%%%%%%%%%%%%
\subsection{Vector-valued distance function on $G/K$} \label{subsec:vector-dist}

Let $(\cdot,\cdot)$ be a $G$-invariant nondegenerate symmetric bilinear form on~$\g$ which is positive definite on~$\p$, negative definite on~$\kk$, and for which $\p$ and~$\kk$ are orthogonal.
(If $G$ is semisimple, we can take $(\cdot,\cdot)$ to be the Killing form of~$\g$.)
Let $\Vert \cdot \Vert$ be the associated Euclidean norm on~$\p$.
Viewing $\p = \kk^{\perp}$ as the tangent space $T_{x_0} G/K$, the restriction of $(\cdot,\cdot)$ to~$\p$ extends to a $G$-invariant Riemannian metric on $G/K$, making $G/K$ a Riemannian symmetric space of nonpositive curvature.
We denote by $d_{G/K}$ the corresponding distance function.

The Cartan projection $\mu : G\to\aaa^+$ of Section~\ref{subsec:Cartan-decomp} is by construction invariant under left and right multiplication by~$K$, and therefore defines a ``vector-valued distance function''\ $\multidist : G/K \times G/K \to \aaa^+$ on the symmetric space $G/K$, given by 
\begin{equation} \label{eqn:multidist}
\multidist(g\cdot x_0, g'\cdot x_0) = \mu(g^{-1} g')
\end{equation}
for all $g,g'\in G$ (see \cite[\S\,2.2.3]{par12}).
It satisfies the following properties:
\begin{enumerate}
  \item $\multidist$ is $G$-invariant: $\multidist(g\cdot x,g\cdot x') = \multidist (x,x')$ for all $g\in G$ and all $x,x'\in G/K$;
  \item $\multidist(x',x) = \iota (\multidist(x,x'))$, where $\iota : \aaa^+ \to \aaa^+$ is the opposition involution;
  \item $G$ acts $2$-transitively on $(G/K, \multidist)$: given $x_1,x_2,x_1', x_2' \in G/K$, there exists $g\in G$ such that $x_1'=g\cdot x_1$ and $x_2'= g\cdot x_2$ if and only if $\multidist(x_1,x_2) = \multidist(x_1',x_2')$;
  \item the norm of~$\multidist$ is the Riemannian distance function $d_{G/K}$: for any $x,x'\in G/K$,
  \begin{equation} \label{eqn:norm-multidist}
  \Vert \multidist(x,x') \Vert = d_{G/K}(x,x') .
  \end{equation}
\end{enumerate}

Using \eqref{eqn:multidist} and \eqref{eqn:norm-multidist}, one can show (see \eg \cite[Lem.\,2.1]{kas08}) that $\mu$ is ``strongly subadditive'' in the following sense: for any $g,g_1,g_2\in G$,
\begin{equation} \label{eqn:mu-subadd}
\Vert\mu(g_1 g g_2) -\mu(g)\Vert \leq \Vert\mu(g_1)\Vert + \Vert\mu(g_2)\Vert .
\end{equation}

%%%%%%%%%%%%%%%%%%%%%%%%%
\subsection{Cartan decomposition for~$H$} \label{subsec:Cartan-decomp-H}

We now fix (until the end of Section~\ref{subsec:polar-decomp}) a closed connected subgroup $H$ of~$G$ which is stable under the Cartan involution $\sigma$ of~$G$.
Then $H$ is itself a real linear reductive Lie group, with Cartan involution the restriction of~$\sigma$, and with maximal compact subgroup $K_H \equaldef K\cap H$; the Lie algebra $\h$ of~$H$ admits the Cartan decomposition $\h = (\kk\cap\h) + (\p\cap\h)$.
The pair $(G,H)$ is called a \emph{reductive pair} and the homogeneous space $G/H$ is said to be \emph{of reductive type}.

Let $\aaa_H$ be a maximal abelian subspace of $\p\cap\h$.
We may and shall choose the maximal abelian subspace $\aaa$ of~$\p$ from Section~\ref{subsec:Cartan-decomp} to contain~$\aaa_H$.
The Cartan decomposition of~$H$ states that every $h\in H$ can be written $h = k \exp(a) k'$ for some $k, k' \in K_H$ and some $a$ in (some $H$-Weyl chamber of)~$\aaa_H$; while the vector $a$ need not be contained in~$\aaa^+$, there exists an element $w \in W$ such $w \cdot a \in \aaa^+$ and we have $\mu(h) = w\cdot a$.
Thus the image of $H$ under the Cartan projection $\mu : G\to\aaa^+$ is
\begin{equation} \label{eqn:mu-H}
\mu(H) =  \aaa^+ \cap \bigcup_{w \in W} w \cdot \aaa_H ,
\end{equation}
\ie $\mu(H)$ is the intersection of $\aaa^+$ with the finitely many linear subspaces $w\cdot\aaa_H$ of~$\aaa$, for $w\in W$.

%%%%%%%%%%%%%%%%%%%%%%%%%
\subsection{A duality between $G/K$ and $G/H$}

Throughout this section, we denote points in the Riemannian symmetric space $G/K$ with the letter~$x$ and points in the homogeneous space $G/H$ with the letter~$y$.
In particular, we denote by
\begin{itemize}
  \item $x_0$ the class of $\1_G$ in $G/K$, i.e. the point fixed by $K$,
  \item $y_0$ the class of $\1_G$ in $G/H$, i.e. the point fixed by $H$.
\end{itemize}

As mentioned in Section~\ref{subsec:vector-dist}, the $G$-invariant nondegenerate symmetric bilinear form $(\cdot,\cdot)$ on~$\g$, restricted to $\p = \kk^{\perp} \simeq T_{x_0} G/K$, extends to a $G$-invariant Riemannian metric on $G/K$, making $G/K$ a Riemannian symmetric space of nonpositive curvature.
On the other hand, the restriction of $(\cdot,\cdot)$ to $\q \equaldef \h^{\perp} \simeq T_{y_0} G/H$, extends to a $G$-invariant \emph{pseudo-Riemannian} metric on $G/H$, of signature $(\mathtt{d}_+,\mathtt{d}_-)$ as in \eqref{eqn:d+d-}.

The $K$-orbit $\mathcal{K}_0 \equaldef K\cdot y_0$ in $G/H$ is a totally geodesic timelike subspace of dimension $\mathtt{d}_-$, isomorphic to $K/K_H$, and maximal for inclusion among all totally geodesic timelike subspaces of $G/H$.
Its stabilizer in $G$ is the group $K$, hence we can identify $G/K$ with the space of $G$-translates of~$\mathcal{K}_0$.
Given $x = g\cdot x_0 \in G/K$ (where $g\in G$), we denote by $\mathcal{K}_x \equaldef g \cdot \mathcal{K}_0$ the corresponding maximal timelike subspace of $G/H$.

Dually, the $H$-orbit $\mathcal H_0\equaldef H\cdot x_0$ in $G/K$ is a totally geodesic subspace of $G/K$, isomorphic to $H/K_H$, and we can identify $G/H$ with the space of $G$-translates of~$\mathcal H_0$.
Given $y = g\cdot y_0 \in G/H$ (where $g\in G$), we denote by $\mathcal H_y \equaldef g \cdot \mathcal H_0$ the corresponding totally geodesic subspace of $G/K$.

For $x\in G/K$ and $y\in G/H$, we then have the following equivalence:
\begin{equation} \label{eq: Duality G/H G/K}
x\in \mathcal H_y \Longleftrightarrow y \in \mathcal{K}_x .
\end{equation}

%%%%%%%%%%%%%%%%%%%%%%%%%
\subsection{Polar decomposition with respect to $H$} \label{subsec:polar-decomp}

As above, let $\p$ and~$\q$ be the orthogonal complements in~$\g$ of $\kk$ and~$\h$, respectively, with respect to the $G$-invariant nondegenerate symmetric bilinear form $(\cdot,\cdot)$.
Then $\p \cap \q$ can be identified with the orthogonal of $T_{x_0} \mathcal H_0$ in $T_{x_0} G/K$, and also with the orthogonal of $T_{y_0} \mathcal{K}_0$ in $T_{y_0} G/H$.

The following polar decomposition is well known when $G/H$ is symmetric; we refer to \cite[Lem.\,2.7]{kob89} for a proof in the general case.

\begin{fact} \label{fact:polar-decomp}
Any $g\in G$ can be written as $g = k \exp(b) h$ for some $k\in K$, some $b\in\p\cap\q$, and some $h\in H$.
Moreover, this decomposition is unique up to~$K_H$: if $k \exp(b) h = k'\exp(b') h'$ where $k,k'\in K$ and $b,b'\in\p\cap\q$ and $h,h'\in H$, then there exists $\ell\in K_H$ such that $k'=k\ell^{-1}$ and $b' = \Ad_\ell(b)$ and $h' = \ell h$.
\end{fact}

For $g = k\exp(b)h$ as above, we denote by $\nu(g)$ the class of $b$ in $K_H \backslash (\p \cap \q)$ (which is well defined by the second part of Fact~\ref{fact:polar-decomp}) and by $\Vert\nu(g)\Vert$ its norm.
By construction, $\Vert \nu (\cdot)\Vert$ is invariant under left multiplication by~$K$ and right multiplication by~$H$.

\begin{remark}
If $H= K$, then $\q = \p$ and $K_H \backslash \p \simeq \aaa^+$ and $\mu = \nu$.
More generally, suppose $G/H$ is a \emph{pseudo-Riemannian symmetric space}, \ie $H$ is the group of fixed points of some involution $\tau$ of~$G$ (which commutes with our chosen Cartan involution~$\sigma$).
Then the space $K_H \backslash (\p \cap \q)$ of $K_H$-orbits in $\p \cap \q$ identifies with a closed Weyl chamber $\bb^+$ in a Cartan subspace $\bb$ of the Lie subalgebra $\g^{\sigma\tau}$ of~$\g$ consisting of those vectors fixed by the derivative of $\sigma \tau$ at $\1_G$.
\end{remark}

We define a function $D : G/K \times G/H \to \R_{\geq 0}$ by setting, for any $x = g\cdot x_0 \in G/K$ and any $y= g' \cdot y_0 \in G/H$,
\begin{equation} \label{eqn:D-x-y}
D(x,y) = \Vert \nu(g^{-1} g') \Vert .
\end{equation}
The function $D$ is invariant under the diagonal action of $G$ on $G/K \times G/H$.
It admits the following dual interpretations. Consider again $x = g\cdot x_0 \in G/K$ and $y= g' \cdot y_0 \in G/H$, and (using Fact~\ref{fact:polar-decomp}) write $g^{-1} g' = k \exp(b) h$ where $k\in K$ and $b\in \p\cap \q$ and $h\in H$.\\

\noindent {\it Riemannian interpretation of $D$:} The curve $c= \left(\exp(-tb)\cdot x_0\right)_{t\in [0,1]}$ is a geodesic segment from $x_0$ to $\exp(-b) \cdot x_0$ in $G/K$.
It has length $\Vert b \Vert= \Vert \nu(g^{-1} g')\Vert$ and is orthogonal at~$x_0$ to the totally geodesic subspace $\mathcal H_0$.
The relation
\[g' h^{-1} \exp(-b) = g k\]
implies that $g'h^{-1}\cdot c$ is a geodesic segment of length $\Vert \nu(g^{-1} g')\Vert$ from
\[g'h^{-1} \cdot x_0 \in g' \cdot \mathcal H_0 = \mathcal H_y\]
to $g' h^{-1}\exp(-b) \cdot x_0 = x$, which is orthogonal at $g'h^{-1} \cdot x_0$ to~$\mathcal H_y$.
Thus $g'h^{-1} \cdot x_0$ is the orthogonal projection of $x$ to $\mathcal H_y$, and 
\[D(x,y) = d_{G/K}(x, \mathcal H_y) \]
is the Riemannian distance from $x$ to the totally geodesic subspace $\mathcal H_y$ in $G/K$.\\

\noindent {\it Pseudo-Riemannian interpretation of $D$:} Dually, the curve $c' = \left(\exp(tb)\cdot y_0\right)_{t\in [0,1]}$ is a spacelike geodesic segment from $y_0$ to $\exp(b)\cdot y_0$ in $G/H$.
It is orthogonal at~$y_0$ to the maximal timelike subspace $\mathcal{K}_0$.
Therefore $gk \cdot c'$ is a spacelike geodesic segment from
\[gk \cdot y_0 \in g \cdot \mathcal{K}_0 = \mathcal{K}_x\]
to $gk \exp(b) \cdot y_0 = g' h^{-1} \cdot y_0 = y$, which is orthogonal at $gk \cdot y_0$ to~$\mathcal{K}_x$.
Thus $gk \cdot y_0$ is the orthogonal projection $\pi_{\mathcal{K}_x}(y)$ of $y$ to $\mathcal{K}_x$, and $D(x,y)$ is the length of the spacelike geodesic segment between $y$ and $\pi_{\mathcal{K}_x}(y)$ in $G/H$.

\begin{remark}
The projection $\pi_{\mathcal{K}_x}(y)$ is well defined because $k$ is defined up to right multiplication by~$K_H$.
However, it does not have an obvious metric interpretation as in the Riemannian case.
In fact, among the points of $\mathcal{K}_x$ which are connected to $y$ by a spacelike geodesic segment, $\pi_{\mathcal{K}_x}(y)$ \emph{maximizes} the length of this segment.
\end{remark}

%%%%%%%%%%%%%%%%%%%%%%%%%
\subsection{Compactly generated groups} \label{subsec:compact-gen}

In this section we recall that the notions of word length and quasi-isometric embedding can be defined, not only for finitely generated groups, but more generally for any compactly generated locally compact group.

We first recall that if $(X,d_X)$ and $(Y,d_Y)$ are two metric spaces, a map $f : X\to Y$ is called a \emph{quasi-isometric embedding} if there exist $c,c'>0$ such that for all $x,x'\in X$,
$$ c\,d_X(x,x') - c' \leq d_Y(f(x),f(x')) \leq c^{-1}\,d_X(x,x') + c'. $$
We say that $f$ is a \emph{quasi-isometry} if moreover $\sup_{y\in Y} d_Y(y,f(X)) < +\infty$.
Note that the composition of two quasi-isometric embeddings is still a quasi-isometric embedding, and similarly for quasi-isometries.

A locally compact topological group $\Gamma$ is \emph{compactly generated} if it is generated by a compact subset.
Given a compact generating subset~$S$, which we assume to be symmetric (\ie $S= S^{-1}$) and to contain a neighborhood of the identity element $1$ of~$\Gamma$, we can define a word length function and a left-invariant distance function on~$\Gamma$ by setting, as usual,
\[\vert \gamma \vert_S \equaldef \inf \{n\in \N \mid \gamma \in S^n\}\]
and
\[d_S(\gamma,\gamma') \equaldef \vert \gamma^{-1} \gamma' \vert_S \]
for all $\gamma,\gamma'\in\Gamma$.
(Here $S^n$ denotes the set of all possible products of $n$ elements of~$S$, with the convention $S^0 = \{ 1\}$.)

Since $S$ contains a neighborhood of~$1$, the set $S^{n-1}$ is contained in the interior $\mathrm{Int}(S^n)$ of $S^n$ for all $n\geq 1$.
Therefore, if $S'$ is another symmetric compact generating subset of~$\Gamma$, then $S'$ is contained in $\bigcup_{n\geq 1} \mathrm{Int}(S^n)$, hence in $\mathrm{Int}(S^n)$ for some $n\geq 1$ by the Borel--Lebesgue property; we deduce that $d_S(\gamma,\gamma') \leq n\,d_{S'}(\gamma,\gamma')$ for all $\gamma,\gamma'\in\Gamma$.
Similarly (assuming $S'$ also contains a neighborhood of~$1$), there exists an integer $m\geq 1$ such that $S \subset {S'}^m$, hence
\[\frac{1}{m} d_{S'}(\gamma,\gamma') \leq d_S(\gamma,\gamma') \leq n\,d_{S'}(\gamma,\gamma') \]
for all $\gamma,\gamma'\in\Gamma$.
Thus the identity map $\gamma\mapsto\gamma$ is a quasi-isometry between $(\Gamma,d_S)$ and $(\Gamma,d_{S'})$, and it makes sense to talk about a quasi-isometric embedding of~$\Gamma$ to some metric space without referring to an explicit compact generating set.

If $\Gamma$ is a discrete group, then any compact generating subset $S$ of~$\Gamma$ is finite and we recover the usual word length and distance functions associated to the Cayley graph $\mathrm{Cay}(\Gamma,S)$.
If $\Gamma$ is not discrete, then the distance function $d_S$ is not compatible with the topology of $\Gamma$ (since every nontrivial element of~$\Gamma$ is at distance at least $1$ from the identity element of~$\Gamma$).
This is not an issue, as $d_S$ is meant to capture the ``large scale geometry'' of $\Gamma$.

The \v Svarc--Milnor lemma generalizes to the setting of compactly generated groups as follows (see \eg \cite[Th.\,4.C.5]{ch16}).

\begin{fact} \label{fact:SvarcMilnor}
Let $\Gamma$ be a locally compact topological group acting properly and cocompactly, by isometries, on a proper geodesic metric space~$X$.
Then $\Gamma$ is compactly generated and for any $x\in X$, the orbital map
\begin{eqnarray*}
\Gamma & \to & X \\
\gamma & \mapsto & \gamma \cdot x
\end{eqnarray*}
is a quasi-isometry.
\end{fact}

For instance, Fact~\ref{fact:SvarcMilnor} applied to a connected Lie group $H = \Gamma = X$ gives the following.

\begin{corollary} \label{cor:QI-word-length-Riem-metric}
Let $H$ be a connected Lie group, $S$ a symmetric compact generating subset of~$H$ containing a neighborhood of the identity, and $\mathtt{g}$ a left-invariant Riemannian metric on~$H$ with associated distance function~$d_{\mathtt{g}}$.
Then the identity map $h\mapsto h$ is a quasi-isometry between $(H,d_S)$ and $(H,d_{\mathtt{g}})$.
\end{corollary}

In a connected Lie group~$H$, any closed subgroup $\Gamma$ of~$H$ is locally compact, and by Corollary~\ref{cor:QI-word-length-Riem-metric}, the natural inclusion $\Gamma\hookrightarrow H$ is a quasi-isometric embedding with respect to a left-invariant Riemannian metric on~$H$ if and only if it is a quasi-isometric embedding with respect to the metric associated with a compact generating subset of~$H$.
In this case, we say that $\Gamma$ is \emph{quasi-isometrically embedded in~$H$}.

We can also apply Fact~\ref{fact:SvarcMilnor} to $\Gamma=H'$ and $X=H$ to obtain the following.

\begin{corollary} \label{cor:cocompact-QI}
Let $H$ be a connected Lie group and $H'$ a closed subgroup of~$H$ such that $H/H'$ is compact.
Then $H'$ is compactly generated and the natural inclusion $H' \hookrightarrow H$ is a quasi-isometry.
\end{corollary}

In particular, any parabolic subgroup $H'$ of a noncompact connected real linear reductive Lie group~$H$ is quasi-isometric to~$H$.
For instance, the affine group of the real line is quasi-isometric to $\PSL(2,\R)$, hence to the hyperbolic plane.

Recall that we have fixed a noncompact connected real linear reductive Lie group~$G$, with maximal compact subgroup~$K$.
The following is an easy consequence of Fact~\ref{fact:SvarcMilnor} and of the interpretation $\Vert\mu(g)\Vert = d_{G/K}(x_0,g\cdot x_0)$ in terms of the Cartan projection $\mu : G\to\aaa^+$ from Section~\ref{subsec:vector-dist}.

\begin{lemma} \label{lem:QI-embed-mu}
For any closed subgroup $\Gamma$ of~$G$ with a compact generating subset~$S$, the following are equivalent:
\begin{itemize}
  \item the natural inclusion $\Gamma\hookrightarrow G$ is a quasi-isometric embedding;
  \item there exist $c,c'>0$ such that $\Vert\mu(\gamma)\Vert \geq c\,|\gamma|_S - c'$ for all $\gamma\in\Gamma$.
\end{itemize}
\end{lemma}

\begin{proof}
By Fact~\ref{fact:SvarcMilnor} (since $K$ is compact), the natural inclusion $\Gamma\hookrightarrow G$ is a quasi-isometric embedding if and only if the orbital map $\gamma \mapsto \gamma\cdot x_0$ from $\Gamma$ to $G/K$ is a quasi-isometric embedding.
By the triangle inequality, for any $\gamma,\gamma'\in\Gamma$ we have
\begin{align*}
d_{G/K}(x_0,\gamma\gamma'\cdot x_0) & = d_{G/K}(\gamma^{-1}\cdot x_0,\gamma'\cdot x_0)\\
& \leq d_{G/K}(\gamma^{-1}\cdot x_0,x_0) + d_{G/K}(x_0,\gamma'\cdot x_0)\\
& = d_{G/K}(x_0,\gamma\cdot x_0) + d_{G/K}(x_0,\gamma'\cdot x_0).
\end{align*}
This implies that $d_{G/K}(x_0,\gamma\cdot x_0) \leq \kappa \, |\gamma|_S$ for all $\gamma\in\Gamma$, where $\kappa \equaldef \max_{s\in S} \Vert\mu(s)\Vert$.
Therefore the natural inclusion $\Gamma\hookrightarrow G$ is a quasi-isometric embedding if and only if there exist $c,c'>0$ such that $d_{G/K}(x_0,\gamma\cdot x_0) \geq c\,|\gamma|_S - c'$ for all $\gamma\in\Gamma$.
We conclude by using the interpretation $\Vert\mu(g)\Vert = d_{G/K}(x_0,g\cdot x_0)$ of Section~\ref{subsec:vector-dist}.
\end{proof}

\begin{corollary} \label{cor:sharp-embed-explain}
Let $\Gamma$ be a closed subgroup of~$G$ with a compact generating subset~$S$, and let $H$ be another closed subgroup of~$G$.
Then $\Gamma$ is sharply embedded in~$G$ with respect to~$H$ (Definition~\ref{def:sharp-embed}) if and only if the action of $\Gamma$ on $G/H$ is sharp (Definition~\ref{def:sharp}) and $\Gamma$ is quasi-isometrically embedded in~$G$.
\end{corollary}

\begin{proof}
By Lemma~\ref{lem:QI-embed-mu}, if the action of $\Gamma$ on $G/H$ is sharp and $\Gamma$ is quasi-isometrically embedded in~$G$, then $\Gamma$ is sharply embedded in~$G$ with respect to~$H$.

Conversely, suppose that $\Gamma$ is sharply embedded in~$G$ with respect to~$H$: there exist $c,c'>0$ such that $d_{\aaa}(\mu(\gamma),\mu(H)) \geq c\,|\gamma|_S - c'$ for all $\gamma\in\Gamma$.
We saw in the proof of Lemma~\ref{lem:QI-embed-mu} that there exists $\kappa\geq 0$ such that $\Vert\mu(\gamma)\Vert = d_{G/K}(x_0,\gamma\cdot x_0) \leq \kappa \, |\gamma|_S$ for all $\gamma\in\Gamma$.
Therefore $d_{\aaa}(\mu(\gamma),\mu(H)) \geq c\kappa^{-1}\,\Vert\mu(\gamma)\Vert - c'$ for all $\gamma\in\Gamma$, and so the action of $\Gamma$ on $G/H$ is sharp.
Moreover, there exists $h\in H$ such that $d_{\aaa}(\mu(\gamma),\mu(H)) = \Vert\mu(\gamma)-\mu(h)\Vert \leq \Vert\mu(\gamma)\Vert + \Vert\mu(h)\Vert$, and so $\Gamma$ is quasi-isometrically embedded in~$G$ by Lemma~\ref{lem:QI-embed-mu}.
\end{proof}

%%%%%%%%%%%%%%%%%%%%%%%%%%%%%%%%%%%%%%%%%%%%%%%%%%%
\section{Proof of the Sharpness Theorem} \label{sec:proof-sharpness}

In this section we prove the following generalization of Theorem~\ref{thm:sharp}, where $\Gamma$ does not need to be discrete in~$G$.
We refer to Definition~\ref{def:sharp-embed} and Section~\ref{subsec:compact-gen} for the notion of sharp embeddedness.

\begin{theorem} \label{thm:sharp-general}
Let $G$ be a connected real linear reductive Lie group, and $H,\Gamma$ two closed subgroups of~$G$.
If $\Gamma$ acts properly and cocompactly on $G/H$, then $\Gamma$ is compactly generated and sharply embedded in~$G$ with respect to~$H$.
\end{theorem}

It is sufficient to prove Theorem~\ref{thm:sharp-general} for $H$ connected and stable under some Cartan involution of~$G$ (\ie $G/H$ of reductive type).
Indeed, for general~$H$, we use the fact that $\Gamma$ acts properly and cocompactly on $G/H$ if and only if $\Gamma\times H$ acts properly and cocompactly on $(G\times G)/\Diag(G)$, which is of reductive type. We can apply Theorem~\ref{thm:sharp-general} to $(G\times G,\Diag(G),\Gamma\times H)$ and conclude thanks to the following elementary fact.

\begin{lemma} \label{lem:GammaxH-sharp-embed}
Let $\Gamma$ and $H$ be two compactly generated closed subgroups of~$G$.
Then
\begin{enumerate}
  \item\label{item:GammaxH-sharp-embed-1} $\Gamma \times H$ is quasi-isometrically embedded in $G\times G$ if and only if $\Gamma$ and $H$ are both quasi-isometrically embedded in~$G$;
  \item\label{item:GammaxH-sharp-embed-2} the action of $\Gamma \times H$ on $(G\times G)/\Diag(G)$ is sharp if and only if the action of $\Gamma$ on $G/H$ is sharp.
\end{enumerate}
In particular, $\Gamma \times H$ is sharply embedded in $G\times G$ with respect to $\Diag(G)$ if and only if $H$ is quasi-isometrically embedded in~$G$ and $\Gamma$ is sharply embedded in~$G$ with respect to~$H$.
\end{lemma}

\begin{proof}
Choosing a compact generating subset $S$ of $\Gamma \times H$ of the form $S'\times S''$, we have
\[ |(\gamma,h)|_S = \max (|\gamma|_{S'}, |h|_{S''}) \]
for all $(\gamma,h) \in \Gamma\times H$, which implies~\eqref{item:GammaxH-sharp-embed-1}.
On the other hand, the limit cone $\mathcal L_{\Gamma \times H}$ of $\Gamma \times H$ in $\aaa^+ \times \aaa^+$ (see Section~\ref{subsec:lim-cone}) is the product $\mathcal L_\Gamma \times \mathcal L_H$, which avoids $\Diag(\aaa^+) = \mu(\Diag(G))$ if and only if $\mathcal L_\Gamma$ avoids $\mathcal L_H$.
This implies~\eqref{item:GammaxH-sharp-embed-2}.
We conclude using Corollary~\ref{cor:sharp-embed-explain}.
\end{proof}

We now assume that $H$ is connected and stable under some Cartan involution of~$G$, and we use the notation of Section~\ref{sec:Cartan-polar}.
Our goal is to prove Theorem~\ref{thm:sharp-general}.

%%%%%%%%%%%%%%%%%%%%%%%%%
\subsection{A spacelike gap between maximal timelike subspaces}

Let us introduce the function $\delta : G/K \times G/K \to \R_{\geq 0}$ given by
\begin{equation} \label{eqn:delta}
\delta(x,x') \equaldef d_\aaa\big( \multidist(x,x'), \mu(H)\big)
\end{equation}
for all $x,x'\in G/K$, where $\multidist : G/K \times G/K \to \aaa^+$ is the vector-valued distance function of Section~\ref{subsec:vector-dist} and
\[d_\aaa(u, \mu(H)) = \inf_{h\in H} \Vert u- \mu(h)\Vert = \min_{w\in W} d_\aaa(u, w \cdot \aaa_H) .\]

By definition \eqref{eqn:multidist} and $G$-invariance of~$\multidist$, the function $\delta$ is $G$-invariant and satisfies
\begin{equation} \label{eqn:delta-dist-mu(H)}
\delta(x_0,g\cdot x_0) = d_\aaa(\mu(g),\mu(H)).
\end{equation}
for all $g\in G$.
Moreover, $\delta(x,x') = \delta(x',x)$ for all $x,x'\in G/K$ because $\multidist(x',x) = \iota (\multidist(x,x'))$ and the opposition involution $\iota : \aaa^+ \to \aaa^+$ is an isometry for $d_{\aaa}$ that preserves $\mu(H)$.

Our goal is to prove that if $\Gamma$ acts properly and cocompactly on $G/H$, then $\Gamma$ is compactly generated and there exist $c,c'>0$ such that
$$\delta(\gamma\cdot o, \gamma'\cdot o) \geq c\,d_S(\gamma,\gamma') - c'$$
for all $\gamma,\gamma'\in\Gamma$, where $d_S : (\gamma,\gamma') \mapsto \vert \gamma^{-1} \gamma' \vert_S$ is the distance function on~$\Gamma$ associated to the choice of a compact generating subset~$S$ as in Section~\ref{subsec:compact-gen}.

\begin{remark}
If $H$ is compact, then $\mu(H) = \{0\}$ and $\delta$ is simply the distance function $d_{G/K}$.
In this case, the group $\Gamma$ is discrete and acts properly and cocompactly on $G/H$ if and only if it is a uniform lattice in $G$, and Theorem~\ref{thm:sharp} then simply states that a uniform lattice is quasi-isometrically embedded in $G$, which follows from the \v{S}varc--Milnor lemma (see Fact~\ref{fact:SvarcMilnor}).
\end{remark}

Our first key lemma is the following interpretation of $\delta$ as a \emph{spacelike gap} between maximal timelike subspaces in $G/H$.
Recall the $G$-invariant function $D : G/K \times G/H \to \R_{\geq 0}$ and the sets $\mathcal H_0 = H\cdot x_0 \subset G/K$ and $\mathcal{K}_x \subset G/H$ from Section~\ref{subsec:polar-decomp}.

\begin{lemma} \label{lem:interpretation-delta-projection}
Let $x = g\cdot x_0$ and $x'= g'\cdot x_0$ be two points in $G/K$ (where $g,g'\in G$).
Then
\[\delta(x,x') = \min_{k \in K} d_{G/K}(x,g' k\cdot \mathcal H_0) = \min_{y \in \mathcal{K}_{x'}} D(x,y) .\]
\end{lemma}

\begin{proof}
To simplify notation, we set $d_{G/K}(x,uK\cdot \mathcal H_0) \equaldef \min_{k\in K} d_{G/K}(x,uk\cdot \mathcal H_0)$ for all $u\in G$.

Let us prove the first equality.
Since $\delta$ and $d_{G/K}$ are both $G$-invariant, we may assume $g'=\1_G$, hence $x'=x_0$.
For any $h\in H$, we have
\begin{eqnarray*}
d_{G/K}(g\cdot x_0,K\cdot \mathcal H_0) & \leq & d_{G/K}(g\cdot x_0,Kh\cdot x_0) = d_{G/K}(\exp(\mu(g))\cdot x_0,K\exp(\mu(h))\cdot x_0)\\
& \leq & d_{G/K}(\exp(\mu(g))\cdot x_0,\exp(\mu(h))\cdot x_0) = d_{\aaa}(\mu(g),\mu(h)).
\end{eqnarray*}
Taking the infimum over all $h\in H$ and using \eqref{eqn:delta-dist-mu(H)}, we get \[d_{G/K}(g\cdot x_0,K\cdot \mathcal H_0) \leq d_{\aaa}(\mu(g),\mu(H)) = \delta(x, x_0) .\]
On the other hand, for any $g,g'\in G$ we have $d_{\aaa}(\mu(g),\mu(g')) \leq d_{G/K}(g\cdot x_0,g'\cdot x_0)$ (see \eqref{eqn:mu-subadd}).
In particular, for any $h\in H$ and any $k\in K$, we have $d_{\aaa}(\mu(g),\mu(h)) \leq d_{G/K}(g\cdot x_0,kh\cdot x_0)$, and so (by taking the infimum over $k\in K$ and then $h\in H$)
\[\delta(x,x_0) = d_{\aaa}(\mu(g),\mu(H)) \leq d_{G/K}(x,K\cdot\mathcal{H}_0)~.\]

We now prove the second equality.
The compact set $g' K\cdot\mathcal H_0$ is the union of all $G$-translates of~$\mathcal H_0$ passing through $x' = g'\cdot x_0$.
By \eqref{eq: Duality G/H G/K}, these are exactly the $\mathcal H_y$ with $y\in \mathcal{K}_{x'}$.
Using the Riemannian interpretation of $D(x,y)$ from Section~\ref{subsec:polar-decomp}, we deduce
\[\delta(x,x') = \min_{y\in \mathcal{K}_{x'}} d_{G/K}(x, \mathcal H_y) = \min_{y\in \mathcal{K}_{x'}} D(x,y) .\qedhere\]
\end{proof}

\begin{remark} \label{rem:D-spacelike-distance}
Using the pseudo-Riemannian interpretation of $D(x,y)$ as a ``spacelike distance'' between $y$ and $\mathcal{K}_x$ in $G/H$ (see Section~\ref{subsec:polar-decomp}), the second equality interprets $\delta(x,x')$ as the infimum of spacelike distances from a point $y\in \mathcal{K}_{x'}$ to $\mathcal{K}_x$.
In this sense, $\delta(x,x')$ measures a ``spacelike gap'' between $\mathcal{K}_x$ and $\mathcal{K}_{x'}$ in $G/H$.
In particular, we have
\[\delta(x,x') = 0 \Longleftrightarrow \mathcal{K}_x \cap \mathcal{K}_{x'} \neq \emptyset .\]
\end{remark}

%%%%%%%%%%%%%%%%%%%%%%%%%
\subsection{The function~$\Delta_{\Gamma}$}

We now pull back the function $\delta$ on $G/K\times G/K$ to a $\Gamma$-invariant function $\Delta_{\Gamma}$ on $G/H\times G/H$.
For this, we observe that the set of maximal timelike subspaces passing through a point~$x$ is contractible, and so we can choose such a space for every $x$ in a continuous and $\Gamma$-equivariant way.
Let us give more details.

We identify $G/K_H$ with the space of pairs $(x,y)\in G/K \times G/H$ such that $y\in \mathcal{K}_x$.
The projections to the first and second factor are respectively given by the natural $G$-equivariant fibrations $\pi_H : G/K_H \to G/H$ with fiber $H/K_H$, and $\pi_K : G/K_H \to G/K$ with fiber $K/K_H$.

Let $\Gamma$ be a closed subgroup of~$G$ acting freely, properly, and cocompactly on $G/H$.
The fibrations above induce fibrations $\pi_H : \Gamma\backslash G/K_H \to \Gamma\backslash G/H$ and $\pi_K : \Gamma\backslash G/K_H \to \Gamma\backslash G/K$.
Since the fibers $H/K_H$ of~$\pi_H$ are contractible, we can find a smooth section $\Gamma\backslash G/H \to \Gamma\backslash G/K_H$ of~$\pi_H$, which lifts to a smooth $\Gamma$-equivariant map $G/H \to G/K_H$.
Composing with~$\pi_K$ yields a $\Gamma$-equivariant map $\varphi : G/H \to G/K$ with the property that for any $y\in G/H$ we have $y \in \mathcal{K}_{\varphi(y)}$, or equivalently $\varphi(y) \in \mathcal H_y$.
We use $\varphi$ to define a function $\Delta_{\Gamma} : G/H \times G/H \to \R_{\geq 0}$ by
\begin{equation} \label{eqn:definition-Delta}
\Delta_{\Gamma}(y,y') \equaldef \delta(\varphi(y),\varphi(y')) .
\end{equation}
Since $\delta$ is $G$-invariant and symmetric in its arguments, the function $\Delta_{\Gamma}$ is $\Gamma$-invariant and satisfies $\Delta_{\Gamma}(y,y') = \Delta_{\Gamma}(y',y)$ for all $y,y'\in G/H$.

We will prove that the function $\Delta_{\Gamma}$ is coarsely bounded from below by any $\Gamma$-invariant distance function on $G/H$, and conclude the proof using the \v Svarc--Milnor lemma (Fact~\ref{fact:SvarcMilnor}).
The key properties of~$\Delta_{\Gamma}$ that make this approach possible are given by the following Lemmas \ref{lem:Delta-small-d} and~\ref{lem:cut-at-1}.

\begin{lemma} \label{lem:Delta-small-d}
For any compact subset $\mathcal{C}$ of $G/H$ and any $M\geq 0$, the set 
\[\{y\in G/H \mid \exists y'\in\mathcal{C},\ \Delta_{\Gamma}(y,y')\leq M\}\]
is compact.
\end{lemma}

\begin{proof}
Let $\mathcal{C}$ be a compact subset of $G/H$.
We define two closed subsets $\mathcal{C}'$ and~$\mathcal{C}''$ of $G/H$ as follows:
\[\mathcal{C}' \equaldef \bigcup_{y'\in\mathcal{C}} \{y\in G/H \mid D(\varphi(y'), y)\leq M\} \quad\mathrm{and}\quad \mathcal{C}'' = \bigcup_{y'\in \mathcal{C}} \{y\in G/H \mid \Delta_{\Gamma}(y',y)\leq M\} .\]
In order to prove the compactness of~$\mathcal{C}''$ stated by the lemma, we first prove the compactness of~$\mathcal{C}'$.\\

\noindent {\it Compactness of $\mathcal{C}'$:} Let $(y_n)_{n\in\N}$ be a sequence in $\mathcal{C}'$: for each $n\in\N$, there exists $y'_n\in\mathcal{C}$ such that $D(\varphi(y'_n), y_n)\leq M$.
By Fact~\ref{fact:polar-decomp}, we can write $y_n = k_n \exp(b_n)\cdot y_0$ for some $k_n\in K$ and $b_n \in \p\cap\q$, and we have $\Vert b_n \Vert = D(x_0,y_n)$.
Using the Riemannian interpretation of~$D$ from Section~\ref{subsec:polar-decomp} and the triangle inequality, we get 
\begin{eqnarray*}
D(x_0,y_n) & = & d_{G/K}(x_0, \mathcal H_{y_n})\\
& \leq & d_{G/K}(x_0,\varphi(y'_n)) + d_{G/K}(\varphi(y'_n), \mathcal H_{y_n})\\
& \leq & \sup_{y'\in \mathcal{C}} d_{G/K}(x_0, \varphi(y')) + M,
\end{eqnarray*}
which is finite because $\mathcal{\mathcal{C}}$ is compact and $\varphi$ is continuous.
Therefore the sequence $(\Vert b_n \Vert)_{n\in\N}$ is bounded, and so $(y_n)_{n\in\N}$ has an accumulation point in $G/H$, which still belongs to~$\mathcal{C}'$ since $\mathcal{C}'$ is closed.
This holds for any sequence $(y_n)_{n\in\N}$, hence $\mathcal{C}'$ is compact.\\

\noindent {\it Compactness of $\mathcal{C}''$:} Since the action of $\Gamma$ on $G/H$ is cocompact, up to enlarging $\mathcal{C}$, we may assume that 
\[\bigcup_{\gamma \in \Gamma} \gamma \cdot \mathcal{C} = G/H~.\]
Since $\mathcal{C}'$ is compact and the action of $\Gamma$ on $G/H$ is proper, the set $S \equaldef \{\gamma \in \Gamma \mid \gamma \cdot \mathcal{C}' \cap \mathcal{C}' \neq \emptyset\}$ is compact.
Note that $\mathcal{C}'$ contains $\bigcup_{y'\in\mathcal{C}} \mathcal{K}_{\varphi(y')}$. 

Let us check that $\mathcal{C}'' \subset \bigcup_{\gamma\in S} \gamma^{-1}\cdot\mathcal{C}$.
Consider a point $y\in\mathcal{C}''$: there exists $y'\in\mathcal{C}$ such that $\Delta_{\Gamma}(y',y) = \delta(\varphi(y'),\varphi(y)) \leq M$.
By Lemma~\ref{lem:interpretation-delta-projection}, there exists $z\in\mathcal{K}_{\varphi(y)}$ such that $D(\varphi(y'),z)\leq M$, hence $z\in\mathcal{C}'$.
Let $\gamma \in \Gamma$ be such that $\gamma\cdot y \in \mathcal{C}$.
Then 
\[\gamma\cdot z \in  \mathcal{K}_{\gamma \cdot \varphi(y)} = \mathcal{K}_{\varphi(\gamma \cdot y)} \subset \mathcal{C}'~,\]
hence $\gamma$ belongs to the compact set~$S$.
This shows that $\mathcal{C}''$ is contained in $\bigcup_{\gamma \in S} \gamma^{-1}\cdot \mathcal{C}$.
Since moreover $\mathcal{C}''$ is closed in $G/H$, we conclude that $\mathcal{C}''$ is compact.
\end{proof}

\begin{lemma} \label{lem:cut-at-1}
For any $y,y'\in G/H$ and any $t\in [0,1]$, there exists $z_t\in G/H$ such that $\Delta_{\Gamma}(y,z_t) \leq t\,\Delta_{\Gamma}(y, y')$ and $\Delta_{\Gamma}(z_t,y') \leq (1-t)\,\Delta_{\Gamma}(y,y')$.
\end{lemma}

\begin{proof}
Let us fix $y,y'\in G/H$.
We can write $\varphi(y)= g \cdot x_0$ and $\varphi(y') = g'\cdot x_0$ for some $g, g' \in G$.
By definition \eqref{eqn:delta} of~$\delta$, there exists $w\in W$ such that
\[\Delta_{\Gamma}(y,y') = \delta(\varphi(y), \varphi(y')) = d_{\aaa}\big(w\cdot\mu(g^{-1}g'),\aaa_H\big) .\]
We can thus write
\[w\cdot \mu(g^{-1} g') = a + b ,\]
where $a \in \aaa_H$ and $b\in \aaa \cap \aaa_H^\perp \subset \p \cap \q$ satisfies $\Vert b\Vert = \Delta_{\Gamma}(y,y')$.

By the Cartan decomposition (see Section~\ref{subsec:Cartan-decomp}), there exist $k,k'\in K$ such that $g^{-1} g' = k \, \exp(w\cdot\mu(g^{-1} g')) \, k'$.
For any $t\in [0,1]$, we set
\[z_t \equaldef g k \exp(w\cdot t\mu(g^{-1} g')) \cdot y_0 = gk \exp(tb)\cdot y_0 ,\]
where the right-hand inequality comes from the fact that $a$ and~$b$ belong to the abelian subspace $\aaa$ of~$\g$ and that $y_0$ is fixed by $\exp(ta)\in H$.

Note that $z_0 \in gK\cdot y_0 = \mathcal{K}_{\varphi(y)}$ and $z_1 \in g' K\cdot y_0 = \mathcal{K}_{\varphi(y')}$.
From the definition \eqref{eqn:D-x-y} of~$D$ and the expressions $\varphi(y) = g\cdot x_0$ and $z_t = gk \exp(tb)\cdot y_0$ where $b\in \p \cap \q$, for any $t\in [0,1]$ we have
\begin{equation} \label{eqn:D-phi-y-z_t}
D(\varphi(y), z_t) = t\Vert b \Vert = t \Delta_{\Gamma}(y,y') .
\end{equation}
In particular, using the definition \eqref{eqn:definition-Delta} of~$\Delta_{\Gamma}$ and Lemma~\ref{lem:interpretation-delta-projection}, we get
\[D(\varphi(y), z_1) = \delta(\varphi(y), \varphi(y')) = \min_{z \in \mathcal{K}_{y'}} D(\varphi(y), z)~.\]
In other words (see Remark~\ref{rem:D-spacelike-distance}), $z_1$ is the point of $\mathcal{K}_{y'}$ minimizing its spacelike distance to $\mathcal{K}_y$, the point $z_0$ is its projection to $\mathcal{K}_y$, and $(z_t)_{t\in [0,1]}$ is the spacelike geodesic segment between them.

Fix $t\in [0,1]$.
By definition \eqref{eqn:definition-Delta} of~$\Delta_{\Gamma}$, Lemma~\ref{lem:interpretation-delta-projection}, and \eqref{eqn:D-phi-y-z_t}, we have
\begin{equation} \label{eqn:Delta-Gamma-y-z_t}
\Delta_{\Gamma}(y,z_t) = \delta(\varphi(y), \varphi(z_t)) = \min_{z\in \mathcal{K}_{\varphi(z_t)}} D(\varphi(y),z) \leq D(\varphi(y),z_t) = t \Delta_{\Gamma}(y,y') .
\end{equation}
On the other hand, we can write
\begin{eqnarray*}
z_t &=& g k \exp(tb) \cdot y_0\\
&=& g k \exp(b+a) k' {k'}^{-1} \exp((t-1) b - a)\cdot  y_0\\
&=& g' {k'}^{-1} \exp((t-1)b) \cdot y_0~,
\end{eqnarray*}
from which we deduce that
\[D(\varphi(y'), z_t) = (1-t) \Vert b \Vert = (1-t) \Delta_{\Gamma}(y,y')~.\]
We conclude as in \eqref{eqn:Delta-Gamma-y-z_t} that $\Delta_{\Gamma}(y',z_t)\leq (1-t) \Delta_{\Gamma}(y,y')$.
\end{proof}

%%%%%%%%%%%%%%%%%%%%%%%%%
\subsection{A \v{S}varc--Milnor lemma for $\Delta_{\Gamma}$}

We can finally prove that $\Delta_{\Gamma}$ grows at least linearly with some $\Gamma$-invariant distance function on $G/H$ and conclude the proof of Theorem~\ref{thm:sharp-general} (hence of Theorem~\ref{thm:sharp}).

\begin{proposition} \label{prop:Delta-grows-linearly}
For any continuous $\Gamma$-invariant distance function $d_{G/H}$ on $G/H$, there exists $R\geq 1$ such that $\Delta_{\Gamma}(y,y') \geq R^{-1}\,d_{G/H}(y,y') - 1$ for all $y,y'\in G/H$.
\end{proposition}

\begin{proof}
Let $\mathcal{C}$ be a compact subset of $G/H$ with $\Gamma \cdot \mathcal{C} = G/H$.
By Lemma~\ref{lem:Delta-small-d}, there exists $R>0$ such that $d_{G/H}(y,y')\leq R$ for all $y\in G/H$ and $y'\in\mathcal{C}$ with $\Delta_{\Gamma}(y,y')\leq 1$.
By $\Gamma$-invariance of $d_{G/H}$ and $\Delta_{\Gamma}$, we have $d_{G/H}(y,y')\leq R$ for all $y,y'\in G/H$ with $\Delta_{\Gamma}(y,y')\leq 1$.

Consider $y,y'\in G/H$, and let $N\in\N$ be the integer part of $\Delta_{\Gamma}(y,y')$.
We set $z_0 \equaldef y$ and $z_{N+1} \equaldef y'$.
By Lemma~\ref{lem:cut-at-1}, there exists $z_1\in G/H$ such that $\Delta_{\Gamma}(z_0,z_1)\leq 1$ and $\Delta_{\Gamma}(z_1,z_{N+1}) \leq \Delta_{\Gamma}(z_0,z_{N+1})-1$.
Applying Lemma~\ref{lem:cut-at-1} again, we find, by induction, points $z_2,\dots,z_N\in G/H$ such that $\Delta_{\Gamma}(z_i,z_{i+1})\leq 1$ and $\Delta_{\Gamma}(z_{i+1},z_{N+1}) \leq \Delta_{\Gamma}(z_i,z_{N+1})-1 \leq \Delta_{\Gamma}(z_0,z_{N+1})-(i+1)$ for all $1\leq i\leq N-1$.
In particular, $\Delta_{\Gamma}(z_N,z_{N+1})\leq 1$, and so $\Delta_{\Gamma}(z_i,z_{i+1})\leq 1$ for all $0\leq i\leq N$.
We then have $d_{G/H}(z_i,z_{i+1})\leq R$ for all $0\leq i\leq N$, hence, by the triangle inequality, 
\[d_{G/H}(y,y') = d_{G/H}(z_0,z_{N+1}) \leq (N+1)R \leq (\Delta_{\Gamma}(y,y') + 1) R .\qedhere\]
\end{proof}

\begin{proof}[Proof of Theorem~\ref{thm:sharp-general}]
We may assume without loss of generality that the map $\varphi$ satisfies $\varphi(y_0) = x_0$.
By $\Gamma$-equivariance of~$\varphi$, we then have
\[\Delta_{\Gamma}(y_0,\gamma \cdot y_0) = \delta(\varphi(y_0), \varphi(\gamma\cdot y_0)) = \delta(x_0, \gamma \cdot x_0) = d_\aaa(\mu(\gamma), \mu(H))\]
for all $\gamma \in \Gamma$.
(If we do not assume $\varphi(y_0) = x_0$, then the difference between the first and last term is uniformy bounded by $2 d_{G/K}(x_0, \varphi(y_0))$.)

Let us choose an arbitrary Riemannian metric on $\Gamma \backslash G/H$ and lift it to a $\Gamma$-invariant Riemannian metric on $G/H$, with associated distance function $d_{G/H}$.
This turns $G/H$ into a proper geodesic metric space on which $\Gamma$ acts properly and cocompactly, by isometries.
By Proposition~\ref{prop:Delta-grows-linearly}, there exists $R_1\geq 1$ such that
\[\Delta_{\Gamma}(y_0, \gamma \cdot y_0) \geq R_1^{-1} d_{G/H}(y_0, \gamma\cdot y_0) - R_1~\]
for all $\gamma\in\Gamma$.
On the other hand, by Fact~\ref{fact:SvarcMilnor}, the group $\Gamma$ is compactly generated and the orbital map $\gamma\mapsto\gamma \cdot y_0$ is a quasi-isometry from $\Gamma$ to $(G/H, d_{G/H})$.
In particular, there exists $R_2\geq 1$ such that $d_{G/H}(y_0, \gamma\cdot y_0)\geq R_2^{-1}\,|\gamma|_S-R_2$ for all $\gamma\in\Gamma$, where $S$ is any fixed compact generating subset of~$\Gamma$.
Combining the two inequalities gives
\[d_{\aaa}(\mu(\gamma), \mu(H)) \geq R_1^{-1} R_2^{-1} |\gamma|_S - (R_1^{-1} R_2 + R_1) \]
for all $\gamma\in\Gamma$, and so $\Gamma$ is sharply embedded in~$G$ with respect to~$H$.
\end{proof}

%%%%%%%%%%%%%%%%%%%%%%%%%%%%%%%%%%%%%%%%%%%%%%%%%%%
\section{Sharpness implies Anosov} \label{sec:sharp-Ano}

In this section we prove Theorem~\ref{thm:proper-compact<->Ano} and Proposition~\ref{prop:cork-1-proper-compact->Ano}.
We use the notation of Section~\ref{sec:Cartan-polar}. We start with some reminders.

%%%%%%%%%%%%%%%%%%%%%%%%%
\subsection{Reminders: limit cones} \label{subsec:lim-cone}

As above, let $G$ be a connected real linear reductive Lie group, with Cartan projection $\mu : G\to\aaa^+$ as in Section~\ref{subsec:Cartan-decomp}.
Given a subgroup $\Gamma$ of~$G$, the \emph{limit cone} of $\Gamma$ is defined as
\[\mathcal L_\Gamma \equaldef \bigcap_{n\in \N}\ \overline{\bigcup_{\gamma \in \Gamma,\ \Vert \mu(\gamma)\Vert \geq n} \R_{\geq 0}\,\mu(\gamma)}.\]
It is a closed subcone of~$\aaa^+$, invariant under the opposition involution~$\iota$.
It is nonempty as soon as $\Gamma$ is unbounded in~$G$.
When $\Gamma$ is Zariski-dense in~$G$, Benoist \cite{ben97} proved that $\Lambda_\Gamma$ is convex with nonempty interior.

The Jordan (or Lyapunov) projection $\lambda: G \to \aaa^+$ of~$G$ can be defined as
\begin{equation} \label{eqn:lambda}
\lambda(g) \equaldef \lim_{n\to +\infty} \frac{1}{n}\mu(g^n)
\end{equation}
for all $g\in G$.
It is a continuous function which is invariant under conjugation.
Given a subgroup $\Gamma$ of~$G$, the \emph{Jordan limit cone} of $\Gamma$ is defined as
\[\mathcal L_\Gamma^\lambda \equaldef \overline{\bigcup_{\gamma \in \Gamma} \R_{\geq 0}\,\lambda(\gamma)} .\]
It is also a closed subcone of~$\aaa^+$, invariant under the opposition involution~$\iota$.
The inclusion $\mathcal L_\Gamma^\lambda \subset \mathcal L_\Gamma$ always holds.

The following fact is a a consequence of a theorem of Abels--Margulis--Soifer \cite{ams95}; point~\eqref{item:AMS-1} is due to Benoist \cite{ben97} (see \cite[Th.\,4.12]{ggkw17} for an explicit statement and proof) and point~\eqref{item:AMS-2} to Quint \cite[\S\,II.3.5]{qui02}.
Point~\eqref{item:AMS-1} implies that $\mathcal L_\Gamma^\lambda = \mathcal L_\Gamma$ whenever the Zariski closure of $\Gamma$ in~$G$ is reductive (see \cite[Th.\,4.12]{ggkw17}).

\begin{fact} \label{fact:AMS}
Let $G$ be a connected real linear reductive Lie group and $\Gamma$ a Zariski-dense subgroup of~$G$.
Then there exists a finite subset $F$ of~$\Gamma$ and a constant $C>0$ such that:
\begin{enumerate}
  \item\label{item:AMS-1} for every $\gamma\in\Gamma$, there exists $f\in F$ such that
\[\Vert \lambda (\gamma f) - \mu(\gamma)\Vert \leq C ,\]
  \item\label{item:AMS-2} for every $\gamma,\gamma'\in\Gamma$, there exists $f\in F$ such that
\[\Vert \mu(\gamma f\gamma') - \mu(\gamma) -\mu(\gamma') \Vert \leq C .\]
\end{enumerate}
\end{fact}

\begin{remark} \label{rem:sharp-lim-cone}
The sharpness condition (Definition~\ref{def:sharp}) can be reformulated in terms of limit cones: given a discrete subgroup $\Gamma$ and a closed subgroup $H$ of~$G$, the action of $\Gamma$ on $G/H$ is sharp if and only if
\[\mathcal L_\Gamma \cap \mathcal L_H = \{0\} .\]
\end{remark}

%%%%%%%%%%%%%%%%%%%%%%%%%
\subsection{Background on Anosov representations} \label{subsec:remind-Ano}

As mentioned in Section~\ref{subsec:intro-Ano}, Anosov representations are representations of infinite word hyperbolic groups $\Gamma$ to noncompact real reductive Lie groups, which have finite kernel and discrete image, and which are defined by strong dynamical properties.
These dynamical properties are very similar to those satisfied by convex cocompact representations of word hyperbolic groups into simple Lie groups of real rank one (see \eg \cite{lab06,gw12,klp-survey,ggkw17,bps19}).
There is in fact a geometric way to see Anosov representations as convex cocompact representations, in the world of convex projective geometry \cite{dgk-Hpq-cc,dgk-proj-cc,zim}.

Given a reductive Lie group~$G$, there are several possible types of Anosov representations into~$G$, depending on the choice of a proper parabolic subgroup $P$ of~$G$ up to conjugation (or equivalently on the choice of a flag variety $G/P$ of~$G$).
Fix such a parabolic subgroup $P$ of~$G$, and let $P^*$ be an opposite parabolic subgroup.
Recall that there are several possible transversality positions between a point $x$ of $G/P$ and a point $y$ of $G/P^*$; we shall say that $x$ and~$y$ are \emph{compatible} (\resp \emph{transverse}) if they correspond to parabolic subgroups of~$G$ whose intersection is parabolic (\resp reductive).

\begin{example} \label{ex:Grass}
For $\mathbb{K}=\R$, $\C$, or the ring $\mathbb{H}$ of quaternions, and for integers $d,i$ with $1\leq i\leq d/2$, let $G = \SL(d,\mathbb{K})$ and let $P = P_i$ be the stabilizer in~$G$ of an $i$-dimensional linear subspace of $\mathbb{K}^d$.
Then $G/P = \mathrm{Gr}_i(\mathbb{K}^d)$ is the Grassmannian of $i$-planes of~$\mathbb{K}^d$, and $G/P^* = \mathrm{Gr}_{d-i}(\mathbb{K}^d)$ is the Grassmannian of $(d-i)$-planes of~$\mathbb{K}^d$.
Two points $x \in G/P$ and $y \in G/P^*$ are compatible (\resp transverse) if the $i$-plane~$x$ is contained in (\resp is transverse to) the $(d-i)$-plane~$y$.
\end{example}

\begin{remark} \label{rem:comp-trans-self-opposite}
If $P$ and~$P^*$ are conjugate in~$G$, then $G/P \simeq G/P^*$, and it makes sense to say that two points of $G/P$ are \emph{compatible} (which means equal) or \emph{transverse}.
\end{remark}

The original definition of Anosov representations is due to Labourie \cite{lab06}.
Let $\Gamma$ be an infinite word hyperbolic group with Gromov boundary $\partial_{\infty}\Gamma$.
A representation $\rho : \Gamma\to G$ is \emph{$P$-Anosov} if there exists two continuous, $\rho$-equivariant maps (called \emph{boundary maps}) $\xi : \partial_{\infty}\Gamma\to G/P$ and $\xi^* : \partial_{\infty}\Gamma\to G/P^*$ which
\begin{enumerate}
  \item\label{item:def-Ano-compat} are \emph{compatible} (\ie $\xi(\eta)$ and $\xi^*(\eta)$ are compatible for all $\eta$ in $\partial_{\infty}\Gamma$),
  \item\label{item:def-Ano-transv} are \emph{transverse} (\ie $\xi(\eta)$ and $\xi^*(\eta')$ are transverse for all $\eta\neq\eta'$ in $\partial_{\infty}\Gamma$),
  \item\label{item:flow} have an associated flow with some uniform contraction/expansion properties described in \cite{lab06,gw12}.
\end{enumerate}
We do not state condition~\eqref{item:flow} precisely, but refer instead to the original papers \cite{lab06,gw12}, and to the survey \cite{kas-notes} for various characterizations established by several authors.

In the current paper, we shall use the following characterization due to Kapovich, Leeb, and Porti \cite{klp-morse}, for which an alternative proof was later given by Bochi, Potrie, and Sambarino \cite{bps19}.

\begin{fact}[\cite{klp-morse}] \label{fact:klp}
Let $G$ be a noncompact real reductive Lie group, $\theta \subset \Delta$ a nonempty subset of the simple restricted roots of~$G$, and $\Gamma$ an infinite group with finite generating subset~$S$.
For any representation $\rho : \Gamma\to G$, the following are equivalent:
\begin{enumerate}
  \item\label{item:klp-1} $\Gamma$ is word hyperbolic and $\rho : \Gamma\to G$ is $P_{\theta}$-Anosov,
  \item\label{item:klp-2} there exist $c,c'>0$ such that $\langle\alpha,\mu(\rho(\gamma))\rangle \geq c\,|\gamma|_S - c'$ for all $\gamma \in \Gamma$ and all $\alpha \in \theta$.
\end{enumerate}
\end{fact}

Here $P_{\theta}$ is the standard parabolic subgroup of~$G$ with Lie algebra
\begin{equation} \label{eqn:P-theta}
\g_0 \oplus \bigoplus_{\alpha\in\Sigma^+} \g_{\alpha} \oplus \bigoplus_{\alpha\in\Sigma^+\cap\mathrm{span}(\Delta\smallsetminus\theta)} \g_{-\alpha},
\end{equation}
where for each $\alpha\in\aaa^*$ we set
$$\g_{\alpha} \equaldef \{ v\in\g \mid \ad(a)\,v = \alpha(a)\,v \quad\forall a\in\aaa\}.$$
We denote by $|\cdot|_S : \Gamma\to\N$ the word length function with respect to~$S$ as in Section~\ref{subsec:compact-gen}.

In condition~\eqref{item:klp-2}, the constants $c,c'>0$ depend on the choice of the finite generating subset~$S$, but not their existence.
Condition~\eqref{item:klp-2} is equivalent to asking that $\rho$ be a quasi-isometric embedding and that the limit cone $\mathcal L_{\rho(\Gamma)}$ not meet the walls $\mathrm{Ker}(\alpha)$, $\alpha\in\theta$, outside of~$0$.

\begin{remark} \label{rem:Anosov-theta-iota}
A representation $\rho : \Gamma\to G$ is $P_{\theta}$-Anosov if and only if it is $P_{\iota(\theta)}$-Anosov, if and only if it is $P_{\theta\cup\iota(\theta)}$-Anosov.
\end{remark}

%%%%%%%%%%%%%%%%%%%%%%%%%
\subsection{Sharpness and Anosov representations when $\mu(H)$ is a union of walls}

Theorem~\ref{thm:proper-compact<->Ano} is an immediate consequence of Theorem~\ref{thm:sharp}, Fact~\ref{fact:vcd}, and of the following fact and proposition.
(See Section~\ref{subsec:proof-vcd} below for the notion of virtual cohomological dimension.)

\begin{fact}[Bestvina--Mess, \cite{bm91}] \label{fact:dim-bound-vcd}
Let $\Gamma$ be a word hyperbolic group.
Then the covering dimension of the Gromov boundary of~$\Gamma$ is equal to the virtual cohomological dimension of~$\Gamma$ minus~$1$.
\end{fact}

\begin{proposition} \label{prop:QI-sharp<->Ano}
Let $G/H$ be a homogeneous space of reductive type such that $\mu(H) = \aaa^+ \cap \bigcup_{\alpha\in\theta} \mathrm{Ker}(\alpha)$ for some subset $\theta\subset\Delta$ of the simple roots of~$G$.
Then for any finitely generated discrete subgroup $\Gamma$ of~$G$, the following conditions are equivalent:
\begin{enumerate}
  \item\label{item:QI-sharp} $\Gamma$ is sharply embedded in~$G$ with respect to~$H$ (Definition~\ref{def:sharp-embed}),
  \item\label{item:Ano} $\Gamma$ is word hyperbolic and the natural inclusion $\Gamma \hookrightarrow G$ is a $P_{\theta}$-Anosov representation.
\end{enumerate}
In this case, there is a neighborhood $\mathcal{U} \subset \Hom(\Gamma,G)$ of the natural inclusion consisting entirely of discrete and faithful representations $\rho$ for which $\rho(\Gamma)$ is sharply embedded in~$G$ with respect to~$H$.
\end{proposition}

\begin{proof}[Proof of Proposition~\ref{prop:QI-sharp<->Ano}]
Let $S$ be a finite generating subset of~$\Gamma$.

If $\Gamma$ is sharply embedded in~$G$ with respect to~$H$, then there exist $c,c'>0$ such that $\langle\alpha,\mu(\gamma)\rangle \geq c\,|\gamma|_S - c'$ for all $\alpha\in\theta$ and $\gamma\in\Gamma$; therefore $\Gamma$ is word hyperbolic and the natural inclusion is $P_{\theta}$-Anosov by Fact~\ref{fact:klp}.

Conversely (see \cite[Cor.\,1.9]{ggkw17}), if $\Gamma$ is word hyperbolic and the natural inclusion $\Gamma \hookrightarrow G$ is $P_{\theta}$-Anosov, then $\Gamma$ is quasi-isometrically embedded in~$G$ and the limit cone $\mathcal{L}_{\Gamma} \subset \aaa^+$ of~$\Gamma$ does not meet the walls $\mathrm{Ker}(\alpha)$, $\alpha\in\theta$, outside of~$0$; therefore, $\Gamma$ is sharply embedded in~$G$ with respect to~$H$ by Remark~\ref{rem:sharp-lim-cone}.

Suppose that $\Gamma$ is word hyperbolic and the natural inclusion $\Gamma \hookrightarrow G$ is a $P_{\theta}$-Anosov representation.
Since being Anosov is an open property \cite{lab06}, there is a neighborhood $\mathcal{U} \subset \Hom(\Gamma,G)$ of the natural inclusion such that each $\rho\in\mathcal{U}$ is a $P_{\theta}$-Anosov representation.
In particular, each $\rho\in\mathcal{U}$ has finite kernel and $\rho(\Gamma)$ is sharply embedded in~$G$ with respect to~$H$.
Since $\Gamma$ is word hyperbolic, it has only finitely many finite normal subgroups~$F$ (see \eg \cite[Th.\,III.$\Gamma$.3.2]{bh99}), and for each such~$F$, the set of representations whose kernel contains~$F$ is closed in $\Hom(\Gamma,G)$.
Therefore the set of faithful representations is open in~$\mathcal{U}$, and in particular in contains a neighborhood of the natural inclusion.
\end{proof}

Proposition~\ref{prop:QI-sharp<->Ano} applies in particular to the triples $(G,H,\theta)$ in Table~\ref{table1}.
For these triples, compactifications of locally homogeneous spaces $\Gamma\backslash G/H$ were described in \cite{ggkw17bis} when $\Gamma$ is word hyperbolic and $\Gamma\hookrightarrow G$ is $P_{\theta}$-Anosov.
In the case that $\Gamma\backslash G/H$ is already compact, using the Sharpness Theorem~\ref{thm:sharp} and Proposition~\ref{prop:QI-sharp<->Ano}, we obtain the following.

\begin{corollary} \label{cor:foliation}
Let $G/H$ be as in cases \textup{(i)}, \textup{(ii)}, \textup{(iii)} (\resp \textup{(iv)}, \textup{(v)}) of Table~\ref{table1}, and let $Q=P_{\{\alpha_1\}}$ (\resp $Q=P_{\{\alpha_2\}}$).
Then for any discrete subgroup $\Gamma$ of~$G$ acting properly discontinuously and cocompactly on $G/H$, there exists a continuous $\Gamma$-equivariant fibration of $G/Q\to \partial_\infty \Gamma$ whose fibers are maximal isotropic subspaces (\resp subsets of the form
$$\mathcal{L}_x = \{ y\in G/P_{\{\alpha_2\}}\subset\mathrm{Gr}_2(\mathbb{K}^{2p,2}) \,|\, x\in y\}$$
for $x\in G/P_{\{\alpha_1\}}$).
\end{corollary}

\begin{proof}
Let $\theta = \{\alpha_i\}$ be as in Table~\ref{table1}, where $i=q$ in cases \textup{(i)}, \textup{(ii)}, \textup{(iii)} and $i=1$ in cases \textup{(iv)}, \textup{(v)}.
Then $P_{\theta}$ is conjugate to any opposite parabolic subgroup~$P_{\theta}^*$, hence $G/P_{\theta} \simeq G/P_{\theta}^*$ (see Remark~\ref{rem:comp-trans-self-opposite}).
In each case we have $\mu(H) = \aaa^+ \cap \mathrm{Ker}(\alpha_i)$.

Let $\Gamma$ be a discrete subgroup of~$G$ acting properly discontinuously and cocompactly on $G/H$.
By Theorem~\ref{thm:sharp}, the group $\Gamma$ is sharply embedded in~$G$ with respect to~$H$.
By Proposition~\ref{prop:QI-sharp<->Ano}, the group $\Gamma$ is word hyperbolic and the natural inclusion $\Gamma\hookrightarrow G$ is $P_{\theta}$-Anosov.
In particular, there is a continuous $\Gamma$-equivariant boundary map $\xi : \partial_{\infty}\Gamma \to G/P_{\theta}$ such that $\xi(\eta)\in G/P_{\theta}$ and $\xi(\eta')\in G/P_{\theta}\simeq G/P_{\theta}^*$ are transverse for all $\eta\neq\eta'$ in $\partial_{\infty}\Gamma$.

For any $\eta\in\partial_{\infty}\Gamma$, let $\mathcal{L}_{\xi(\eta)}\subset G/Q$ be the set of points in $G/Q$ which are \emph{not} transverse to $\xi(\eta)$.
In cases (i), (ii), (iii), the points in $G/Q$ are isotropic lines, the points in $G/P_{\theta}$ are maximal isotropic subspaces, and so $\mathcal{L}_{\eta}\subset G/Q\subset\PP(V)$ is the projectivization of a maximal isotropic subspace.
In cases (iv), (v), the points in $G/Q$ are isotropic $2$-planes, the points in $G/P_{\theta}$ are isotropic lines, and $\mathcal{L}_{\xi(\eta)}$ is the set of points in $G/Q\subset\mathrm{Gr}_2(\mathbb{K}^{2p,2})$ which contain $\xi(\eta)$.

By \cite{ggkw17bis}, we can compactify $\Gamma\backslash G/H$ by adding, at infinity, the set
$$\Gamma \backslash \bigg(G/Q \smallsetminus \bigcup_{\eta\in\partial\Gamma} \mathcal{L}_{\xi(\eta)}\bigg) .$$
Since $\Gamma\backslash G/H$ is already compact, we obtain
\begin{equation} \label{eqn:fib-G/Q}
\bigcup_{\eta\in\partial_{\infty}\Gamma} \mathcal{L}_{\xi(\eta)} = G/Q.
\end{equation}
This union is disjoint because the Anosov boundary map $\xi : \partial_{\infty}\Gamma \to G/P_{\theta}$ is transverse.
Moreover, $\mathcal{L}_{\xi(\gamma\cdot\eta)} = \gamma\cdot\mathcal{L}_{\xi(\eta)}$ for all $\gamma\in\Gamma$ and $\eta\in\partial_{\infty}\Gamma$ since $\xi$ is $\Gamma$-equivariant.
\end{proof}

%%%%%%%%%%%%%%%%%%%%%%%%%
\subsection{Sharpness implies Anosov in the general corank-one case} \label{subsec:corank-1-sharp->Ano}

Proposition~\ref{prop:cork-1-proper-compact->Ano} is a consequence of Theorem~\ref{thm:sharp}, Fact~\ref{fact:vcd}, Fact~\ref{fact:dim-bound-vcd}, and of the following proposition.

\begin{proposition} \label{prop:cork-1-sharp->Ano}
Let $G$ be a connected real linear reductive Lie group and $H$ a closed subgroup of~$G$ such that $\mu(H)$ is the intersection of~$\aaa^+$ with the $W$-orbit of a hyperplane in~$\aaa$.
Let $\Gamma$ be a finitely generated discrete subgroup of~$G$ which is sharply embedded in~$G$ with respect to~$H$.
Then $\Gamma$ is word hyperbolic.
If $\Gamma$ is not virtually cyclic, then $\h \not\supset \aaa\cap [\g,\g]$ and the natural inclusion $\Gamma \hookrightarrow G$ is a $P$-Anosov representation, for some proper parabolic subgroup $P$ of~$G$.
\end{proposition}

\begin{remark}
The assumption that $\Gamma$ be not virtually cyclic is necessary for Anosovness in Proposition~\ref{prop:cork-1-sharp->Ano}: for instance, if $G/H = \SL(3,\R)/\SL(2,\R)$ and $\Gamma$ is the cyclic group generated by $\mathrm{diag}(2,2,1/4)$, or if $G/H = \GL(2,\R)/\SL(2,\R)$ and $\Gamma$ is the cyclic group generated by $\mathrm{diag}(2,2)$, then $\Gamma$ is sharply embedded in~$G$ with respect to~$H$ but it is not Anosov in~$G$.
However, even for virtually cyclic groups~$\Gamma$, the fact that $\Gamma$ is sharply embedded in~$G$ with respect to~$H$ still implies that it satisfies some Anosov property \emph{after composing with some linear representation of~$G$}: see Corollary~\ref{cor:corank-1-sharp-qi-open} below.
\end{remark}

\begin{remark} \label{rem:H-varphi}
Let $\varphi \in \aaa^*$ and let $H_{\varphi} \equaldef \exp(\mathrm{Ker}(\varphi)) \subset A$.
By \eqref{eqn:mu-H}, for any closed subgroup $\Gamma$ of~$G$ with a compact generating subset~$S$, the following are equivalent:
\begin{enumerate}[(i)]
  \item $\Gamma$ is sharply embedded in~$G$ with respect to~$H_{\varphi}$;
  \item there exist $c,c'>0$ such that $d_{\aaa}(\mu(\gamma),\bigcup_{w\in W} \mathrm{Ker}(w\cdot\varphi)) \geq c\,|\gamma|_S - c'$.
\end{enumerate}
After this work was completed, we learnt that for $G$ semisimple and $\Gamma$ discrete in~$G$, condition~(ii) was also considered by Davalo--Riestenberg \cite{dr}, who named it \emph{$\varphi$-undistortion} and also independently proved \cite[Prop.\,3.3]{dr} that it implies Anosovness in~$G$.
\end{remark}

In order to prove Proposition~\ref{prop:cork-1-sharp->Ano}, we consider the $W$-invariant scalar product $(\cdot,\cdot)_*$ on~$\aaa^*$ associated with the restriction to~$\aaa$ of the $G$-invariant bilinear form $(\cdot,\cdot)$ on~$\g$ from Section~\ref{subsec:vector-dist}.
We have an orthogonal direct sum $\aaa = \aaa_z \oplus \aaa_s$ for $(\cdot,\cdot)$, where $\aaa_z \equaldef \aaa \cap \mathfrak{z}(\g)$ and $\aaa_s \equaldef \aaa \cap [\g,\g]$, and correspondingly an orthogonal direct sum $\aaa^* = \aaa_s^0 \oplus \aaa_z^0$ for $(\cdot,\cdot)_*$, where $\aaa_s^0$ (\resp $\aaa_z^0$) is the set of linear forms on~$\aaa$ that vanish identically on $\aaa_s$ (\resp $\aaa_z$).
(When $G$ is semisimple, we have $\aaa = \aaa_s$ and $\aaa^* = \aaa_s^0$.)
For each simple restricted root $\alpha\in\Delta$, let $\omega_{\alpha}\in\aaa^*$ be the fundamental weight associated with~$\alpha$, defined by 
\begin{equation}\label{eqn:omega-alpha}
2\frac{( \omega_{\alpha}, \beta)_*}{(\beta, \beta)_*} = \delta_{\alpha, \beta} \quad \text{for all} \
\beta \in \Delta ,
\end{equation}
where $\delta_{\cdot, \cdot}$ is the Kronecker symbol.
Then $\{\omega_{\alpha} \,|\, \alpha\in\Delta\}$ is a basis of~$\aaa_z^0$.
We also denote by $s_{\alpha} \in W$ the reflection associated with~$\alpha$, which acts on~$\aaa^*$ by
\begin{equation} \label{eqn:s-alpha}
s_{\alpha}\cdot\varphi \equaldef \varphi - 2 \frac{(\varphi,\alpha)_*}{(\alpha,\alpha)_*} \, \alpha.
\end{equation}

\begin{lemma} \label{lem:C-Ker-alpha}
Let $\varphi\in\aaa^*\smallsetminus\aaa_s^0$.
Let $\mathcal{C}$ be a connected component of $\aaa^+ \smallsetminus \bigcup_{w\in W} \mathrm{Ker}(w\cdot\varphi)$ on which $\varphi>0$, and such that $\overline{\mathcal{C}}\cap\mathrm{Ker}(\varphi) \neq \{ 0\}$.
Then there exists $\alpha\in\Delta$ such that $(\varphi,\alpha)_*>0$, and for any such~$\alpha$, we have $\mathcal{C} \cap \mathrm{Ker}(\alpha) = \emptyset$.
\end{lemma}

\begin{proof}
We can write $\varphi = \psi + \sum_{\alpha\in\Delta}\, t_{\alpha}\,\omega_{\alpha}$ where $\psi\in\aaa_s^0$ and $t_{\alpha} = (\varphi,\alpha)_*/(\alpha,\alpha)_* \in \R$.
The Weyl group $W$ acts trivially on~$\aaa_z$, hence $\mathcal{C} \cap \aaa_s \neq \emptyset$ and $\mathcal{C} = (\mathcal{C}\cap \aaa_s) + \aaa_z$.
Therefore there exists $v\in\mathcal{C}\cap\aaa_s$ such that $\langle\varphi,v\rangle > 0$.
Since $\omega_{\alpha}\geq 0$ on~$\aaa^+$ for all~$\alpha$, this implies that there exists $\alpha\in\Delta$ such that $(\varphi,\alpha)_*>0$.

Let $\alpha\in\Delta$ satisfy $(\varphi,\alpha)_*>0$.
We then have $s_{\alpha}\cdot\varphi<0$ on $\mathrm{Ker}(\varphi)$ (see \eqref{eqn:s-alpha}).
Since $\overline{\mathcal{C}}\cap\mathrm{Ker}(\varphi) \neq \{ 0\}$, there exists $v'\in\mathcal{C}$ such that $(s_{\alpha}\cdot\varphi)(v')<0$.
Since $\mathcal{C}$ is connected and misses $\mathrm{Ker}(s_{\alpha}\cdot\varphi)$, we have $s_{\alpha}\cdot\varphi<0$ everywhere on~$\mathcal{C}$.
Thus $\varphi - s_{\alpha}\cdot\varphi > 0$ on~$\mathcal{C}$.
But $\varphi - s_{\alpha}\cdot\varphi$ is a positive multiple of~$\alpha$ (see \eqref{eqn:s-alpha}), hence $\alpha>0$ on~$\mathcal{C}$.
\end{proof}

We also use the following result of the first-named author in \cite{kas08}.
Recall that $\iota : \aaa\to\aaa$ denotes the opposition involution.

\begin{fact}[{\cite[Th.\,1.1]{kas08}}] \label{fact:kas08}
Let $G$ be a connected real linear reductive Lie group and $H$ a closed subgroup of~$G$ such that $\mu(H)$ is the intersection of~$\aaa^+$ with finitely many hyperplanes of~$\aaa$.
Let $\Gamma$ be a discrete subgroup of~$G$ acting properly discontinuously on $G/H$.
Then there is a connected component $\mathcal{C}$ of $\aaa^+ \smallsetminus \mu(H)$ such that $\mu(\gamma) \in \mathcal{C} \cup \iota(\mathcal{C})$ for all but finitely many $\gamma\in\Gamma$.
Moreover, if $\Gamma$ is \emph{not} virtually cyclic, then $\mathcal{C} = \iota(\mathcal{C})$.
\end{fact}

\begin{proof}[Proof of Proposition~\ref{prop:cork-1-sharp->Ano}]
By assumption, there exists $\varphi\in\aaa^*$ such that $\mu(H) = \aaa^+ \cap \bigcup_{w\in W} \mathrm{Ker}(w\cdot\varphi)$.

We first treat the case that $\varphi\notin\aaa_s^0$.
By Corollary~\ref{cor:sharp-embed-explain}, the group $\Gamma$ is quasi-isometrically embedded in~$G$ and the action of $\Gamma$ on $G/H$ is sharp, hence properly discontinuous.
If $\Gamma$ is virtually cyclic, then it is word hyperbolic.
Suppose $\Gamma$ is \emph{not} virtually cyclic.
By Fact~\ref{fact:kas08}, there is a connected component $\mathcal{C}$ of $\aaa^+\smallsetminus\mu(H)$ such that $\mu(\gamma) \in \mathcal{C}$ for all but finitely many $\gamma\in\Gamma$.
Since the action of $\Gamma$ on $G/H$ is sharp, the limit cone $\mathcal{L}_{\Gamma}$ is then a closed subcone of~$\aaa^+$ contained in $\mathcal{C}\cup\{0\}$ (see Remark~\ref{rem:sharp-lim-cone}), and so there exist $c,c'>0$ such that $d_{\aaa}(\mu(\gamma),\aaa^+\smallsetminus\mathcal{C}) \geq c\,\Vert\mu(\gamma)\Vert - c'$ for all $\gamma\in\Gamma$.
By Lemma~\ref{lem:C-Ker-alpha}, there exists $\alpha\in\Delta$ such that $\mathcal{C} \cap \mathrm{Ker}(\alpha) = \emptyset$.
Then, for all $\gamma \in \Gamma$, we have
\[\frac{\langle \alpha, \mu(\gamma)\rangle}{\Vert \alpha \Vert} = d_{\aaa}(\mu(\gamma),\mathrm{Ker}(\alpha)) \geq c\,\Vert\mu(\gamma)\Vert - c' .\]
Since $\Gamma$ is quasi-isometrically embedded in~$G$, Lemma~\ref{lem:QI-embed-mu} implies the existence of\linebreak $c'',c'''>0$ such that $\langle\alpha,\mu(\gamma)\rangle \geq c''\,|\gamma|_S - c'''$ for all $\gamma\in\Gamma$, where $S$ is any fixed finite generating subset of~$\Gamma$.
Therefore $\Gamma$ is word hyperbolic and the natural inclusion is $P_{\{\alpha\}}$-Anosov by Fact~\ref{fact:klp}.

We now treat the case that $\varphi\in\aaa_s^0$; then $\mu(H) = \aaa'_z \oplus \aaa_s$ for some hyperplane $\aaa'_z$ of~$\aaa_z$.
Let us show that $\Gamma$ is virtually cyclic in this case.
Since $G$ is a connected real linear reductive Lie group, the group homomorphism
$$\pi : \exp(\aaa_z) \times (K\cap Z(G)) \times [G,G] \longrightarrow G$$
given by multiplication is surjective with finite kernel.
The group $\Gamma' \equaldef \pi^{-1}(\Gamma)$ is still finitely generated; let $S'$ be a finite generating subset.
The projection $\mathrm{pr}_{\exp(\aaa_z)} : \Gamma'\to\exp(\aaa_z)$ to the first factor defines a group homomorphism
$$\log\circ\mathrm{pr}_{\exp(\aaa_z)} : \Gamma' \longrightarrow \aaa_z$$
with abelian image.
Since $\Gamma$ is sharply embedded in~$G$ with respect to~$H$, there exist $c,c'>0$ such that $d_{\aaa}(\log\circ\mathrm{pr}_{\exp(\aaa_z)}(\gamma'),\aaa'_z) \geq c\,|\gamma'|_{S'} - c'$ for all $\gamma'\in\Gamma'$.
This implies that the kernel of $\log\circ\mathrm{pr}_{\exp(\aaa_z)}$ is a finite subgroup of~$\Gamma'$, and that the image of $\log\circ\mathrm{pr}_{\exp(\aaa_z)}$ is a discrete subgroup of~$\aaa_Z$.
Since $\aaa'_Z$ is a hyperplane of~$\aaa_Z$, this discrete subgroup must be contained in a line of~$\aaa_z$, hence must be isomorphic to~$\Z$.
We deduce that $\Gamma'$ is virtually cyclic, and so is $\Gamma$.
\end{proof}

\begin{proof}[Proof of Proposition~\ref{prop:cork-1-proper-compact->Ano}]
Suppose $\h \supset \aaa\cap [\g,\g]$.
Then by Proposition~\ref{prop:cork-1-sharp->Ano}, the group $\Gamma$ is virtually cyclic, hence word hyperbolic.

Suppose $\h \not\supset \aaa\cap [\g,\g]$.
Then $H$ does not contain a maximal unipotent subgroup of~$G$; since the action of $\Gamma$ on $G/H$ is properly discontinuous and cocompact, the group $\Gamma$ cannot be virtually cyclic (see \eg \cite[Cor.\,4.1]{ben96}).
By Theorem~\ref{thm:sharp}, the group $\Gamma$ is sharply embedded in~$G$ with respect to~$H$.
By Proposition~\ref{prop:cork-1-sharp->Ano}, the group $\Gamma$ is word hyperbolic and the natural inclusion $\Gamma \hookrightarrow G$ is a $P$-Anosov representation, for some proper parabolic subgroup $P$ of~$G$.

In both cases, by Facts \ref{fact:vcd} and~\ref{fact:dim-bound-vcd}, the Gromov boundary of~$\Gamma$ has covering dimension $\dim(G/K) - \dim(H/K_H) - 1$.
\end{proof}

%%%%%%%%%%%%%%%%%%%%%%%%%%%%%%%%%%%%%%%%%%%%%%%%%%%
\section{Nonexistence of compact quotients} \label{sec:no-cpt-quot}

The goal of this section is to prove Corollary~\ref{cor:no-cpt-quot}.
This is done in Section~\ref{subsec:proof-no-cpt-quot}, after recalling some bounds on cohomological dimension in Sections \ref{subsec:proof-vcd} and~\ref{subsec:Ano-vcd}.
We also prove and discuss Theorem~\ref{thm:cpt-quotient->H-QI} in Section~\ref{subsec:H-QI}.

%%%%%%%%%%%%%%%%%%%%%%%%%
\subsection{A bound on the cohomological dimension for properly discontinuous and cocompact actions} \label{subsec:proof-vcd}

We define the cohomological dimension $\mathrm{cd}(X)$ of a CW-complex~$X$ to be the largest integer $n\in \N$ for which there exists a local system $\mathrm L$ on~$X$ with $H^n(X,L) \neq \{0\}$.
The \emph{cohomological dimension of a group $\Gamma$}, denoted by $\mathrm{cd}(\Gamma)$, is then the cohomological dimension of any classifying space of $\Gamma$, or equivalently, the largest integer $n\in\N$ for which there exists a $\Z[\Gamma]$-module $M$ with $H^n(\Gamma,M) \neq \{0\}$.

When $\Gamma$ admits a finite-index subgroup which is torsion-free (which is the case \eg if $\Gamma$ is a finitely generated linear group, by the Selberg lemma \cite[Lem.\,8]{sel60}), then all torsion-free finite-index subgroups of~$\Gamma$ have the same cohomological dimension, called the \emph{virtual cohomological dimension} of~$\Gamma$ and denoted by $\mathrm{vcd}(\Gamma)$.

Fact~\ref{fact:vcd} is a direct consequence of the following similar inequality for the cohomological dimension of torsion-free (not necessarily finitely generated) discrete subgroups, which is due to Kobayashi \cite{kob89} for reductive $H$ and was generalized by Morita in \cite{mor17}.

\begin{fact} \label{fact:cd}
Let $G$ be a real connected Lie group and $H$ a closed connected subgroup of~$G$.
Let $K$ be a maximal compact subgroup of~$G$ and $K_H$ a maximal compact subgroup of~$H$, with $K_H\subset K$.
Let $\Gamma$ be any torsion-free discrete subgroup of~$G$ acting properly discontinuously on $G/H$.
Then
\[\mathrm{cd}(\Gamma) \leq \dim(G/K) - \dim(H/K_H) ,\]
with equality if and only if the action of $\Gamma$ on $G/H$ is cocompact.
\end{fact}

\begin{proof}
The manifold $\Gamma \backslash G/K_H$ admits two smooth fiber bundle maps
\[ \pi_H: \Gamma \backslash G/K_H \longrightarrow \Gamma \backslash G/H \quad\mathrm{and}\quad \pi_K: \Gamma \backslash G/K_H \longrightarrow \Gamma \backslash G/K .\]

By the Cartan--Iwasawa--Malcev theorem (see \cite{bor-bourbaki50}), the spaces $G/K$ and $H/K_H$ are contractible.
This has the following two consequences:
\begin{itemize}
  \item The space $\Gamma \backslash G/K$ is a classifying space for~$\Gamma$, hence
  \[\mathrm{cd}(\Gamma) = \mathrm{cd} (\Gamma \backslash G/K) .\]
  \item The fibration $\pi_H$ has contractible fibers and is thus a homotopy equivalence, hence
  \[\mathrm{cd} (\Gamma \backslash G/K_H) = \mathrm{cd}(\Gamma \backslash G/H) .\]
\end{itemize}

The fibration $\pi_K$, on the other hand, is a fibration with compact fibers $K/K_H$ over a classifying space for $\Gamma$.
A classical consequence of the Leray--Serre spectral sequence for fibrations then gives
\[\mathrm{cd}(\Gamma \backslash G/K_H) =  \mathrm{cd} (\Gamma \backslash G/K) + \dim(K/K_H) .\]

Putting these identities together, we get
\[\mathrm{cd}(\Gamma) = \mathrm{cd}(\Gamma \backslash G/H) - \dim(K/K_H) .\]
Finally, $\Gamma \backslash G/H$ is a manifold, hence
\[\mathrm{cd}(\Gamma \backslash G/H) \leq \dim(\Gamma\backslash G/H) = \dim(G/H) ,\]
with equality if and only if $\Gamma \backslash G/H$ is closed.
Therefore
\[\mathrm{cd}(\Gamma) \leq \dim(G/H) - \dim(K/K_H) = \dim(G/K)-\dim(H/K_H) ,\]
with equality if and only if the action of $\Gamma$ on $G/H$ is cocompact.
\end{proof}

The following table recalls the dimensions of the Riemannian symmetric spaces of the classical real simple Lie groups and of~$G_{2(2)}$, from which one can compute the values of $\dim(G/K) - \dim(H/K_H)$ given in Table~\ref{table3}.

\begin{table}[!ht]
\centering
\begin{tabular}{|c|c|}
\hline
$G$ & $\dim(G/K)$\tabularnewline
\hline
$\SL(n,\R)$ & $(n-1)(n+2)/2$\tabularnewline
$\SL(n,\C)$ & $n^2 - 1$\tabularnewline
$\SL(n,\mathbb{H})$ & $(n-1)(2n+1)$\tabularnewline
$\SO(p,q)$ & $pq$\tabularnewline
$\SU(p,q)$ & $2 pq$\tabularnewline
$\Sp(p,q)$ & $4 pq$\tabularnewline
$\Sp(2n,\R)$ & $n(n+1)$\tabularnewline
$\Sp(2n,\C)$ & $n(2n+1)$\tabularnewline
$\SO^*(2n)$ & $n(n-1)$\tabularnewline
$G_{2(2)}$ & 8\tabularnewline
\hline
\end{tabular}
\vspace{0.2cm}
\caption{Dimensions of the Riemannian symmetric spaces $G/K$ of certain real semisimple Lie groups~$G$}
\label{table5}
\end{table}

%%%%%%%%%%%%%%%%%%%%%%%%%
\subsection{Bounds on the cohomological dimension for Anosov representations} \label{subsec:Ano-vcd}

Recall that the virtual cohomological dimension of a Gromov hyperbolic group is related to the covering dimension of its boundary by Fact~\ref{fact:dim-bound-vcd} due to Bestvina and Mess.
This can be used to give bounds on the virtual cohomological dimension of Anosov subgroups. Here we will use a rather elementary bound.
We refer for instance to the work of Canary and Tsouvalas \cite{ct20} for more refined bounds.

We start by stating the bound for $G=\SL(2\ell,\mathbb{K})$. Let $P_{\ell} = P_{\{\alpha_{\ell}\}}$ denote the stabilizer in $G$ of an $\ell$-dimensional linear subspace of $\mathbb{K}^{2\ell}$.

\begin{lemma} \label{lem:Ano-vcd-SL}
Let $G=\SL(2\ell,\mathbb{K})$ where $\ell\geq 1$ and $\mathbb{K}=\R$ (\resp $\C$, \resp the ring $\mathbb{H}$ of quaternions).
If a word hyperbolic group $\Gamma$ admits a $P_{\ell}$-Anosov representation to~$G$, then $\mathrm{vcd}(\Gamma)$ is at most $\ell+1$ (\resp $2\ell+1$, \resp $4\ell+1$).
\end{lemma}

\begin{proof}
Let $\rho : \Gamma\to G$ be a $P_{\ell}$-Anosov representation, with boundary map $\xi : \partial_{\infty}\Gamma\to G/P_\ell = \mathrm{Gr}_\ell(\mathbb{K}^{2\ell})$ (see Example~\ref{ex:Grass}).
We can see each $\xi(\eta)$, for $\eta\in\partial_{\infty}\Gamma$, as an $(\ell-\nolinebreak 1)$-dimensional projective subspace of $\PP(\mathbb{K}^{2\ell})$; these projective subspaces are pairwise disjoint by transversality of~$\xi$ (see Section~\ref{subsec:remind-Ano}).
Therefore the closed subset $\bigcup_{\eta\in\partial_{\infty}\Gamma} \xi(\eta)$ of $\PP(\mathbb{K}^{2\ell})$ fibers over $\partial_{\infty}\Gamma$ with fiber $\PP(\mathbb{K}^\ell)$.
We deduce that the covering dimension of $\partial_{\infty}\Gamma$ is at most $\ell\,\dim(\mathbb{K})$, and conclude using Fact~\ref{fact:dim-bound-vcd}.
\end{proof}

More generally, the following holds with a similar proof.

\begin{lemma} \label{lem:Ano-vcd}
Let $G$ be a connected real reductive algebraic group and $P,Q$ two proper parabolic subgroups of~$G$, with $P$ conjugate to an opposite parabolic subgroup~$P^*$.
Suppose that the $G$-orbit of $(P,Q)$ is closed in $G/P \times G/Q$ and that the complement consists of a single open $G$-orbit.
If a word hyperbolic group $\Gamma$ admits a $P$-Anosov representation to~$G$, then $\mathrm{vcd}(\Gamma) \leq \dim(G) - \dim(\Lie(P)+\Lie(Q)) + 1$.
\end{lemma}

\begin{proof}
Let $\mathcal{O} \subset G/P\times G/Q$ be the $G$-orbit of $(P,Q)$; it is a closed algebraic subset of $G/P\times G/Q$, of dimension $\dim(G) - \dim(P\cap Q)$.
For any $y\in G/P$,
$$\mathcal{N}_y \equaldef \{ z\in G/Q \,|\, (y,z) \in \mathcal{O}\}$$
is a closed algebraic subset of $G/Q$; we have $\mathcal{N}_{g\cdot y} = g\cdot\mathcal{N}_y$ for all $g\in G$ and all $y\in G/P$.
The stabilizer of $y$ in~$G$ is a conjugate of~$P$; since $G$ acts transitively on~$\mathcal{O}$, this stabilizer acts transitively on~$\mathcal{N}_y$, and the stabilizer of any point of~$\mathcal{N}_y$ is a conjugate of $P\cap Q$; therefore, $\dim(\mathcal{N}_y) = \dim(P) - \dim(P\cap Q)$.
Moreover, we have $\mathcal{N}_y \cap \mathcal{N}_{y'} = \emptyset$ whenever $y\in G/P$ and $y'\in G/P\simeq G/P^*$ are transverse (see \eg \cite[Lem.\,3.27]{st22}).

Let $\rho : \Gamma\to G$ be a $P$-Anosov representation, with boundary map $\xi : \partial_{\infty}\Gamma\to G/P$.
By transversality of~$\xi$ (see Remark~\ref{rem:comp-trans-self-opposite}), the subset $\bigcup_{\eta\in\partial_{\infty}\Gamma}\, \mathcal{N}_{\xi(\eta)}$ of $G/Q$ fibers over $\partial_{\infty}\Gamma$ with fibers of dimension $\dim(P) - \dim(P\cap Q)$.
We deduce that the covering dimension of $\partial_{\infty}\Gamma$ is at most
\begin{align*}
\dim(G/Q) - \dim(P) + \dim(P\cap Q) & = \dim(G) - \dim(Q) - \dim(P) + \dim(P\cap Q)\\
& = \dim(G) - \dim\big(\Lie(P)+\Lie(Q)\big).
\end{align*}
We conclude using Fact~\ref{fact:dim-bound-vcd}.
\end{proof}

If we take standard parabolic subgroups $P = P_{\theta}$ and $Q = P_{\theta'}$ as in \eqref{eqn:P-theta}, where $\theta,\theta' \subset \Delta$ are subsets of the simple restricted roots of~$G$, then $\dim(G) - \dim(\Lie(P)+\Lie(Q))$ is equal to the sum of the dimensions of the root spaces $\g_{\alpha}$ for $\alpha \in \Sigma^+\smallsetminus (\mathrm{span}(\Delta\smallsetminus\theta) \cup \mathrm{span}(\Delta\smallsetminus\theta'))$.

As a special case, we recover the following fact, where $G$ is the automorphism group of a nondegenerate $\R$-bilinear form on some finite-dimensional vector space over $\R$, $\C$, or the quaternions, and $P_{\ell} = P_{\{\alpha_{\ell}\}}$ is the stabilizer in~$G$ of a totally isotropic $\ell$-dimensional linear subspace of~$V$.

\begin{fact}[{\cite[Prop.\,8.3]{gw12}}] \label{fact:Ano-vcd-Aut(b)}
Let $G=\SO(p,q)$, $\SU(p,q)$, $\Sp(p,q)$, $\Sp(2q,\R)$, $\Sp(2q,\C)$, $\SO^*(4q)$, or $\SO^*(4q+2)$, where $p\geq q\geq 1$ are integers.
If a word hyperbolic group~$\Gamma$ admits a $P_1$-Anosov representation or a $P_q$-Anosov representation to~$G$, then $\mathrm{vcd}(\Gamma)$ is bounded as in Table~\ref{table6} below.
\end{fact}

Table~\ref{table6} actually corrects \cite[Prop.\,8.3]{gw12}, which contained typos for $G=\SO(2q,\C)$, $\SO(2q+1,\C)$, $\Sp(2q,\C)$, $\SO^*(4q)$, and $\SO^*(4q+2)$.

\begin{table}[!ht]
\centering
\begin{tabular}{|c|c|}
\hline
$G$ & Bound on $\mathrm{vcd}(\Gamma)$\tabularnewline
\hline
$\SO(p,q)$ & $p$\tabularnewline
$\SU(p,q)$ & $2p$\tabularnewline
$\Sp(p,q)$ & $4p$\tabularnewline
$\SO(2q,\C)$ & $q-1$\tabularnewline
$\SO(2q+1,\C)$ & $q+1$\tabularnewline
$\Sp(2q,\R)$ & $q+1$\tabularnewline
$\Sp(2q,\C)$ & $2q+1$\tabularnewline
$\SO^*(4q)$ & $4q-2$\tabularnewline
$\SO^*(4q+2)$ & $4q+2$\tabularnewline
\hline
\end{tabular}
\vspace{0.2cm}
\caption{Upper bound on $\mathrm{vcd}(\Gamma)$ for a word hyperbolic group~$\Gamma$ admitting a $P_1$-Anosov representation or a $P_q$-Anosov representation into~$G$. Here $p\geq q\geq 1$ are integers.}
\label{table6}
\end{table}

%%%%%%%%%%%%%%%%%%%%%%%%%
\subsection{Proof of Corollary~\ref{cor:no-cpt-quot}} \label{subsec:proof-no-cpt-quot}

We treat separately the various cases of Table~\ref{table3}.

\begin{proof}[Proof of Corollary~\ref{cor:no-cpt-quot} in case~(i)]
Write $G=\SL(2\ell,\mathbb{K})$ where $\mathbb{K}=\R$, $\C$, or the ring $\mathbb{H}$ of quaternions.
The restricted root system of~$G$ is of type $A_{2\ell-1}$.
We can identify $\aaa$ with the hyperplane of equation $t_1 + \dots + t_{2\ell} = 0$ in~$\R^{2\ell}$, and take
$$\aaa^+ = \{ (t_1,\dots,t_{2\ell}) \in \aaa \,|\, t_1\geq\dots\geq t_{2\ell}\}$$
(see Example~\ref{ex:Cartan-decomp-SLn}).
We then have $\mu(H) = \bigcup_{2\leq j\leq 2\ell-1} \mathrm{Ker}(e_j^*)$, and so $\aaa^+ \smallsetminus \mu(H)$ has $2\ell - 1$ connected components, namely $\mathcal{C}_1,\dots,\mathcal{C}_{2\ell-1}$ where
$$\mathcal{C}_i \equaldef \{ (t_1,\dots,t_{2\ell})\in\aaa^+ \,|\, t_1 \geq \dots \geq t_i > 0 > t_{i+1} \geq \dots \geq t_{2\ell}\}.$$
The opposition involution $\iota : \aaa\to\aaa$ maps $(t_1,\dots,t_{2\ell})$ to $(-t_{2\ell},\dots,-t_1)$, hence switches $\mathcal{C}_i$ and $\mathcal{C}_{2\ell-i}$ for all $1\leq i\leq 2\ell-1$.
In particular, there is a unique connected component of $\aaa^+ \smallsetminus \mu(H)$ which is invariant under~$\iota$, namely $\mathcal{C}_{\ell}$.

Let $\Gamma$ be a discrete subgroup of~$G$ acting properly discontinuously and cocompactly on $G/H$.
By Fact~\ref{fact:kas08}, the set $\mu(\Gamma)$ is contained in~$\mathcal{C}_{\ell}$ up to finitely many points.
Observe that the closure $\overline{\mathcal{C}_{\ell}}$ of $\mathcal{C}_{\ell}$ in~$\aaa^+$ satisfies $\overline{\mathcal{C}_{\ell}} \cap \mathrm{Ker}(\alpha_{\ell}) \subset \mu(H)$, where $\alpha_{\ell} : (t_1,\dots,t_{2\ell}) \mapsto t_{\ell}$.
Therefore
$$\langle\alpha_{\ell},\mu(\gamma)\rangle = d_{\aaa}(\mu(\gamma),\mathrm{Ker}(\alpha_{\ell})) \geq d_{\aaa}(\mu(\gamma),\mu(H))$$
for all but finitely many $\gamma\in\Gamma$.
By Theorem~\ref{thm:sharp}, the group $\Gamma$ is finitely generated and sharply embedded in~$G$ with respect to~$H$.
Therefore there exist $c,c'>0$ such that $\langle\alpha_{\ell},\mu(\gamma)\rangle \geq c\,|\gamma|_S - c'$ for all $\gamma\in\Gamma$, where $S$ is any fixed finite generating subset of~$\Gamma$.
By Fact~\ref{fact:klp}, this implies that the group $\Gamma$ is word hyperbolic and that the natural inclusion $\Gamma\hookrightarrow G$ is $P_{\ell}$-Anosov.
Lemma~\ref{lem:Ano-vcd-SL} then states that $\mathrm{vcd}(\Gamma)$ is at most $\ell+1$ (\resp $2\ell+1$, \resp $4\ell+1$) if $\mathbb{K}=\R$ (\resp $\C$, \resp $\mathbb{H}$).
Using the last column of Table~\ref{table3}, we see that this is strictly less than $\dim(G/K) - \dim(H/K_H)$ in all cases (since $\ell\geq 2$).
This contradicts Fact~\ref{fact:vcd}, which states that we must have $\mathrm{vcd}(\Gamma) = \dim(G/K) - \dim(H/K_H)$.
\end{proof}

\begin{proof}[Proof of Corollary~\ref{cor:no-cpt-quot} in cases (iii)--(iv)]
The restricted root system of~$G$ is of type $B_2$ in case~(iii), and of type $(BC)_2$ in case~(iv).
In both cases $\mu(H)$ is the wall $\mathrm{Ker}(\alpha_1)$.

Let $\Gamma$ be a discrete subgroup of~$G$ acting properly discontinuously and cocompactly on $G/H$.
By Theorem~\ref{thm:proper-compact<->Ano}, the group $\Gamma$ is word hyperbolic and the natural inclusion $\Gamma\hookrightarrow G$ is $P_{\{\alpha_1\}}$-Anosov.
Fact~\ref{fact:Ano-vcd-Aut(b)} then states that $\mathrm{vcd}(\Gamma)$ is at most $2k+k'$ in case~(iii), and $2(2k+k')$ in case~(iv).
Using the last column of Table~\ref{table3}, we see that this is strictly less than $\dim(G/K) - \dim(H/K_H)$ for $k'\geq 1$.
This contradicts Fact~\ref{fact:vcd}.
\end{proof}

\begin{proof}[Proof of Corollary~\ref{cor:no-cpt-quot} in cases (ii) and (v)--(x)']
The restricted root system of~$G$ is of type $B_{\ell}$ in case~(ix), of type $C_{\ell}$ in cases (ii), (vi), (vii), (viii), and~(x), of type $(BC)_{\ell}$ in cases (ix)', (ix)'', and~(x)', and of type $D_{\ell}$ in case~(v).
In all cases $\mu(H)$ is the wall $\mathrm{Ker}(\alpha_{\ell})$.

Let $\Gamma$ be a discrete subgroup of~$G$ acting properly discontinuously and cocompactly on $G/H$.
By Theorem~\ref{thm:proper-compact<->Ano}, the group $\Gamma$ is word hyperbolic and the natural inclusion $\Gamma\hookrightarrow G$ is $P_{\{\alpha_{\ell}\}}$-Anosov.
Fact~\ref{fact:Ano-vcd-Aut(b)} then gives an upper bound on $\mathrm{vcd}(\Gamma)$.
Using the last column of Table~\ref{table3}, we see that this upper bound is strictly less than $\dim(G/K) - \dim(H/K_H)$ in all cases.
This contradicts Fact~\ref{fact:vcd}.
\end{proof}

\begin{proof}[Proof of Corollary~\ref{cor:no-cpt-quot} in case~(xi)]
The restricted root system of~$G$ is of type $B_3$, and $\mu(H) = \aaa^+ \cap \mathrm{Ker}(\alpha_1 - \alpha_3)$.
The set $\aaa^+\smallsetminus\mu(H)$ has two connected components, namely $\mathcal{C}_1 = \{ \alpha_1>\alpha_3\}$ and $\mathcal{C}_3 = \{ \alpha_3>\alpha_1\}$.

Let $\Gamma$ be a discrete subgroup of~$G$ acting properly discontinuously and cocompactly on $G/H$.
By Fact~\ref{fact:kas08}, there exists $i\in\{1,3\}$ such that $\mu(\gamma) \in \mathcal{C}_i$ for all but finitely many $\gamma\in\Gamma$.
Observe that $\overline{\mathcal{C}_i} \cap \mathrm{Ker}(\alpha_i) \subset \mu(H)$.
Therefore
$$\langle\alpha_i,\mu(\gamma)\rangle = d_{\aaa}(\mu(\gamma),\mathrm{Ker}(\alpha_i)) \geq d_{\aaa}(\mu(\gamma),\mu(H))$$
for all but finitely many $\gamma\in\Gamma$.
By Theorem~\ref{thm:sharp}, the group $\Gamma$ is finitely generated and sharply embedded in~$G$ with respect to~$H$.
Therefore there exist $c,c'>0$ such that $\langle\alpha_i,\mu(\gamma)\rangle \geq c\,|\gamma|_S - c'$ for all $\gamma\in\Gamma$, where $S$ is any fixed finite generating subset of~$\Gamma$.
By Fact~\ref{fact:klp}, the group $\Gamma$ is word hyperbolic and the natural inclusion is $P_{\{\alpha_i\}}$-Anosov.
Fact~\ref{fact:Ano-vcd-Aut(b)} then implies $\mathrm{vcd}(\Gamma)\leq 5$, which is strictly less than {$\dim(G/K) - \dim(H/K_H) =~7$}.
This contradicts Fact~\ref{fact:vcd}.
\end{proof}

\begin{proof}[Proof of Corollary~\ref{cor:no-cpt-quot} in case~(xii)]
The restricted root system of~$G$ is of type $B_4$, and $\mu(H) = \aaa^+ \cap \mathrm{Ker}(\alpha_1 - \alpha_3)$.
The set $\aaa^+\smallsetminus\mu(H)$ has two connected components, namely $\mathcal{C}_1 = \{ \alpha_1>\alpha_3\}$ and $\mathcal{C}_3 = \{ \alpha_3>\alpha_1\}$.

Let $\Gamma$ be a discrete subgroup of~$G$ acting properly discontinuously and cocompactly on $G/H$.
Arguing exactly as in case~(xi) just above, we obtain that there exists $i\in\{1,3\}$ such that $\Gamma$ is word hyperbolic and the natural inclusion is $P_{\{\alpha_i\}}$-Anosov.
Fact~\ref{fact:Ano-vcd-Aut(b)} (for $i=1$) or Lemma~\ref{lem:Ano-vcd} (for $i=3$, with $P=P_{\{\alpha_3\}}$ and $Q=P_{\{\alpha_1\}}$) then implies $\mathrm{vcd}(\Gamma)\leq 6$, which is strictly less than $\dim(G/K) - \dim(H/K_H) = 8$.
This contradicts Fact~\ref{fact:vcd}.
\end{proof}

%%%%%%%%%%%%%%%%%%%%%%%%%
\subsection{The group $H$ must be quasi-isometrically embedded in~$G$} \label{subsec:H-QI}

Theorem~\ref{thm:sharp-general} easily implies Theorem~\ref{thm:cpt-quotient->H-QI}, which states that if $G/H$ admits compact quotients, then $H$ must be quasi-isometrically embedded in~$G$.

\begin{proof}[Proof of Theorem~\ref{thm:cpt-quotient->H-QI}]
Suppose that there is a discrete subgroup $\Gamma$ of~$G$ acting properly discontinuously and cocompactly on $G/H$.
Then $\Gamma\times H$ acts properly and cocompactly on $(G\times G)/\Diag(G)$.
By Theorem~\ref{thm:sharp-general}, the group $\Gamma\times H$ is compactly generated and quasi-isometrically embedded in $G\times G$, hence $H$ is quasi-isometrically embedded in~$G$ (see Lemma~\ref{lem:GammaxH-sharp-embed}).
\end{proof}

For reductive pairs $(G,H)$, the subgroup $H$ is always quasi-isometrically embedded in~$G$.
Indeed, finding a Cartan involution of $G$ preserving $H$ amounts to finding an $H$-equivariant geodesic embedding of $H/K_H \hookrightarrow G/K$.

On the other hand, Theorem~\ref{thm:cpt-quotient->H-QI} provides an obstruction to the existence of compact quotients of $G/H$ for nonreductive~$H$.
As an application, we obtain for instance a new proof of the fact \cite[Ex.\,7.1]{mor17'} that if $U$ is a unipotent subgroup of~$G$, then $G/U$ does not admit any compact quotients.
More generally, we obtain the following.

\begin{corollary}
If $H = H'\times U$ where $U$ is a unipotent subgroup of~$G$, then $G/H$ does not admit compact quotients.
\end{corollary}

\begin{proof}
The unipotent subgroup $U$ is \emph{not} quasi-isometrically embedded in~$G$.
Indeed, there exists an element $a\in G$ normalizing~$U$ such that for any nontrivial $u\in U$, the sequence $(a^n u a^{-n})_{n\in\N}$ grows exponentially in~$U$ but only linearly in~$G$ as $n$ tends to infinity.
On the other hand, the factor $U$ is quasi-isometrically embedded in $H= H'\times U$.
Therefore $H$ is not quasi-isometrically embedded in~$G$.
We conclude using Theorem~\ref{thm:cpt-quotient->H-QI}.
\end{proof}

The question of which closed connected subgroups of a reductive Lie group are quasi-isometrically embedded, though subtle, has a definitive answer given by the following criterion of Abels and Alperin \cite{aa11}.
Let $H$ be a closed connected subgroup of~$G$, with unipotent radical~$N_H$.
Let $L$ be a Levi factor of~$H$, so that $H = L \ltimes N_H$.
Let $A_L = \exp(\aaa_L)$ where $\aaa_L$ is a Cartan subspace of the Lie algebra of~$L$ (as in Section~\ref{subsec:Cartan-decomp}), and let $N_L$ be a maximal unipotent subgroup of $L$ normalized by~$A_L$.
Finally, let $U_H$ be the nilpotent subgroup $N_L\ltimes N_H$, which is still normalized by~$A_L$.
Note that $A_L\ltimes U_H$ is cocompact in~$H$ (by the Iwasawa decomposition in~$L$), hence $H$ is quasi-isometrically embedded in~$G$ if and only if $A_L \ltimes U_H$ is (see Corollary~\ref{cor:cocompact-QI}).

\begin{fact}[{\cite[Th.\,3.2 \& Rem.\,3.3]{aa11}}] \label{fact:QI embedding solvable}
Let $G$ be a connected real linear reductive Lie group and $H$ a closed connected subgroup of~$G$.
Then with the notation above, $H$ is quasi-isometrically embedded in~$G$ if and only if the action of $A_L$ on $U_H/[U_H,U_H]$ does not have a nontrivial global fixed point.
\end{fact}

Compact quotients of $G/H$ for nonreductive~$H$ have not been as thoroughly studied as for reductive~$H$, but we can mention the work of Oh--Witte \cite{ow00, ow02} exhibiting subtle examples of compact quotients of $G/H$ for $G = \SO(2n,2)$ and $H$ non reductive, as well as the work of Morita \cite{mor17',mor24}.

%%%%%%%%%%%%%%%%%%%%%%%%%%%%%%%%%%%%%%%%%%%%%%%%%%%
\section{Openness for proper and cocompact actions} \label{sec:openness}

We introduce the following terminology.

\begin{definition} \label{def:tau-proj-Ano}
Let $G$ be a connected real linear reductive Lie group, $(\tau,V)$ a finite-dimensional real linear representation of~$G$, and $\Gamma$ a word hyperbolic group.
A representation $\rho : \Gamma\to G$ is \emph{$\tau$-projective Anosov} if $\tau\circ\rho : \Gamma\to\GL(V)$ is $P_1$-Anosov, where $P_1$ is the stabilizer in $\GL(V)$ of a line of~$V$.
\end{definition}

The goal of this section is to establish the following proposition and corollary.
For any open subset $\mathcal{C}$ of~$\aaa^+$ and any group~$\Gamma$, we denote by $\Hom_{\mathcal{C}}(\Gamma,G)$ the set of representations $\rho : \Gamma\to G$ such that $\mu(\rho(\gamma)) \in \mathcal{C} \cup \iota(\mathcal{C})$ for all but possibly finitely many $\gamma\in\Gamma$ (where $\iota : \aaa\to\aaa$ is the opposition involution).
We say that a hyperplane of~$\aaa$ is \emph{rational} if it is the kernel of a linear form $\varphi_0\in\aaa^*$ satisfying $(\varphi_0,\alpha)_*/(\alpha,\alpha)_* \in \Q$ for all $\alpha\in\Delta$.

\begin{proposition} \label{prop:corank-1-sharp<->Ano}
Let $G$ be a connected real linear reductive Lie group and $H$ a closed subgroup of~$G$ such that $\mu(H)$ is the intersection of~$\aaa^+$ with the $W$-orbit of a rational hyperplane of~$\aaa$.
Let $\mathcal{C}$ be a connected component of $\aaa^+\smallsetminus\mu(H)$.
Then there is a finite-dimensional real linear representation $(\tau,V)$ of~$G$ such that for any finitely generated group~$\Gamma$ and any representation $\rho\in\Hom_{\mathcal{C}}(\Gamma,G)$, the following are equivalent:
\begin{enumerate}
  \item\label{item:sharp-Ano-1} $\rho$ is a quasi-isometric embedding and the action of $\rho(\Gamma)$ on $G/H$ is sharp;
  \item\label{item:sharp-Ano-2} $\Gamma$ is word hyperbolic and $\rho$ is $\tau$-projectively Anosov.
\end{enumerate}
For fixed~$\Gamma$, the set of representations $\rho\in\Hom_{\mathcal{C}}(\Gamma,G)$ satisfying these conditions is open in $\Hom(\Gamma,G)$.
\end{proposition}

\begin{corollary} \label{cor:corank-1-sharp-qi-open}
Let $G$ be a connected real linear reductive Lie group and $H$ a closed subgroup of~$G$ such that $\mu(H)$ is the intersection of~$\aaa^+$ with the $W$-orbit of a (not necessarily rational) hyperplane of~$\aaa$.
Then for any finitely generated group~$\Gamma$,
\begin{enumerate}
  \item if there is a quasi-isometric embedding $\rho : \Gamma\to G$ such that the action of $\rho(\Gamma)$ on $G/H$ is sharp, then $\Gamma$ is word hyperbolic and $\rho$ is $\tau$-projective Anosov for some finite-dimensional real linear representation $(\tau,V)$ of~$G$;
  \item the set of quasi-isometric embeddings $\rho : \Gamma\to G$ such that the action of $\rho(\Gamma)$ on $G/H$ is sharp is open in $\Hom(\Gamma,G)$. 
\end{enumerate}
\end{corollary}

Theorem~\ref{thm:openness} is an immediate consequence of Theorem~\ref{thm:sharp} and Corollary~\ref{cor:corank-1-sharp-qi-open}.

%%%%%%%%%%%%%%%%%%%%%%%%%
\subsection{Reminders: restricted weights of linear representations of~$G$}
\label{subsec:weights}

We use the notation of Sections \ref{subsec:Cartan-decomp} and~\ref{subsec:corank-1-sharp->Ano}.
In particular, $\aaa$ is a maximal abelian subspace of~$\p$ and we have an orthogonal direct sum $\aaa = \aaa_z \oplus \aaa_s$ for $(\cdot,\cdot)$, where $\aaa_z \equaldef \aaa \cap \mathfrak{z}(\g)$ and $\aaa_s \equaldef \aaa \cap [\g,\g]$, and correspondingly an orthogonal direct sum $\aaa^* = \aaa_s^0 \oplus \aaa_z^0$ for $(\cdot,\cdot)_*$, where $\aaa_s^0$ (\resp $\aaa_z^0$) is the set of linear forms on~$\aaa$ that vanish identically on $\aaa_s$ (\resp $\aaa_z$), and $(\cdot,\cdot)_*$ is the $W$-invariant scalar product on~$\aaa^*$ dual to $(\cdot, \cdot)$.
The subspace $\aaa_z^0$ admits a natural basis $\{\omega_{\alpha} \,|\, \alpha\in\Delta\}$, where $\omega_{\alpha}\in\aaa^*$ is the fundamental weight associated with~$\alpha$ (see \eqref{eqn:omega-alpha}).

Any finite-dimensional real linear representation $(\tau,V)$ of~$G$ decomposes under the action of~$\aaa$; the corresponding joint eigenvalues (elements of~$\aaa^*$) are called the \emph{restricted weights} of $(\tau,V)$.
The union of the restricted weights of all possible such representations is the set
\[\Phi \equaldef \left \{ \alpha \in \aaa^* \: \middle | \: 2\frac{(\alpha, \beta)_*}{(\beta, \beta)_*} \in \Z \quad \forall \beta \in \Sigma \right \} = \aaa_s^0 + \sum_{\alpha \in \Delta} \Z \, \omega_{\alpha}. \]
The set of \emph{dominant} weights is the semigroup
$$\Phi^+ = \aaa_s^0 + \sum_{\alpha \in \Delta} \N \, \omega_{\alpha}.$$
There is a partial ordering on~\(\aaa^*\) given by
\begin{equation*}
\nu \leq \nu' \quad \Longleftrightarrow \quad \nu' - \nu \in \sum_{\alpha \in \Delta} \R_{\geq 0} \, \alpha.
\end{equation*}
Given an \emph{irreducible} finite-dimensional real linear representation $(\tau,V)$ of~$G$, the set of restricted weights of $\tau$ admits, for that ordering, a unique maximal element (see \eg \cite[Cor.\,3.2.3]{gw09}), which is a dominant weight $\chi_{\tau}$ called the \emph{highest weight} of~$\tau$.

For each linear representation $(\tau,V)$ of~$G$ of dimension $n\in\N^*$, we choose a $\tau(K)$-invariant Euclidean structure on $V\simeq\R^n$ and denote by
$$\mu_V : G \longrightarrow \aaa_V^+ \equaldef \{\mathrm{diag}(t_1,\dots,t_n) \,|\, t_1\geq\dots\geq t_n\}$$
the composition of~$\tau$ with the corresponding Cartan projection of $\GL(V)$ (see Example~\ref{ex:Cartan-decomp-SLn}).
For $1\leq i\leq n$, we denote by $\varepsilon_i$ the linear form on $\aaa_V \equaldef \{\mathrm{diag}(t_1,\dots,t_n) \,|\, t_1,\dots,t_n\in\R\}$ given by $\langle\varepsilon_i,\mathrm{diag}(t_1,\dots,t_n)\rangle = t_i$.

\begin{lemma} \label{lem:mu-repr}
\begin{enumerate}
  \item\label{item:mu-repr-1} For any finite-dimensional real linear representations $(\tau^{(1)},V^{(1)}),\dots,\linebreak (\tau^{(m)},V^{(m)})$ of~$G$, we have
  $$\langle\varepsilon_1-\varepsilon_2, \mu_{V^{(1)}\otimes\dots\otimes V^{(m)}}(g)\rangle = \min_{1\leq\ell\leq m} \langle\varepsilon_1-\varepsilon_2, \mu_{V^{(\ell)}}(g)\rangle \quad \text{for all }g\in G.$$
  \item\label{item:mu-repr-2} For any finite-dimensional real linear representations $(\tau,V), (\tau',V')$ of~$G$ and any $g\in G$, if $\langle\chi_{\tau},\mu(g)\rangle \geq \langle\chi_{\tau'},\mu(g)\rangle$, then
  $$\langle\varepsilon_1-\varepsilon_2, \mu_{V\oplus V'}(g)\rangle = \min\big( \langle\chi_{\tau}-\chi_{\tau'},\mu(g)\rangle, \langle\varepsilon_1-\varepsilon_2, \mu_V(g)\rangle\big).$$
  \item\label{item:mu-repr-3} There exists $N\in\N^*$ such that for any $\chi\in  N\Phi^+$, there is an irreducible representation $V_{\chi}$ of~$G$ with highest weight~$\chi$ such that
  $$\langle\varepsilon_1-\varepsilon_2, \mu_{V_{\chi}}(g)\rangle = \min_{\alpha\in\Delta,\, (\alpha,\chi)_*\neq 0} \,\langle\alpha,\mu(g)\rangle \quad \text{for all }g\in G.$$
\end{enumerate}
\end{lemma}

\begin{proof}
\eqref{item:mu-repr-1} This easily follows from the fact that the restricted weights of\linebreak $(\bigotimes_{1\leq i\leq m} \tau^{(i)}, \bigotimes_{1\leq i\leq m} V^{(i)})$ are the linear forms $\sum_{1\leq i\leq m} \chi^{(i)}$ where $\chi^{(i)}$ ranges through the restricted weights of $(\tau^{(i)},V^{(i)})$.

\eqref{item:mu-repr-2} This is contained in \cite[Lem.\,7.11]{ggkw17}.

\eqref{item:mu-repr-3} By \cite[Lem.\,3.2 \& 3.7]{ggkw17}, there exists $N\in\N^*$ such that for any $\chi''\in\linebreak  N\sum_{\alpha \in \Delta} \N\,\omega_{\alpha}$, there is an irreducible representation $V_{\chi''}$ with highest weight~$\chi''$ such that $\langle\varepsilon_1-\nolinebreak\varepsilon_2, \mu_{V_{\chi''}}(g)\rangle = \min_{\alpha\in\Delta,\, (\alpha,\chi'')_*\neq 0}\, \langle\alpha,\mu(g)\rangle$ for all $g\in G$.
Let $\chi \equaldef \chi' + \chi''$ where $\chi'\in\aaa_s^0$.
There is a one-dimensional real linear representation $V_{\chi'}$ of~$G$ (factoring through the center of~$G$) with highest weight~$\chi'$.
Then $V_{\chi} \equaldef V_{\chi'} \otimes V_{\chi''}$ is irreducible with highest weight~$\chi$, and satisfies the required properties.
\end{proof}

%%%%%%%%%%%%%%%%%%%%%%%%%
\subsection{Proof of Proposition~\ref{prop:corank-1-sharp<->Ano}}

Let us first prove the equivalence between conditions \eqref{item:sharp-Ano-1} and~\eqref{item:sharp-Ano-2}.
By assumption, there exists $\varphi_0\in\aaa^*$ such that $\mu(H) = \aaa^+ \cap \bigcup_{w\in W} \mathrm{Ker}(w\cdot\varphi_0)$, and we can write $\varphi_0 = \psi_0 + \sum_{\alpha\in\Delta} t_{\alpha} \, \omega_{\alpha}$ where $\psi_0 \in \aaa_s^0$ and $t_{\alpha} = (\varphi_0,\alpha)_*/(\alpha,\alpha)_* \in \Q$ for all $\alpha\in\Delta$.

We first suppose that $\h \supset \aaa\cap [\g,\g]$, \ie $\varphi_0 = \psi_0 \in \aaa_s^0$.
Then, up to multiplying $\varphi_0$ by an integer, we can find a character $\chi : G\to \R_{>0}$ such that $d_{\aaa}(\mu(g),\mu(H)) = \langle\phi_0,\mu(g)\rangle = \log(\chi(g))$ for all $g\in G$.
Let $\tau : G\to\SL(2,\R)$ be the linear representation sending $g$ to $\mathrm{diag}(\chi(g),\chi(g)^{-1})$.
Then for any group~$\Gamma$ with a finite generating subset~$S$ and any representation $\rho\in\Hom_{\mathcal{C}}(\Gamma,G)$, conditions \eqref{item:sharp-Ano-1} and~\eqref{item:sharp-Ano-2} are both equivalent to the existence of $c,c'>0$ such that $\log(\chi(\rho(\gamma))) \geq c\,|\gamma|_S - c'$ for all $\gamma\in\Gamma$, by Corollary~\ref{cor:sharp-embed-explain} and Fact~\ref{fact:klp}.
Note that, in that case, any group acting sharply on $G/H$ is virtually cyclic by Proposition~\ref{prop:cork-1-proper-compact->Ano}.

We now suppose that $\h \not\supset \aaa\cap [\g,\g]$, \ie $\varphi_0 \notin \aaa_s^0$ (and so $w\cdot\varphi_0 \notin \aaa_s^0$ for all $w\in W$).
Let $N\in\N^*$ be given by Lemma~\ref{lem:mu-repr}.\eqref{item:mu-repr-3}.
Since the Weyl group $W$ preserves the $\Q$-vector space spanned by the fundamental weights, for any $w\in W$ we can write $w\cdot\varphi_0 = \sum_{\alpha\in\Delta} t_{w,\alpha}\,\omega_{\alpha}$ where $t_{w,\alpha} \in \Q$ for all $\alpha\in\Delta$.
Up to multiplying $\varphi_0$ by an integer (which does not change the kernels $\mathrm{Ker}(w\cdot\varphi_0)$), we may and shall assume that $t_{w,\alpha} \in N\Z$ for all $w\in W$ and all $\alpha\in\Delta$.

The connected component $\mathcal{C}$ of $\aaa^+\smallsetminus\mu(H)$ can be written as
$$\mathcal{C} = \aaa^+ \cap \bigcap_{\varphi\in\mathcal{F}_{\mathcal{C}}} \{ \varphi>0\}$$
for some nonempty subset $\mathcal{F}_{\mathcal{C}}$ of $W\cdot\varphi_0 \cup W\cdot (-\varphi_0)$ such that $\overline{\mathcal{C}} \cap \mathrm{Ker}(\varphi) \neq \{0\}$ for all $\varphi\in\mathcal{F}_{\mathcal{C}}$.
We have
\begin{equation} \label{eqn:dist-to-boundary-in-C}
d_{\aaa}(x,\mu(H)) = \min_{\varphi\in\mathcal{F}_{\mathcal{C}}} \frac{\langle\varphi,x\rangle}{\sqrt{(\varphi,\varphi)_*}} \quad\text{ for all }x\in\mathcal{C}.
\end{equation}
For each $\varphi\in\mathcal{F}_{\mathcal{C}}$, let $\Delta^+_{\varphi}$ be the set of simple roots $\alpha\in\Delta$ such that $(\varphi,\alpha)_*>0$; it is nonempty by Lemma~\ref{lem:C-Ker-alpha}.
We can write $\varphi = \varphi^+ - \varphi^-$ where
$$\varphi^+ \equaldef \psi_0 + \sum_{\alpha\in\Delta^+_{\varphi}} \frac{(\varphi,\alpha)_*}{(\alpha,\alpha)_*} \, \omega_{\alpha} \quad\quad\mathrm{and}\quad\quad \varphi^- \equaldef \varphi^+ - \varphi$$
both belong to $N\Phi^+$.
Consider the linear representation
\begin{equation} \label{eqn:big-rep}
(\tau,V) \equaldef \bigotimes_{\varphi\in\mathcal{F}_{\mathcal{C}}} \left(V_{\varphi^+} \oplus V_{\varphi^-}\right).
\end{equation}
where $V_{\varphi^+}$ and~$V_{\varphi^-}$ are given by Lemma~\ref{lem:mu-repr}.\eqref{item:mu-repr-3}.
Let us check that it satisfies the conclusion of Proposition~\ref{prop:corank-1-sharp<->Ano}.

We first observe that by Lemma~\ref{lem:mu-repr}.\eqref{item:mu-repr-1}--\eqref{item:mu-repr-3}, for any $g\in G$ with $\mu(g)\in\mathcal{C}$ we have
\begin{eqnarray}
\langle\varepsilon_1-\varepsilon_2, \mu_V(g)\rangle & = & \min_{\varphi\in\mathcal{F}_{\mathcal{C}}} \langle\varepsilon_1-\varepsilon_2, \mu_{V_{\varphi^+}\oplus V_{\varphi^-}}(g)\rangle\nonumber\\
& = & \min_{\varphi\in\mathcal{F}_{\mathcal{C}}} \min\big( \langle\varphi^+ - \varphi^-,\mu(g)\rangle, \langle\varepsilon_1-\varepsilon_2, \mu_{V_{\varphi^+}}(g)\rangle\big)\nonumber\\
& = & \min_{\varphi\in\mathcal{F}_{\mathcal{C}}} \min\bigg( \langle\varphi,\mu(g)\rangle, \min_{\alpha\in\Delta^+_{\varphi}} \langle\alpha,\mu(g)\rangle\bigg)\label{eqn:mu1-mu2-V} .
\end{eqnarray}
We fix a symmetric finite generating subset $S$ of~$\Gamma$.\\

\noindent \eqref{item:sharp-Ano-2}~$\Rightarrow$~\eqref{item:sharp-Ano-1}: Suppose $\rho\in\Hom_{\mathcal{C}}(\Gamma,G)$ is $\tau$-projective Anosov (Definition~\ref{def:tau-proj-Ano}).
Then there exist $c,c'>0$ such that $\langle\varepsilon_1 - \varepsilon_2, \mu_V(\rho(\gamma))\rangle \geq c\,|\gamma|_S - c'$ for all $\gamma\in\Gamma$.
By \eqref{eqn:mu1-mu2-V} we have $\langle\varphi,\mu(\rho(\gamma))\rangle \geq c\,|\gamma|_S - c'$ for all $\varphi\in\mathcal{F}_{\mathcal{C}}$ and all $\gamma\in\Gamma$ with $\mu(\rho(\gamma))\in\mathcal{C}$, hence (see \eqref{eqn:dist-to-boundary-in-C}) there exist $c'',c'''>0$ such that $d_{\aaa}(\mu(\rho(\gamma)),\mu(H)) \geq c''\,|\gamma|_S - c'''$ for all $\gamma\in\Gamma$ with $\mu(\rho(\gamma))\in\mathcal{C}$.
On the other hand, for any $\gamma\in\Gamma$ with $\mu(\rho(\gamma))\in\iota(\mathcal{C})$, we have $\mu(\rho(\gamma^{-1})) = \iota(\mu(\rho(\gamma))) \in \mathcal{C}$; since $\mu(H) = \iota(\mu(H))$ and $|\gamma^{-1}|_S = |\gamma|_S$, we get
$$d_{\aaa}(\mu(\rho(\gamma)),\mu(H)) = d_{\aaa}(\mu(\rho(\gamma^{-1})),\mu(H)) \geq c''\,|\gamma|_S - c'''. $$
Since $\mu(\rho(\gamma)) \in \mathcal{C}\cup\iota(\mathcal{C})$ for all but possibly finitely many $\gamma\in\Gamma$, we deduce (see Corollary~\ref{cor:sharp-embed-explain}) that $\rho$ is a quasi-isometric embedding and that the action of $\rho(\Gamma)$ on $G/H$ is sharp.\\

\noindent \eqref{item:sharp-Ano-1}~$\Rightarrow$~\eqref{item:sharp-Ano-2}: Suppose that $\rho\in\Hom_{\mathcal{C}}(\Gamma,G)$ is a quasi-isometric embedding and that the action of $\rho(\Gamma)$ on $G/H$ is sharp.
Then there exist $c,c'>0$ such that\linebreak $d_{\aaa}(\mu(\rho(\gamma)),\mu(H)) \geq c\,|\gamma|_S - c'$ for all $\gamma\in\Gamma$.
By \eqref{eqn:dist-to-boundary-in-C} and Lemma~\ref{lem:C-Ker-alpha}, there exists $c''>0$ such that for any $\varphi\in\mathcal{F}_{\mathcal{C}}$, any $\alpha\in\Delta^+_{\varphi}$, and any $x\in\mathcal{C}$, 
$$\langle\varphi,x\rangle \geq c''\,d_{\aaa}(x,\mu(H)) \quad\quad\mathrm{and}\quad\quad \langle\alpha,x\rangle \geq c''\,d_{\aaa}(x,\mu(H)).$$
By \eqref{eqn:mu1-mu2-V}, we then have $\langle\varepsilon_1 - \varepsilon_2, \mu_V(\rho(\gamma))\rangle \geq cc''\,|\gamma|_S - c'c''$ for all $\gamma\in\Gamma$ with $\mu(\rho(\gamma))\in\mathcal{C}$.
Similarly, there exist $c''',c''''>0$ such that $\langle\varepsilon_1 - \varepsilon_2, \mu_V(\rho(\gamma))\rangle \geq c'''\,|\gamma|_S - c''''$ for all $\gamma\in\Gamma$ with $\mu(\rho(\gamma))\in\iota(\mathcal{C})$.
Since $\mu(\rho(\gamma)) \in \mathcal{C}\cup\iota(\mathcal{C})$ for all but possibly finitely many $\gamma\in\Gamma$, we deduce from Fact~\ref{fact:klp} that $\Gamma$ is word hyperbolic and $\rho$ is $\tau$-projective Anosov.\\

Suppose $\rho\in\Hom_{\mathcal{C}}(\Gamma,G)$ satisfies conditions \eqref{item:sharp-Ano-1} and~\eqref{item:sharp-Ano-2}.
By \cite{lab06}, there is a neighborhood $\mathcal{U}$ of $\rho$ in $\Hom(\Gamma,G)$ such that any $\rho'\in\mathcal{U}$ still satisfies~\eqref{item:sharp-Ano-2}.
We claim that, up to making $\mathcal{U}$ smaller, it is contained in $\Hom_{\mathcal{C}}(\Gamma,G)$.
Indeed, let $\gamma\in\Gamma$ be an infinite-order element.
Then $\lambda(\gamma)\neq 0$ and $\lambda(\rho(\gamma))\in\mathcal{C}\cup\iota(\mathcal{C})$ by sharpness, where $\lambda : G\to\aaa^+$ is the Jordan projection of \eqref{eqn:lambda}.
By continuity of~$\lambda$, up to making $\mathcal{U}$ smaller, for any $\rho'\in\mathcal{U}$ we have $\lambda(\rho'(\gamma))\in\mathcal{C}\cup\iota(\mathcal{C})$; then $\mu(\rho'(\gamma^n))\in\mathcal{C}\cup\iota(\mathcal{C})$ for all large enough $n\in\N$, and so Fact~\ref{fact:kas08} implies that $\rho'\in\Hom_{\mathcal{C}}(\Gamma,G)$.
This completes the proof of Proposition~\ref{prop:corank-1-sharp<->Ano}.

%%%%%%%%%%%%%%%%%%%%%%%%%
\subsection{Proof of Corollary~\ref{cor:corank-1-sharp-qi-open}}

Let $\rho : \Gamma\to G$ be a quasi-isometric embedding such that the action of $\rho(\Gamma)$ on $G/H$ is sharp.
By assumption, there exists $\varphi\in\aaa^*$ such that $\mu(H) = \aaa^+ \cap \bigcup_{w\in W} \mathrm{Ker}(w\cdot\varphi)$.
By Fact~\ref{fact:kas08}, the limit cone $\mathcal{L}_{\rho(\Gamma)}$ of $\rho(\Gamma)$ (see Section~\ref{subsec:lim-cone}) satisfies $\mathcal{L}_{\rho(\Gamma)}\smallsetminus \{0\} \subset \mathcal{C}_\varphi \cup \iota(\mathcal{C}_\varphi)$ for some connected component $\mathcal{C}_\varphi$ of $\aaa^+ \smallsetminus \left( \bigcup_{w\in W} w\cdot \mathrm{Ker}(\varphi) \right)$.
By compactness of the projectivization of the limit cone, for any $\varphi'\in\aaa^*$ close enough to~$\varphi$, we still have $\mathcal{L}_{\rho(\Gamma)} \subset \mathcal{C}_{\varphi'} \cup \iota(\mathcal{C}_{\varphi'})$ for some connected component $\mathcal{C}_{\varphi'}$ of $\aaa^+ \smallsetminus \bigcup_{w\in W} \mathrm{Ker}(w\cdot\varphi')$.
By density of $\Q$ in~$\R$, we can find a finite subset $\mathcal{F} \subset \aaa_s^0 + \sum_{\alpha\in\Delta} \Q\,\omega_{\alpha}$ of rational linear forms on~$\aaa$, all contained in a small neighborhood of $\varphi$ in~$\aaa^*$, such that
\[\mathcal L_{\rho(\Gamma)} \subset \bigcap_{\varphi' \in \mathcal F} \big(\mathcal{C}_{\varphi'} \cup \iota(\mathcal{C}_{\varphi'})\big) \subset \mathcal{C}_{\varphi} \cup \iota(\mathcal{C}_{\varphi}) .\]
In particular, $\rho$ belongs to $\Hom_{\mathcal{C}_{\varphi'}}(\Gamma,G)$ for all $\varphi' \in \mathcal{F}$.
By Proposition~\ref{prop:corank-1-sharp<->Ano}, the group $\Gamma$ is word hyperbolic and for every $\varphi' \in \mathcal{F}$, there is a finite-dimensional real linear representation $(\tau_{\varphi'},V_{\varphi'})$ of~$G$ such that $\rho$ is $\tau_{\varphi'}$-projective Anosov; moreover, there is a neighborhood $\mathcal{U}$ of $\rho$ in $\Hom(\Gamma,G)$ such that any $\rho'\in\mathcal{U}$ is still a quasi-isometric embedding with
\[\mathcal L_{\rho'(\Gamma)} \subset \bigcap_{\varphi' \in \mathcal F} \big(\mathcal{C}_{\varphi'} \cup \iota(\mathcal{C}_{\varphi'})\big) \subset \mathcal{C}_{\varphi} \cup \iota(\mathcal{C}_{\varphi}) ,\]
which implies (see Remark~\ref{rem:sharp-lim-cone}) that the action of $\rho'(\Gamma)$ on $G/H$ is sharp.
This completes the proof of Corollary~\ref{cor:corank-1-sharp-qi-open}.

%%%%%%%%%%%%%%%%%%%%%%%%%%%%%%%%%%%%%%%%%%%%%%%%%%%
\section{Manifolds locally modeled on big cells of~$G$} \label{sec:P-Popp}

In this section we prove a result (Theorem~\ref{thm:PxPopp-structures}) that will be useful in Section~\ref{sec:quotients-SL3}. It is a higher-rank variation on a result proved by the second-named author for Lie groups of rank one \cite{tho15}.

Recall that, given a manifold $X$ with a transitive action of a Lie group $L$, a \emph{$(L,X)$-structure} on a manifold~$M$ (in the sense of Ehresmann and Thurston) is the data of a pair $(\mathrm{dev}, \mathrm{hol})$, where $\mathrm{hol}$ is a representation of $\pi_1(M)$ to $L$ (the \emph{holonomy}) and $\mathrm{dev}: \widetilde M \to X$ is a $\mathrm{hol}$-equivariant local diffeomorphism (the \emph{developing map}); we call it \emph{complete} if $\mathrm{dev}$ is a covering map.

\begin{theorem} \label{thm:PxPopp-structures}
Let $G$ be a connected real linear reductive Lie group, with biinvariant volume form~$\omega$.
Let $P$ and~$P^*$ be two opposite proper parabolic subgroups of~$G$, and let $\Omega = P P^*$ be the open $(P\times P^*)$-orbit in~$G$ (for the action by left and right multiplication).
Then
\begin{enumerate}
  \item\label{item:PxPopp-1} A closed manifold cannot admit a $(P\times P^*, \Omega)$-structure.
  \item\label{item:PxPopp-2} A manifold cannot admit a complete $(P \times P^{*}, \Omega)$-structure which is of finite volume with respect to~$\omega$.
  \item\label{item:PxPopp-3} A discrete subgroup of $P \times P^*$ cannot act properly discontinuously and cocompactly on $G$ by left and right multiplication (or equivalently on $(G\times G)/\Diag(G)$).
\end{enumerate}
\end{theorem}

\begin{proof}
By construction, $P\times P^*$ acts transitively on~$\Omega$, and the stabilizer of the identity element $\1_G\in\Omega$ is $L \equaldef P\cap P^*$; therefore, $\Omega$ identifies with $(P\times P^*)/\Diag(L)$.

Let $a\in\g$ be such that $\ad_a$ is diagonalizable over $\R$ and $\mathrm{Lie}(P)$ (\resp $\mathrm{Lie}(P^*)$) is the sum of the nonnegative (\resp nonpositive) eigenspaces of $\ad_a$ in~$\g$.
The adjoint action of the stabilizer $L$ of $\1_G\in\Omega$ fixes~$a$, hence $a$ (seen as a tangent vector to~$\Omega$ at $\1_G$) extends to a $(P \times P^*)$-invariant vector field $X_a$ on~$\Omega$.

At every point $g = p p'$ in $\Omega$ (where $p\in P$ and $p'\in P^*$), we have 
\[X_a = L_{p}\circ R_{p'}(a) = R_g( \Ad_{p}(a)) ,\]
where $L_g : T\Omega\to T\Omega$ (\resp $R_g : T\Omega\to T\Omega$) is the pushforward by left (\resp right) multiplication by~$g$.
Since $\Ad_{p}(a)$ belongs to $\mathrm{Lie}(P)$, and since the left action of $G$ is spanned by right-invariant vector fields, we get that $X_a$ is tangent to the left-$P$-orbit at every $g\in \Omega$.
Moreover, along this $P$-orbit, $X_a$ is left-invariant.
Hence $X_a$ generates a complete flow $(\Phi_t)_{t\in \R}$ on~$\Omega$.

The Lie derivative of $\omega$ along $X_a$ can be written as $X_a \cdot \omega = \mathrm{div}(X_a) \omega$, where $\mathrm{div}(X_a)$ is a constant by $(P\times P^*)$-invariance of both $X_a$ and~$\omega$.
We claim that
\begin{equation} \label{eqn:div<0}
\mathrm{div}(X_a) <0 .
\end{equation} 
Indeed, consider a diagonalization basis $(e_1, \ldots , e_d)$ of $\ad_a$ in~$\g$, with corresponding eigenvalues $(\lambda_1,\ldots , \lambda_d)$.
Extend $(e_1,\ldots, e_d)$ to a left-$G$-invariant frame $(E_1,\ldots, E_d)$ of $TG$ and let $(E_1^*, \ldots, E_d^*)$ be the dual frame of $T^*G$.
Since $\omega$ is left-$G$-invariant, we may assume, up to scaling, that $\omega = E_1^* \wedge \ldots \wedge E_d^*$. 
Let us write $X_a = \sum_{i=1}^d \alpha_i E_i$.
The Cartan formula gives
\[ \mathrm{div}(X_a) = \sum_{i=1}^d  d \alpha_i(E_i)~,\]
which we now compute at the point $\1_G$.
Fix $i\in \{1,\ldots ,d\}$.
If $\lambda_i \geq 0$, then $e_i \in \mathrm{Lie}(P)$; since $X_a$ is left-$P$-invariant, the $\alpha_j$ are constant along left-$P$-orbits and in particular $d \alpha_i (e_i) = 0$.
On the other hand, if $\lambda_i < 0$, then $e_i\in\mathrm{Lie}(P^*)$ and we have
\begin{eqnarray*}
X_a(\exp(t e_i)) &=& R_{\exp(te_i)}(a)\\
&=& L_{\exp(te_i)}(\Ad_{\exp(-t e_i)}(a)) \ =\ L_{\exp(te_i)}(\exp(-t\ad_{e_i})(a))\\
&=& L_{\exp(t e_i)}(a+ \lambda_i t e_i) ;
\end{eqnarray*}
taking the derivative at $t= 0$ and using the left-invariance of the vector fields $E_j$, we get $d \alpha_i(e_i) = \lambda_i  < 0$.
We conclude that $\mathrm{div}(X_a)$ is the sum of all the negative eigenvalues of $\ad_a$, hence \eqref{eqn:div<0} holds.\\

\noindent \eqref{item:PxPopp-1}, \eqref{item:PxPopp-2}: Let $M$ be a manifold equipped with a $(P\times P^{*}, \Omega)$-structure $(\mathrm{dev}, \mathrm{hol})$.
Pulling back $\omega$ and $X_a$ by $\mathrm{dev}$ and quotienting by $\pi_1(M)$, we obtain a volume form $\bar \omega$ and a vector field $\bar X_a$ on~$M$, which still satisfy
\[\bar X_a \cdot \bar \omega = \mathrm{div}(X_a) \bar \omega~.\]
If $M$ is closed, then $\bar X_a$ generates a complete flow $(\Phi_t)_{t\in\R}$ on~$M$.
If $M$ is not assumed to be closed but the $(P\times P^*, \Omega)$-structure is complete, then the completeness of the flow of $X_a$ on $\Omega$ also guarantees that $\bar X_a$ generates a complete flow $(\Phi_t)_{t\in\R}$ on~$M$.
Suppose by contradiction that we are in one of these cases and that $\int_M \bar \omega$ is finite (this is automatic if $M$ is closed).
The change of variable formula gives
\[\int_M \Phi_t^*\,\bar{\omega} = \int_M \bar \omega \]
for all $t\in\R$.
On the other hand, we have
\[\Phi_t^*\, \bar{\omega} = e^{t\,\mathrm{div}(X_a)}\,\bar{\omega} ,\]
where $e^{t\,\mathrm{div}(X_a)}<1$ for $t>0$ by \eqref{eqn:div<0}.
These two identities together yield a contradiction.\\

\noindent \eqref{item:PxPopp-3} Suppose by contradiction that there is a discrete subgroup $\Gamma$ of $P\times P^*$ acting properly discontinuously and cocompactly on~$G$ by left and right multiplication.
Up to passing to a finite-index subgroup, we may assume (by the Selberg lemma \cite[Lem.\,8]{sel60}) that $\Gamma$ is torsion-free, hence that the action of $\Gamma$ on~$G$ is also torsion-free.
Then 
\[\Gamma \backslash \Omega \subset \Gamma \backslash G\]
is a manifold equipped with a complete $(P\times P^*, \Omega)$-structure of finite volume, contradicting~\eqref{item:PxPopp-2}.
\end{proof}

%%%%%%%%%%%%%%%%%%%%%%%%%%%%%%%%%%%%%%%%%%%%%%%%%%%
\section{Quotients of group manifolds of type $A_2$} \label{sec:quotients-SL3}

In this section, we consider a real linear simple Lie group~$G_0$ with a restricted root system of type~$A_2$: namely, $G_0$ is locally isomorphic to $\SL(3,\K)$ for $\K= \R, \C$ or the ring $\mathbb{H}$ of quaternions, or to $E_{6(-26)}$.
Our goal is to prove Theorem~\ref{thm:group-mfd-SL3}, which states that any discrete subgroup $\Gamma$ of $G_0 \times G_0$ acting properly discontinuously and cocompactly on the group manifold $G_0$ is virtually a uniform lattice of~$G_0$, embedded as a factor of~$G_0\times G_0$.

The key fact here is that the opposition involution of the root system of type $A_2$ is nontrivial and has a fixed locus of dimension $1$. To explain how we use this property, let us start by sketching the proof of Theorem~\ref{thm:group-mfd-SL3} in the case that $\Gamma$ is Zariski-dense in $G_0 \times G_0$.

%%%%%%%%%%%%%%%%%%%%%%%%%
\subsection{The Zariski-dense case} \label{subsec:SL3-Zariski-dense-case}

Let $\aaa^+$ be a closed Weyl chamber in a Cartan subspace $\aaa$ of the Lie algebra of~$G_0$, let $\iota: \aaa^+ \to \aaa^+$ be the opposition involution, and let $v_0\in\aaa^+$ be a nonzero vector fixed by~$\iota$.
Then the fixed locus of $\iota$ is $\R_{\geq 0}\,v_0$.
We denote by $p_1$ (\resp $p_2$) the projection of $G_0\times G_0$ to the first (\resp second) factor of $G_0\times G_0$.

\begin{proposition} \label{prop:Zariski-dense=>discrete-faithful-proj}
Let $G_0$ be a real linear simple Lie group with restricted root system of type $A_2$.
Let $\Gamma$ be a Zariski-dense discrete subgroup of $G_0 \times G_0$.
If the action of $\Gamma$ on\linebreak $(G_0\times G_0)/\Diag(G_0)$ is sharp (Definition~\ref{def:sharp}), then either $p_1$ or $p_2$ restricted to~$\Gamma$ has discrete image and finite kernel.
\end{proposition}

\begin{proof}
Suppose by contradiction that neither $p_1$ nor~$p_2$ restricted to~$\Gamma$ has discrete image and finite kernel.
Then there exist sequences $(\gamma_n)_{n\in\N}$ and $(\gamma'_n)_{n\in\N}$ in $\Gamma$ such that $(p_2(\gamma_n))_{n\in\N}$ and $(p_1(\gamma'_n))_{n\in\N}$ remain bounded in~$G_0$, while $(p_1(\gamma_n))_{n\in\N}$ and $(p_2(\gamma'_n))_{n\in\N}$ leave every compact subset of~$G_0$.
Taking Cartan projections, we deduce that the limit cone $\mathcal{L}_\Gamma \subset \aaa^+ \times \aaa^+$ of $\Gamma$ (see Section~\ref{subsec:lim-cone}) intersects both $\aaa^+ \times \{0\}$ and $\{0\} \times \aaa^+$. 

Let $v, w\in \aaa^+ \smallsetminus \{0\}$ be such that $(v,0)$ and $(0,w)$ belong to $\mathcal{L}_\Gamma$.
Since $\Gamma$ is Zariski-dense in $G_0\times G_0$, its limit cone $\mathcal{L}_\Gamma$ is convex by \cite{ben97}; moreover, $\mathcal{L}_\Gamma$ is invariant under the opposition involution~$\iota$.
Therefore $(v+ \iota(v), 0)$ and $(0,w + \iota(w))$ also belong to $\mathcal{L}_\Gamma$.
Finally, since the fixed locus of $\iota$ is one-dimensional, both $v+\iota(v)$ and $w+\iota(w)$ span $\R_{\geq 0}\,v_0$.
Thus $(v_0,0)$ and $(0,v_0)$ belong to $\mathcal{L}_\Gamma$, and so does $(v_0,v_0)$.
This contradicts the fact that the action of $\Gamma$ on $(G_0\times G_0)/\Diag(G_0)$ is sharp.
\end{proof}

\begin{proposition} \label{prop:discrete-projection=>lattice}
Let $G_0$ be any real linear simple Lie group with $\Rrank(G_0)\geq 2$.
Let $\Gamma$ be a discrete subgroup of $G_0\times G_0$ acting properly discontinuously and cocompactly on the group manifold~$G_0$.
Assume that the restriction of $p_1$ (\resp $p_2$) to $\Gamma$ has discrete image and finite kernel.
Then $\Gamma$ has a finite-index subgroup of the form $\Gamma_0 \times \{\1_{G_0}\}$ (\resp $\{\1_{G_0}\} \times \Gamma_0$), where $\Gamma_0$ is a uniform lattice in~$G_0$.
\end{proposition}

\begin{proof}
We consider the case that the restriction of $p_1$ to~$\Gamma$ has discrete image and finite kernel (the other case is similar).
By the Selberg lemma \cite[Lem.\,8]{sel60}, up to passing to a finite-index subgroup, we may assume that $\Gamma$ is torsion-free; then the restriction of $p_1$ to~$\Gamma$ is discrete and faithful, hence $\Gamma = \{(\gamma, \rho(\gamma)) \mid \gamma \in \Gamma_0\}$ for some discrete subgroup $\Gamma_0$ of~$G_0$ and some group homomorphism $\rho: \Gamma_0 \to G_0$.

The groups $\Gamma$ and $\Gamma_0$ being isomorphic, they have the same virtual cohomological dimension, and since $\Gamma$ acts properly discontinuously and cocompactly on $G_0$, so does $\Gamma_0$ by Fact~\ref{fact:vcd}. Hence $\Gamma_0$ is a uniform lattice in $G_0$.

By Margulis superrigidity (see \eg \cite[\S\,16.1]{wit15}), either $\rho$ has bounded image or there exists a $\rho$-equivariant isometry of the Riemannian symmetric space of~$G_0$.
The latter case is not possible, otherwise we would have $\mu(\rho(\gamma)) = \mu(\gamma)$ for all $\gamma \in \Gamma_0$ and $\Gamma$ would not act properly discontinuously on $(G_0\times G_0)/\Diag(G_0)$.
Therefore $\rho$ has bounded image.

It is then a well-known consequence of Margulis arithmeticity (see \eg \cite[\S\,16.4]{wit15}) that the closure of the image of~$\rho$ is either finite or isogenous to a compact real form of the complexification of $G_0$. But $G_0$ cannot contain a compact real form of its complexification (since it is, itself, a noncompact real form).
Therefore $\rho$ has finite image and is trivial on a finite-index subgroup of~$\Gamma_0$.
\end{proof}

Combining Propositions \ref{prop:Zariski-dense=>discrete-faithful-proj} and~\ref{prop:discrete-projection=>lattice} with the Sharpness Theorem~\ref{thm:sharp}, we conclude that for $G_0$ with restricted root system of type~$A_2$, a discrete subgroup of $G_0\times G_0$ acting properly discontinuously and cocompactly on~$G_0$ is never Zariski-dense in $G_0\times G_0$.

In order to complete the proof of Theorem~\ref{thm:group-mfd-SL3}, we need to deal with all other possible Zariski closures: this will be the object of Section~\ref{subsec:SL3-general-case}.
Before that, we briefly discuss Zariski closures in a general setting.

%%%%%%%%%%%%%%%%%%%%%%%%%
\subsection{Zariski closures for proper and cocompact actions on general homogeneous spaces of reductive type}

Let $G/H$ be a homogeneous space of reductive type.
In general, it is \emph{not} true that any discrete subgroup of~$G$ acting properly discontinuously and cocompactly on $G/H$ has a reductive Zariski closure in~$G$.
For instance, for $G_0=\SO(n,1)$ or $\SU(n,1)$, there exist discrete subgroups $\Gamma$ of $G=G_0\times G_0$ acting properly discontinuously and cocompactly on $G/H = (G_0\times G_0)/\Diag(G_0)$ for which the projection of $\Gamma$ to one of the factors of $G_0\times G_0$ is unipotent \cite{gol85,ghy95,kob98}.

However, we expect that if $\Gamma$ is a discrete subgroup of~$G$ acting properly discontinuously and cocompactly on $G/H$, then the ``reductive part'' $\Gamma_{\red}$ of~$\Gamma$ should also act properly discontinuously and cocompactly on $G/H$.
This ``reductive part'' (also called ``semisimplification'' in \cite{ggkw17}) is defined as follows.
For any subgroup $\Gamma$ of~$G$, the Zariski closure $\overline{\Gamma}^Z$ of $\Gamma$ in~$G$ admits a Levi decomposition $\overline{\Gamma}^Z = L\ltimes U$, where $U$ is the unipotent radical of $\overline{\Gamma}^Z$ and $L$ is a reductive subgroup of $\overline{\Gamma}^Z$, unique modulo conjugation by~$U$; we denote by $\Gamma_{\red}$ the image of~$\Gamma$ under the projection $L\ltimes U \to L$.

\begin{conjecture} \label{conj:proper-cocompact-Gamma-red}
Let $G$ be a connected real linear reductive Lie group and $H$ a reductive subgroup of~$G$.
If $\Gamma$ is a discrete subgroup of~$G$ acting properly discontinuously and cocompactly on $G/H$, then so is~$\Gamma_{\red}$.
\end{conjecture}

We can reformulate the conjecture in terms of faithfulness and discreteness of the natural projection $\Gamma\to\Gamma_{\red}$, thanks to the following lemma.

\begin{lemma} \label{lem:reductive-projection-proper}
Let $G$ be a connected real linear reductive Lie group, $H$ a closed subgroup of~$G$, and $\Gamma$ a discrete subgroup of~$G$ acting properly discontinuously and cocompactly on $G/H$.
Suppose that $\Gamma_{red}$ is discrete in~$G$.
Then the following are equivalent:
\begin{enumerate}
  \item\label{item:red-proj-proper-1} $\Gamma_{red}$ acts properly discontinuously and cocompactly on $G/H$,
  \item\label{item:red-proj-proper-2} the natural projection $\Gamma\to\Gamma_{\red}$ is faithful.
\end{enumerate}
\end{lemma}

\begin{proof}
By the Sharpness Theorem~\ref{thm:sharp}, the action of $\Gamma$ on $G/H$ is sharp.
Therefore the limit cone $\mathcal{L}_\Gamma$ of~$\Gamma$ (see Section~\ref{subsec:lim-cone}) meets the limit cone $\mathcal{L}_H$ only in~$0$, and so does the Jordan limit cone $\mathcal{L}^{\lambda}_\Gamma$ (since it is contained in $\mathcal{L}_\Gamma$).
The projection to the Levi factor preserves the Jordan projection, hence $\mathcal{L}^{\lambda}_{\Gamma_{\red}} = \mathcal{L}^{\lambda}_\Gamma$.
Finally, since the Zariski closure of $\Gamma_{\red}$ in~$G$ is reductive, we have $\mathcal{L}_{\Gamma_{\red}} = \mathcal{L}^{\lambda}_{\Gamma_{\red}}$ (see Section~\ref{subsec:lim-cone}), hence $\mathcal{L}_{\Gamma_{\red}}$ meets $\mathcal{L}_H$ only in~$0$, and so the action of $\Gamma_{\red}$ on $G/H$ is sharp (Remark~\ref{rem:sharp-lim-cone}), hence properly discontinuous.

If the natural projection $\Gamma\to\Gamma_{\red}$ is faithful, then $\Gamma$ and~$\Gamma_{\red}$ are isomorphic, hence $\mathrm{vcd}(\Gamma) = \mathrm{vcd}(\Gamma_{\red})$; therefore the action of $\Gamma_{\red}$ on $G/H$ is cocompact by Fact~\ref{fact:vcd}.

If the natural projection $\Gamma\to\Gamma_{\red}$ is \emph{not} faithful, then its kernel $N$ is infinite and virtually polycyclic. Applying the Lyndon--Hochschild--Serre spectral sequence, one gets $\mathrm{vcd}(\Gamma) = \mathrm{vcd}(\Gamma_{\red}) + \mathrm{vcd}(N) > \mathrm{vcd}(\Gamma_{\red})$; therefore the action of $\Gamma_{\red}$ on $G/H$ is \emph{not} cocompact by Fact~\ref{fact:vcd}.
\end{proof}

The following related result will be useful in the sequel.

\begin{fact}[{Auslander, see \cite[Th.\,8.24]{rag72}}] \label{fact:discrete-proj}
Let $L\ltimes U$ be a connected real linear Lie group, where $L$ is reductive with center $Z(L)$ and $U$ is unipotent (and closed, connected, normal).
For any discrete subgroup $\Gamma$ of $L\ltimes U$, the projection of $\Gamma$ to $L/Z(L)$ is discrete.
\end{fact}

Unfortunately, the same does not hold in general for the projection $\Gamma_{\red}$ of $\Gamma$ to~$L$.
However, we will prove the following.

\begin{lemma} \label{lem:reductive-proj-not-discr}
In the setting of Fact~\ref{fact:discrete-proj}, either the projection of $\Gamma$ to $L$ is discrete and faithful, or there exists $u\in U\smallsetminus \{\1_G\}$ and a compact subset $\mathcal{C}$ of $Z(L)\ltimes U$ such that 
\begin{equation} \label{eqn:unipotent-almost-contained}
\{u^n \,|\, n\in \Z\} \subset \mathcal{C}\,(\Gamma \cap (Z(L)\ltimes U)) .
\end{equation}
\end{lemma}

In other words, either the projection of $\Gamma$ to $L$ is discrete and faithful, or $\Gamma$ ``almost contains'' an infinite nilpotent subgroup of $U$.

\begin{proof}
Suppose that the restriction to~$\Gamma$ of the natural projection $\pi_L : L\ltimes U \to L$ is \emph{not} discrete and faithful.
If $\Gamma$ contains an element $u\in U\smallsetminus\{\1_G\}$, then \eqref{eqn:unipotent-almost-contained} holds with $\mathcal{C} = \{\1_G\}$.
So we now assume $\Gamma \cap U = \{\1_G\}$, \ie the restriction of $\pi_L$ to~$\Gamma$ is faithful but not discrete. 

Let $M_0$ be the identity component of the closure of $\Gamma_{red} = \pi_L(\Gamma)$ in~$L$ (for the real topology).
We set $\Gamma_0 \equaldef \Gamma \cap \pi_L^{-1}(M_0)$.
By Fact~\ref{fact:discrete-proj}, the projection of $\Gamma$ to $L/Z(L)$ is discrete, hence $M_0\subset Z(L)$.
Since the restriction of $\pi_L$ to~$\Gamma$ is faithful, the subgroup $\Gamma_0$ is abelian, and its Zariski closure in $Z(L)\ltimes U$ is of the form $M\times U'$, where $M\subset Z(L)$ is the Zariski closure of $M_0$ and $U'\subset U$.

Since $U'$ is unipotent, it is isomorphic to a finite-dimensional real vector space.
The group $M$, on the other hand, is isomorphic to $K_M\times V$, where $K_M$ is compact and $V$ is a vector space.
The projection $\overline{\Gamma}_0$ of $\Gamma_0$ modulo~$K_M$ is a discrete subgroup of $V\times U'$; therefore it spans a vector subspace $W$ of $V\times U'$ and there is a compact subset $\mathcal{C}_W$ of~$W$ such that $\mathcal{C}_W\,\overline{\Gamma}_0 = W$.
Since the projection of $\Gamma$ modulo $U'$ is not discrete, $W$ must intersect $U'$ nontrivially.
Take $u\in U'\cap W \smallsetminus \{0\}$.
Then \eqref{eqn:unipotent-almost-contained} holds with $\mathcal{C} = K_M\mathcal{C}_W$.
\end{proof}

%%%%%%%%%%%%%%%%%%%%%%%%%
\subsection{Proof of Theorem~\ref{thm:group-mfd-SL3}} \label{subsec:SL3-general-case}

From now on, $G_0$ is a connected real simple linear Lie group with restricted root system of type $A_2$.
As in Section~\ref{subsec:SL3-Zariski-dense-case}, we choose a closed Weyl chamber $\aaa^+$ in a Cartan subspace $\aaa$ of the Lie algebra of~$G_0$, let $\iota: \aaa^+ \to \aaa^+$ be the opposition involution, and let $v_0\in\aaa^+$ be a nonzero vector fixed by~$\iota$.
We denote by $p_1$ (\resp $p_2$) the projection of $G_0\times G_0$ to the first (\resp second) factor of $G_0\times G_0$.

\begin{lemma} \label{lem:unipotent-radical-one-factor}
Let $\Gamma$ be a discrete subgroup of $G_0\times G_0$ acting properly discontinuously and cocompactly on $(G_0\times G_0)/\Diag(G_0)$.
Let $\Gamma_U$ be a normal unipotent subgroup of~$\Gamma$.
Then the projection of $\Gamma_U$ to one of the factors of $G_0\times G_0$ is trivial.
\end{lemma}

(Note that Conjecture~\ref{conj:proper-cocompact-Gamma-red} states that $\Gamma_U$ should always be trivial.)

\begin{proof}[Proof of Lemma~\ref{lem:unipotent-radical-one-factor}]
Write the set of simple restricted roots of~$G_0$ as $\Delta = \{\alpha_1,\alpha_2\}$, so that $\Sigma^+ = \{\alpha_1,\alpha_2,\alpha_1+\alpha_2\}$.
There are two conjugacy classes of (nontrivial) unipotent elements of~$G_0$: the class of $\exp(v_1)$ and the class of $\exp(v_1+v_2)$ where $v_i \in (\g_0)_{\alpha_i} \smallsetminus \{ 0\}$.

Assume by contradiction that both $p_1(\Gamma_U)$ and $p_2(\Gamma_U)$ are nontrivial.

Suppose first that $p_i(\Gamma_U)$ is conjugate to a subgroup of $\exp((\g_0)_{\alpha_1} + (\g_0)_{\alpha_1+\alpha_2})$ for both $i=1,2$.
Then all nontrivial elements of $p_1(\Gamma_U)$ and all nontrivial elements of $p_2(\Gamma_U)$ are conjugate to the same element $\exp(v_1)$ with $v_1 \in (\g_0)_{\alpha_1} \smallsetminus \{ 0\}$.
This contradicts the proper discontinuity of the action of $\Gamma$ on $(G_0\times G_0)/\Diag(G_0)$.

Suppose now there exists $i\in\{1,2\}$ such that $p_i(\Gamma_U)$ is \emph{not} conjugate to a subgroup of $\exp((\g_0)_{\alpha_1} + (\g_0)_{\alpha_1+\alpha_2})$.
Then the normalizer of $p_i(\Gamma_U)$ in~$G_0$ is contained in a minimal parabolic subgroup of~$G_0$, and so $p_i(\Gamma)$ is contained in a minimal parabolic subgroup of~$G_0$.
On the other hand, $p_{3-i}(\Gamma)$ is contained in a proper (not necessarily minimal) parabolic subgroup of~$G_0$.
Hence, up to conjugation, $\Gamma$ is contained in $P\times P^*$ for some proper parabolic subgroup $P$ of~$G_0$.
This contradicts Theorem~\ref{thm:PxPopp-structures}.
\end{proof}

Let $\mu = (\mu_1,\mu_2) : G_0\times G_0 \to \aaa^+\times\aaa^+$ and $\lambda = (\lambda_1,\lambda_2) : G_0\times G_0 \to \aaa^+\times\aaa^+$ be the Cartan projection and Jordan projection of $G_0\times G_0$ as in Sections \ref{subsec:Cartan-decomp} and~\ref{subsec:lim-cone}.
Let $\Vert\cdot\Vert$ be a Euclidean norm on~$\aaa$, invariant under the restricted Weyl group, as in Section~\ref{subsec:vector-dist}.

Let $\Gamma$ be a discrete subgroup of $G_0\times G_0$ and let $i\in\{ 1,2\}$.
We say that $\Gamma$ satisfies:
\begin{itemize}
  \item \emph{property $D_i$} if the restriction of $p_i$ to~$\Gamma$ has finite kernel and discrete image,
  \item \emph{property $L_i$} if for any $\varepsilon>0$, there exists $\gamma \in \Gamma$ such that
\[\left\{ \begin{array}{l}
\Vert \lambda_{3-i}(\gamma)\Vert \geq 1 ,\\
\Vert \lambda_i(\gamma)\Vert < \varepsilon \Vert \lambda_{3-i}(\gamma)\Vert ,\\
\Vert \lambda_{3-i}(\gamma) - \iota (\lambda_{3-i}(\gamma)) \Vert \leq \varepsilon \Vert \lambda_{3-i}(\gamma)\Vert ,
\end{array}\right.\]
  \item \emph{property $C_i$} if there exists $(\gamma_n) \in \Gamma^\N$ such that
\[\left\{ \begin{array}{l}
\Vert \mu_{3-i}(\gamma_n)\Vert = n + o(n) ,\\
\Vert \mu_i(\gamma_n)\Vert = o(n) ,\\
\Vert \mu_{3-i}(\gamma_n) - \iota (\mu_{3-i}(\gamma_n)) \Vert = o(n) .
\end{array}\right.\]
\end{itemize}

Property~$L_1$ is equivalent to asking that $\mathcal L_\Gamma^\lambda$ contain $\{0\} \times \R_{\geq 0}\,v_0$, and property~$C_1$ implies the same for $\mathcal L_\Gamma$.
Indeed, for any $x\in\aaa^+$,
\begin{equation} \label{eqn:dist-fixed-locus-iota}
 \Vert x - \iota(x)\Vert = 2\,d_{\aaa}\big(x,\R_{\geq 0}\,v_0\big) .
\end{equation}
Our goal is to prove the following lemma, from which we will conclude as in the proof of Proposition~\ref{prop:Zariski-dense=>discrete-faithful-proj}.

\begin{lemma} \label{lem:property-Di-or-Ci}
Let $\Gamma$ be a discrete subgroup of $G_0\times G_0$ acting properly discontinuously and cocompactly on $G_0$. Then for each $i=1,2$, the group $\Gamma$ satisfies property $D_i$ or~$C_i$.
\end{lemma}

The proof of this lemma will occupy some time, so it seems useful to postpone it and start by explaining how it implies Theorem~\ref{thm:group-mfd-SL3}.

\begin{proof}[Proof of Theorem~\ref{thm:group-mfd-SL3} assuming Lemma~\ref{lem:property-Di-or-Ci}]
If $\Gamma$ satisfies property $D_1$ or~$D_2$, then it satisfies the conclusion of Theorem~\ref{thm:group-mfd-SL3}  by Proposition~\ref{prop:discrete-projection=>lattice}.

Otherwise, by Lemma~\ref{lem:property-Di-or-Ci}, the group $\Gamma$ must satisfy both properties $C_1$ and $C_2$. Let $(\gamma_n)_{n\in\N}$ and $(\gamma'_n)_{n\in\N}$ be as in the definition of properties $C_1$ and $C_2$ respectively.
For any $n\in\N$, by \eqref{eqn:mu-subadd} we have $\Vert \mu_1(\gamma_n\gamma'_n) - \mu_1(\gamma'_n) \Vert \leq \Vert \mu_1(\gamma_n) \Vert = o(n)$, where $\Vert \mu_1(\gamma'_n) \Vert = n + o(n)$ and $d_{\aaa}(\mu_1(\gamma'_n), \R_{\geq 0}\,v_0) = o(n)$ by \eqref{eqn:dist-fixed-locus-iota}; therefore, $\Vert \mu_1(\gamma_n\gamma'_n) - n v_0 \Vert = o(n)$ by the triangle inequality, where we have normalized the $\iota$-fixed vector~$v_0$ so that $\Vert v_0 \Vert = 1$.
Similarly, writing $\Vert \mu_2(\gamma_n\gamma'_n) - \mu_2(\gamma_n) \Vert \leq \Vert \mu_2(\gamma'_n) \Vert = o(n)$, we obtain $\Vert \mu_2(\gamma_n\gamma'_n) - n v_0 \Vert = o(n)$.
The triangle inequality then yields $\Vert \mu_1(\gamma_n\gamma'_n) - \mu_2(\gamma_n\gamma'_n) \Vert = o(n)$, whereas $\Vert \mu_1(\gamma_n\gamma'_n) \Vert = n + o(n)$.
This contradicts the sharpness of the action of $\Gamma$ on $(G_0 \times G_0)/\Diag(G_0)$ given by Theorem~\ref{thm:sharp}.
\end{proof}

To prepare for the proof of Lemma~\ref{lem:property-Di-or-Ci}, let us start by giving some sufficient conditions for satisfying property $C_i$:

\begin{proposition} \label{prop:criteria-property-Ci}
For $i\in\{1,2\}$, let $N_i$ be the kernel of the restriction of $p_i$ to~$\Gamma$.
If one of the following properties hold:
\begin{enumerate}[{\rm (i)}]
  \item\label{item:Ci-1} $\Gamma$ satisfies property~$L_i$, or
  \item\label{item:Ci-2} $N_i$ contains a nontrivial unipotent element, or
  \item\label{item:Ci-3} the natural projection $N_i\to (N_i)_{\red}$ is \emph{not} discrete and faithful, or
  \item\label{item:Ci-4} $N_i$ is abelian and not virtually cyclic,
\end{enumerate}
then $\Gamma$ satisfies property $C_i$.
\end{proposition}

Here, we call reductive projection of a discrete group $\Lambda$ the projection to a Levi factor of its Zariski closure.

\begin{proof}
We treat the case $i=1$; the case $i=2$ is similar.

{\it Suppose \eqref{item:Ci-1} holds:} there exists $(\beta_k) \in \Gamma^{\N}$ such that
\[\left\{ \begin{array}{l}
\Vert \lambda_2\beta_k)\Vert \geq 1 ,\\
\Vert \lambda_1(\beta_k)\Vert < \Vert \lambda_2(\beta_k)\Vert /k ,\\
\Vert \lambda_2(\beta_k) - \iota (\lambda_2(\beta_k)) \Vert \leq \Vert \lambda_2(\beta_k)\Vert /k .
\end{array}\right.\]
Fix $k\in\N$ and let $\ell_k \equaldef \Vert \lambda_2(\beta_k) \Vert \geq 1$.
By \eqref{eqn:lambda}, we have
\[\frac{1}{n} \mu_1(\beta_k^n) \underset{n\to+\infty}{\longrightarrow} \lambda_1(\beta_k) \quad\mathrm{and}\quad \frac{1}{n} \mu_2(\beta_k^n) \underset{n\to+\infty}{\longrightarrow} \lambda_2(\beta_k) .\]
This implies that
\[ \big\Vert \mu_2\big(\beta_k^{\lfloor n/\ell_k\rfloor}\big) \big\Vert = n + o(n) \]
and that for $n$ large enough we have
\[\big\Vert \mu_1(\beta_k^{\lfloor n/\ell_k\rfloor}) \big\Vert \leq \frac{2n}{k} \quad\mathrm{and}\quad \big\Vert \mu_2(\beta_k^{\lfloor n\ell_k\rfloor}) - \iota \mu_2(\beta_k^{\lfloor n/\ell_k\rfloor}) \big\Vert \leq \frac{2n}{k}~.\]
Now, for any $n\in\N$, let $k_n\in\N$ be the supremum of $n$ and of all indices $k$ satisfying the three conditions
\[\left\{ \begin{array}{l}
(1-\frac{1}{k})n \leq \big\Vert \mu_2(\beta_k^{\lfloor n/\ell_k\rfloor}) \big\Vert \leq (1+\frac{1}{k})n ,\\
\Vert \mu_1(\beta_k^{\lfloor n/\ell_k\rfloor}) \Vert \leq 2n/k ,\\
\big\Vert \mu_2(\beta_k^{\lfloor n/\ell_k\rfloor}) - \iota \mu_2(\beta_k^{\lfloor n/\ell_k\rfloor}) \big\Vert \leq 2n/k .
\end{array}\right.\]
Since, for a given $k$, these properties are satisfied for $n$ sufficiently large, we have $k_n\to +\infty$ as $n\to +\infty$.
Then
\[\gamma_n \equaldef \beta_{k_n}^{\lfloor n/\ell_{k_n}\rfloor} \]
satisfies the conditions of property~$C_1$.\\

{\it Suppose \eqref{item:Ci-2} holds:} the group $\Gamma$ contains a nontrivial element of the form $\beta = (\1_{G_0},u)$ where $u\in G_0$ is unipotent.
There exists $v\in\aaa^+\smallsetminus\{0\}$, depending only on the conjugacy class of $u$ in~$G_0$, such that $(\mu_2(\beta^n) - \log(n)\,v)_{n\geq 1}$ is bounded; moreover $\iota(v) = v$ since $u$ and $u^{-1}$ are conjugate in~$G_0$.
Then
\[\gamma_n \equaldef \beta^{\lfloor e^{n/\Vert v \Vert} \rfloor}\ \]
satisfies the conditions of property~$C_1$.\\

{\it Suppose \eqref{item:Ci-3} holds:} the natural projection $N_i\to (N_i)_{\red}$ is \emph{not} discrete and faithful.
By Lemma~\ref{lem:reductive-proj-not-discr}, there is a nontrivial element of~$\Gamma$ of the form $\beta = (\1_{G_0},u)$ where $u\in G_0$ is unipotent, a compact subset $\mathcal{C}$ of~$\{\1_{G_0}\}\times G_0$, and a sequence $(\gamma_n) \in (N_i)^\N$ such that
\[ \beta^{\lfloor e^{n/\Vert v \Vert} \rfloor} \in \mathcal{C}\,\gamma_n\]
for all $n\in\N$.
By \eqref{eqn:mu-subadd}, the Cartan projection of $\gamma_n$ remains at bounded distance from that of $\beta^{\lfloor e^{n/\Vert v \Vert} \rfloor}$, and we conclude as in case~\eqref{item:Ci-2} that $(\gamma_n)_{n\in\N}$ satisfies the conditions of property~$C_1$.\\

{\it Suppose \eqref{item:Ci-4} holds:} $N_1 = \Gamma \cap (\{\1_{G_0}\}\times G_0)$ is abelian and not virtually cyclic.
Write $N_1 = \{\1_{G_0}\}\times\Lambda$ where $\Lambda\subset G_0$, and the Zariski closure of $\Lambda$ as $A_0' \ltimes U_0'$ where $A_0'$ is reductive and $U_0'$ unipotent (and closed, connected, normal).
If the projection of $\Lambda$ to $A_0'$ is not discrete and faithful, then \eqref{item:Ci-3} holds and we conclude as above.
Otherwise, the image of this projection contains a discrete copy of~$\Z^2$.
Since $G_0$ has real rank $2$, we deduce that the Jordan projection of $A_0'$ is the whole Weyl chamber~$\aaa_0^+$ and that the projection of $\Lambda$ is a lattice in~$A_0'$.
Therefore there exist elements of $\Lambda$ whose images under the Jordan projection are at uniformly bounded distance from any vector in $\{0\} \times \aaa_0^+$.
This implies that $\Gamma$ satisfies property~$L_1$, \ie \eqref{item:Ci-1} holds, and we conclude as above.
\end{proof}

We will also need the following result:

\begin{proposition} \label{prop - reductive cyclic centralizer}
Let $\Gamma$ be a discrete subgroup of $G_0\times G_0$ whose action on\linebreak $(G_0\times G_0)/\Diag(G_0)$ is sharp.
Suppose that there exists a central element of~$\Gamma$ of the form $\gamma = (1_{G_0},\gamma_0)$, where $\gamma_0\in G_0$ is semisimple and generates an infinite discrete cyclic subgroup $\langle\gamma_0\rangle$ of~$G_0$.
Then the action of $\Gamma$ on $(G_0\times G_0)/\Diag(G_0)$ is not cocompact.
\end{proposition}

\begin{proof}
Let $Z_0 \equaldef Z_{G_0}(\gamma_0)$ be the centralizer of $\gamma_0$ in~$G_0$: it is a reductive subgroup of~$G_0$, of real rank~$2$, contained in a Levi factor $L_0$ of a proper parabolic subgroup of~$G_0$.
The center of~$Z_0$ contains the infinite discrete subgroup $\langle\gamma_0\rangle$, hence it contains a one-parameter subgroup $A'_0$ containing $\langle\gamma_0\rangle$.
By construction, $\langle\gamma_0\rangle$ is a uniform lattice in~$A'_0$, and the centralizer of $A'_0$ in~$G_0$ is exactly~$Z_0$.
The group $Z_0/A'_0$ is reductive of real rank~$1$.

We can see $G_0/A'_0$ as the homogeneous space $G'/H'$, where $G' = G_0 \times (Z_0/A'_0)$ and $H'$ is the subgroup of~$G'$ which is the image of the diagonal embedding of~$Z_0$.
This is a homogeneous space of reductive type with $\Rrank(G') - \Rrank(H') = 3-2 = 1$.

The group $\Gamma' \equaldef \Gamma/\langle \gamma \rangle \subset G'$ acts properly discontinuously and cocompactly on $G'/H'$.
By Proposition~\ref{prop:cork-1-proper-compact->Ano}, the group $\Gamma'$ is word hyperbolic and the natural inclusion $\Gamma' \hookrightarrow G' = G_0 \times (Z_0/A'_0)$ is Anosov with respect to some proper parabolic subgroup of~$G'$, which means that either the first projection of $\Gamma'$ to~$G_0$ or the second projection of $\Gamma'$ to $Z_0/A'_0$ is Anosov with respect to some proper parabolic subgroup.

Suppose the first projection of $\Gamma'$ to~$G_0$ is Anosov with respect to some proper parabolic subgroup $P_0$ of~$G_0$.
Then $\partial_\infty \Gamma'$ embeds into the flag manifold $G_0/P_0$.
Since the restricted root system of~$G_0$ is of type~$A_2$, we have $\Delta = \{\alpha_1,\alpha_2\}$ where $\alpha_2 = \iota(\alpha_1)$, and so we may assume that $P_0 = P_{\{\alpha_1\}}$ by Remark~\ref{rem:Anosov-theta-iota}.
This bounds the covering dimension of $\partial_\infty \Gamma'$ by $\dim(G_0/P_{\{\alpha_1\}}) = \dim((\g_0)_{\alpha_2}) + \dim((\g_0)_{\alpha_1+\alpha_2})$.
Using Fact~\ref{fact:dim-bound-vcd}, we obtain
\[\cohdim(\Gamma) = \cohdim(\Gamma') + 1 = \dim \partial_\infty \Gamma' + 2 \leq \dim (\g_0)_{\alpha_2} + \dim (\g_0)_{\alpha_1+\alpha_2} + 2. \]

Now suppose the second projection of $\Gamma'$ to $Z_0/A'_0$ is Anosov; in particular, it has finite kernel and discrete image.
The group $\Gamma'$ acts properly discontinuously on the Riemannian symmetric space of $Z_0/A'_0$, and so Fact~\ref{fact:vcd} implies that $\cohdim(\Gamma')$ is bounded by the dimension of this symmetric space, which is the dimension of the symmetric space of $Z_0$ minus one.
But we saw above that $Z_0$ is contained in a Levi subgroup $L_0$ of a proper parabolic subgroup of~$G_0$, hence the dimension of the symmetric space of~$Z_0$ is bounded by that of~$L_0$, hence by $1 + \dim (\g_0)_{\alpha_1}$.
Using Fact~\ref{fact:dim-bound-vcd}, we obtain
\[\cohdim(\Gamma) = \cohdim(\Gamma') + 1 \leq \dim (\g_0)_{\alpha_1}. \]

In either case, we see that $\cohdim(\Gamma) < \dim (\g_0)_{\alpha_1} + \dim (\g_0)_{\alpha_2} + \dim (\g_0)_{\alpha_1+\alpha_2} + 2$, where the right-hand side is the dimension of the Riemannian symmetric space of~$G_0$.
By Fact~\ref{fact:vcd}, this means that the action of $\Gamma$ on $(G_0\times G_0)/\Diag(G_0)$ cannot be cocompact.
\end{proof}

\begin{lemma} \label{lem:ProjectionIrredZariskiClosure}
Let $\Gamma$ be a discrete subgroup of $G_0\times G_0$.
Assume that the Zariski closure $L_0$ of $p_2(\Gamma)$ in~$G_0$ is reductive with compact center. Then $\Gamma$ satisfies property $D_1$ or $L_1$.
\end{lemma}

\begin{proof}
Let $\aaa_L^+$ be a Weyl chamber in a Cartan subspace of the Lie algebra of~$L_0$, and let $\mu_L$, $\lambda_L$ and $\iota_L$ denote the Cartan projection, Jordan projection and opposition involution of $L_0$, respectively.

Assume $\Gamma$ does not satisfy property~$D_1$.
Then for every $\varepsilon>0$, there exists $\gamma \in \Gamma$ such that
\[\Vert \mu(p_1(\gamma)) \Vert \leq \varepsilon \textrm{ and }\Vert \mu_L(p_2(\gamma)) \Vert \geq \frac{1}{\varepsilon}~.\]
Since $p_2(\Gamma)$ is Zariski-dense in~$L$, we can apply Fact~\ref{fact:AMS} and find $f, f'$ in (a finite subset $F$ of) $\Gamma$ and a constant $C$ (independent of $\varepsilon$) such that
$$\Vert \mu_L(p_2(\gamma f \gamma^{-1})) - \mu_L(p_2((\gamma)) - \iota_L (\mu_L(p_2(\gamma))) \Vert \leq C$$
and
$$\Vert \lambda_L(p_2(\gamma f \gamma^{-1} f')) - \mu_L(p_2(\gamma f \gamma^{-1})) \Vert \leq C.$$
By the triangle inequality, we obtain
$$\big\Vert \lambda_L (p_2(\gamma f \gamma^{-1} f')) - \big(\mu_L(p_2(\gamma)) + \iota_L (\mu_L(p_2(\gamma)))\big) \big\Vert \leq 2C .$$
Since $L$ has compact center, the cone $\aaa_L^+$ is acute and $\mu_L(p_2(\gamma)) + \iota_L \mu_L(p_2(\gamma))$ is a nonzero vector fixed by the opposition involution, with $\Vert \mu_L(p_2(\gamma)) + \iota_L \mu_L(p_2(\gamma)) \Vert \geq \Vert \mu_L(p_2(\gamma)) \Vert \geq 1/\varepsilon$.
Then $\gamma f \gamma^{-1} f'$ satisfies the conditions of property $L_1$ (up to replacing $\varepsilon$ by $4C\varepsilon$, for small enough~$\varepsilon$).
\end{proof}

We finally turn to the proof of Lemma~\ref{lem:property-Di-or-Ci}, which will complete the proof of Theorem~\ref{thm:group-mfd-SL3}.

\begin{proof}[Proof of Lemma~\ref{lem:property-Di-or-Ci}]
We prove that $\Gamma$ satisfies property $D_1$ or~$C_1$; the proof for $D_2$ or~$C_2$ is analogous.
It is sufficient to assume that $\Gamma$ does \emph{not} satisfy property~$D_1$, and to prove that it then satisfies property~$C_1$.
Note that if some subgroup of~$\Gamma$ satisfies property~$C_1$, then so does~$\Gamma$.
Therefore, up to replacing $\Gamma$ by a finite-index subgroup, we may assume that it is torsion-free (by the Selberg lemma \cite[Lem.\,8]{sel60}).

Let us write the Zariski closure of $\Gamma$ as $L\ltimes U$, where $U$ is its unipotent radical and $L$ is reductive.
Then the commutator group $[\Gamma,\Gamma]$ is Zariski-dense in $[L,L] \ltimes U'$, where $U'$ is a subgroup of $U$ and $[L,L]$ is semisimple.
Let $\pi: L \ltimes U \to L$ denote the natural projection.
By Fact~\ref{fact:discrete-proj}, the group $\pi([\Gamma,\Gamma])$ is discrete.

If $\pi([\Gamma,\Gamma])$ satisfies property $L_1$, then so does $[\Gamma,\Gamma]$ (because $\pi$ preserves the Jordan projection), and so does $\Gamma$ (the property is obviously stable by enlarging the group).
Hence $\Gamma$ satisfies property~$C_1$ by Proposition~\ref{prop:criteria-property-Ci}.\eqref{item:Ci-1} and we are done.

We now assume that $\pi([\Gamma,\Gamma])$ does \emph{not} satisfy property~$L_1$.
By Lemma~\ref{lem:ProjectionIrredZariskiClosure}, this implies that $\pi([\Gamma, \Gamma])$ satisfies property~$D_1$: the restriction of $p_1$ to $\pi([\Gamma,\Gamma])$ is discrete and faithful (recall that we have assumed $\Gamma$ to be torsion-free).
By Lemma~\ref{lem:unipotent-radical-one-factor}, the kernel of $\pi_{\vert [\Gamma,\Gamma]}$ is contained in $G_0\times \{\1_{G_0}\}$ or in $\{\1_{G_0}\} \times G_0$.
In particular, its first projection is discrete.
If its second projection is nontrivial, then the kernel of $p_1$ contains a nontrivial unipotent element, $\Gamma$ satisfies property $C_1$ by Proposition~\ref{prop:criteria-property-Ci}.\eqref{item:Ci-3}, and we are done.
Otherwise, $p_1$ restricted to $[\Gamma,\Gamma]$ is discrete and faithful, which we assume from now on.

Recall that we have assumed that $\Gamma$ does not satisfy property~$D_1$.
Therefore, $p_1(\Gamma)$ is not discrete or the restriction of $p_1$ to~$\Gamma$ has discrete image but infinite kernel.\\

\noindent {\it First case: the group $p_1(\Gamma)$ is not discrete.}
Let $A_0'$ be the identity component of the closure of $p_1(\Gamma)$ in~$G_0$ (for the real topology).
Fixing some $\gamma \in \Gamma$, for every $\eta\in \Gamma$ such that $p_1(\eta)$ is sufficiently close to $\1_{G_0}$, we have that $p_1([\gamma, \eta])\in p_1([\Gamma,\Gamma])$ is close to $\1_{G_0}$.
Since $p_1([\Gamma,\Gamma])$ is discrete, we get that $\gamma$ commutes with every such $\eta$ and deduce that $A_0'$ is centralized by $p_1(\Gamma)$.
The fact that $p_1(\Gamma)$ centralizes a $1$-parameter subgroup implies that $p_1(\Gamma)$ is either contained in a minimal parabolic subgroup or in a Levi factor of a proper parabolic subgroup of~$G_0$.
Hence $p_2(\Gamma)$ cannot be contained in a proper parabolic subgroup of~$G_0$ (otherwise $\Gamma$ would be conjugate to a subgroup of $P_{\{\alpha_1\}}\times P_{\{\alpha_2\}}$ or $P_{\{\alpha_2\}} \times P_{\{\alpha_1\}}$, contradicting Theorem~\ref{thm:PxPopp-structures}).
Similarly, no finite-index subgroup of $p_2(\Gamma)$ can be contained in a proper parabolic subgroup of $G_0$, hence the Zariski closure of $p_2(\Gamma)$ is reductive with compact center and we conclude that $\Gamma$ satisfies property~$C_1$ by Lemma~\ref{lem:ProjectionIrredZariskiClosure}.\\

\noindent {\it Second case: the restriction of $p_1$ to~$\Gamma$ has discrete image and infinite kernel~$N_1$.}
Since $N_1$ does not intersect $[\Gamma,\Gamma]$, it is central in~$\Gamma$, and in particular abelian.
Let us prove that $N_1$ is not virtually cyclic or that the natural projection $N_1\to (N_1)_{\red}$ is discrete and faithful; then Proposition~\ref{prop:criteria-property-Ci}.\eqref{item:Ci-3}--\eqref{item:Ci-4} will imply that $\Gamma$ satisfies property~$C_1$.

Suppose by contradiction that $N_1$ is virtually infinite cyclic and that the natural projection $N_1\to (N_1)_{\red}$ is discrete and faithful.
We can write $N_1 = \{ \1_{G_0}\} \times \Lambda$ where $\Lambda$ is virtually the cyclic group generated by an element $\gamma = (\1_{G_0},\gamma_0)$.
Let us write the Zariski closure $\overline{\Lambda}^Z$ of~$\Lambda$ as $A_0' \times U_0'$, where $U_0'$ is its unipotent radical and $A_0'$ is reductive.
Note that $\{\1_{G_0} \}\times U'_0$ is a central (hence normal) unipotent subgroup of $L\ltimes U$, hence $\{\1_{G_0} \}\times U'_0 \subset U$. Up to conjugating the Levi factor $L$ of $\overline{\Lambda}^Z$, we can assume that $\{\1_{G_0} \} \times A_0'\subset L$.
Then $\pi_{\vert \langle \gamma \rangle}$ is the projection to the $\{\1_{G_0} \} \times A_0'$ factor, which is discrete by our assumption on~$\gamma$.

Since $p_2(\Gamma)$ centralizes an infinite cyclic subgroup, arguing as in the first case above, we get that $p_1(\Gamma)$ cannot be contained in a proper parabolic subgroup of~$G_0$ (otherwise we would derive a contradiction from Theorem~\ref{thm:PxPopp-structures}).
Thus $p_1(U) = \{\1_{G_0}\}$ (for otherwise $p_1(\Gamma)$ would normalize a nontrivial unipotent subgroup and be contained in a proper parabolic subgroup of~$G_0$), and $p_1 \circ \pi(\Gamma) = p_1(\Gamma)$ is discrete.
Hence the preimage under $\pi$ of a small neighborhood of the identity is contained in the kernel of~$p_1$.
Since $\pi_{\vert \langle \gamma \rangle}$ is discrete and faithful, we conclude that $\pi_{\vert \Gamma}$ is discrete and faithful.

Now, $\pi(\Gamma)$ must still act properly discontinuously and cocompactly on $G_0$ by Lemma~\ref{lem:reductive-projection-proper} while, on the other hand, $\pi(\Gamma)$ centralizes the cyclic subgroup $\langle \pi(\gamma)\rangle$ generated by a semisimple element.
This contradicts Proposition~\ref{prop - reductive cyclic centralizer}.
\end{proof}

%%%%%%%%%%%%%%%%%%%%%%%%%%%%%%%%%%%%%%%%%%%%%%%%%%%
%%%%%%%%%%%%%%%%%%%%%%%%%%%%%%%%%%%%%%%%%%%%%%%%%%%

\end{document}